\documentclass[hidelinks,onefignum,onetabnum]{siamart220329}
\usepackage{amsfonts,amssymb}
\usepackage{graphicx}
\usepackage{epstopdf}
\usepackage{tikz}
\usepackage{pgfplots}
\usepgfplotslibrary{fillbetween}
\pgfplotsset{compat=newest,
every axis/.append style={axis x line=bottom,
                          axis y line=left,
                          scale only axis,
                          y label style={at={(0.0,1.0)},anchor=south west,rotate=-90}
                          },
}
\usepackage{algorithm,algpseudocode}
\ifpdf
  \DeclareGraphicsExtensions{.eps,.pdf,.png,.jpg}
\else
  \DeclareGraphicsExtensions{.eps}
\fi

\ifpdf
\hypersetup{
  pdftitle={Tensor approximation for almost sure constraints},
  pdfauthor={H. Antil, S. Dolgov and A. Onwunta}
}
\fi

\newtheorem{example}[theorem]{Example}
\newtheorem{assumption}[theorem]{Assumption}
\newtheorem{remark}[theorem]{Remark}

\date{4 July 2024}

\headers{Tensor approximation for almost sure constraints}{H. Antil, S. Dolgov and A. Onwunta}

\title{Smoothed Moreau-Yosida Tensor Train Approximation of State-constrained Optimization Problems under Uncertainty\thanks{HA is partially supported by NSF grant DMS-2110263, the AirForce Office of Scientific Research under Award NO: FA9550-22-1-0248 and the Office of Naval Research (ONR) under Award NO: N00014-24-1-2147.
SD is thankful for the support from Engineering and Physical Sciences Research
Council (EPSRC) New Investigator Award EP/T031255/1 and New Horizons grant EP/V04771X/1.
}}

\author{Harbir Antil\thanks{The Center for Mathematics and Artificial Intelligence
(CMAI) and Department of Mathematical Sciences, George Mason University,
Fairfax, VA 22030, USA. \email{hantil@gmu.edu}}
\and
Sergey Dolgov\thanks{Department of Mathematical Sciences, University of Bath, Bath, BA2 7AY, UK. \email{s.dolgov@bath.ac.uk}}
\and
Akwum Onwunta\thanks{Department of Industrial and Systems Engineering, Lehigh University, Bethlehem, PA 18015, USA. \email{ako221@lehigh.edu}}
}

\begin{document}
\maketitle

\begin{abstract}
We propose an algorithm to solve optimization problems constrained by partial (ordinary) differential equations
under uncertainty, with almost sure constraints on the state variable. To alleviate the computational burden of 
high-dimensional random variables, we approximate all random fields by the tensor-train decomposition.
To enable efficient tensor-train approximation of the state constraints, the latter 
are handled using the Moreau-Yosida penalty, with an additional smoothing of
the positive part (plus/ReLU) function by a softplus function.
In a special case of a quadratic cost minimization constrained by linear elliptic partial differential equations, and some additional constraint qualification,
we prove strong convergence
of the regularized solution to the optimal control.
This result also proposes a practical recipe for selecting the smoothing parameter as a function of the penalty parameter.
We develop a second order Newton type method
with a fast matrix-free action of the approximate Hessian to solve the smoothed Moreau-Yosida problem.
This algorithm is tested on benchmark elliptic problems with random coefficients, optimization problems constrained
by random elliptic variational inequalities, and a real-world epidemiological model with
20 random variables.
These examples demonstrate mild (at most polynomial) scaling with respect to the dimension and regularization 
parameters.
\end{abstract}



\section{Introduction}
\subsection{Informal statement of results}
We consider risk-neutral optimization problems constrained by both a system of random differential equations (PDE/ODE) $c(y,u; \omega) = 0$ on a random field state variable $y(x,\omega)$ with a deterministic control function $u(x)$, and box state constraints $\psi(x,\omega)$,
\begin{align}\label{eq:problema-informal}
 \min_{u \in \mathcal{U}_{ad}} \; & \mathbb{E}_{\mathbb{P}}\left[J(y,u)\right] \\
 \text{such that} \; & c(y(x,\omega), u(x); \omega) = 0 & \text{almost surely}, \\
 \text{and} \; & y(x,\omega) \ge \psi(x,\omega) & \text{almost surely,}
 \label{eq:problemc-informal}
\end{align}
where $\mathcal{U}_{ad}$ is a closed convex set, and $J$ is an appropriate cost function.
While the equality constraint can in many cases be resolved, i.e. there exists a map $y = S(u)$ such that $c(S(u), u; \omega) = 0$ almost surely (a.s.),
the box state constraints are more difficult to deal with.
We consider a smoothed Moreau-Yosida unconstrained optimization problem
\begin{align}\label{eq:my-informal}
\min_{u \in \mathcal{U}_{ad}} \; & \mathbb{E}_{\mathbb{P}}\left[J(S(u),u) + \frac{\gamma}{2} \|g_{\varepsilon}(S(u) - \psi)\|^2_{L^2(D)}\right]
\end{align}
for some penalty parameter $\gamma>0$ and smoothing parameter $\varepsilon > 0$,
and a convex infinitely differentiable function $g_{\varepsilon}: \mathbb{R} \rightarrow \mathbb{R}_+$.
In particular, we will use the softplus function $g_{\varepsilon}(s) = \varepsilon \cdot \log(1 + \exp(-s/\varepsilon))$.
\begin{theorem}[Informal statement of the main theoretical result, Thm.~\ref{thm:conv}]
For linear elliptic PDE constraint, quadratic cost function, and suitable state constraint qualification, the solution $u^{\gamma}$ to \eqref{eq:my-informal}
converges to the solution $u$ to \eqref{eq:problema-informal}--\eqref{eq:problemc-informal} strongly in $L^2(D)$ as $\gamma \rightarrow \infty$, provided that  $\varepsilon = o(\gamma^{-1/2})$.
\end{theorem}
The smoothness of the entire objective function \eqref{eq:my-informal} (which is mainly due to $g_{\varepsilon}$ with $\varepsilon>0$, in contrast to a non-differentiable $g_0$) is necessary for second-order optimization methods with faster convergence in the number of iterations, and functional tensor-train approximation of random fields, with faster convergence in the number of degrees of freedom.
\begin{proposition}[Informal statement of the main practical results]
\begin{itemize}
 \item A Gauss-Newton method for fast optimization of \eqref{eq:my-informal} with a matrix-free action of the approximate Hessian.
 \item A functional tensor-train approximation of finite-dimensional parametrization of random fields with spectral discretization in random parameters (and hence exponential convergence in the number of degrees of freedom), and only a polynomial complexity scaling in the dimension.
 \item Verification of the method on elliptic equations with random coefficients, as well as an ODE example (motivated by a realistic application) with 20 random variables, and an elliptic variational inequality.
 \item In practice, the convergence is achieved also with $\varepsilon = \mathcal{O}(\gamma^{-1/2})$.
\end{itemize}
\end{proposition}

\subsection{The context and existing literature}
Over last two decades optimization problems constrained by physical laws, such as partial (ordinary) 
differential equations (PDEs/ODEs), have emerged as a prominent research area. This is fueled by 
many applications in science and engineering, such as controlling pathogen propagation in 
built environment \cite{RLoehner_HAntil_ASrinivasan_SIdelsohn_EOnate_2021a,
RLoehner_HAntil_SIdelsohn_EOnate_2020a}, digital twins \cite{HAntil_2024a}, shape and topology 
optimization \cite{JSokolowski_JPZolesio_1992a,KMaute_2014a}, optimal strategies to predict 
shutdowns due to pandemics \cite{DGKP-SEIR-2021}. The optimization variables 
consist of state $(y)$ and control/design $(u)$. However, often due to noisy
measurements and ambiguous models due to incomplete  physics, the underlying physical laws 
contain uncertainty. This has led to significant theoretical and algorithmic developments in the 
area of optimization problems constrained by physical laws under uncertainty. See for instance 
\cite{KS16,HAntil_DPKouri_MDLacasse_DRidzal_2018a,GSU21,HAntil_SDolgov_AOnwunta_2022b,Kouri-RB-Risk-2022}
and the references therein for problems with control constraints.

The literature on state-constrained optimization problems under uncertainty is scarce. For
instance, \cite{MFShaker_RHenrion_DHoemberg_2018a,AGeletu_AHoffmann_PSchmidt_PLi_2020a}
 use probability constraints, and \cite{CGeiersbach_WWollner_2021a,
Surowiec-risk-2022,CGeiersbach_MHintermueller_2022a}  consider almost surely type
constraints. It is well-known that even in the deterministic setting, the state 
constrained problems are highly challenging. One of the fundamental difficulties is that
the state constraints are imposed in the sense of continuous functions. As a result, the
Lagrange multipliers corresponding to those constraints are Radon measures that exhibit 
low regularity \cite{ECasas_1986a}. The situation is much 
more delicate in the stochastic setting. 
We refer to the aforementioned references for a detailed discussion on this topic. 
Motivated by the deterministic setting, \cite{Surowiec-risk-2022} introduces
a Moreau-Yosida based approximation scheme to solve the state-constrained optimization
problems when the PDE constraints are given by an elliptic equation with random
coefficients. Further extensions of this work are considered in 
\cite{Surowiec-vi-2022,CGeiersbach_MHintermueller_2022a}.
However, all of these papers approximate expectations of random fields by Monte-Carlo-type 
methods, which may converge slowly.

In \cite{HAntil_SDolgov_AOnwunta_2022b}, we introduced an algorithm (TTRISK) based on the tensor train (TT) decomposition \cite{osel-tt-2011}
to solve risk-averse optimization problems with control constraints. 
We demonstrated that the extra computational cost due to the uncertainty can scale proportionally to $\mathrm{error}^{-0.5}$ when the TT approximation is used, in contrast to a $\mathrm{error}^{-2}$ scaling of Monte Carlo quadratures.
In the current paper, we continue this program and develop a TT based algorithm
for state-constrained optimization problems.

\noindent
{\bf Outline:} The remainder of the paper is organized as follows. In Section~\ref{s:prob},
we provide a rigorous mathematical formulation of the problem under consideration. 
Section~\ref{s:MY} is devoted to the main contributions: the Moreau-Yosida approximation, its convergence, and the second
order Newton-type method for practical computations.
Section~\ref{sec:prac} introduces the practical discretisation of the Moreau-Yosida problem, the function approximation (the Tensor-Train decomposition), and the pseudocode of the Newton-type solver for the Moreau-Yosida problem.
Finally, in Section~\ref{s:nex}, we provide a series of
numerical experiments. At first, we consider an optimization problem with an elliptic 
PDE in one spatial dimension as constraints. This is followed by a two-dimensional 
case. After these benchmarks, an optimization problem with an elliptic
variational inequality as constraint is considered in Section~\ref{s:VI}.
The numerical experiments conclude with a realistic ODE example for 
designing optimal lockdown strategies in 
Section~\ref{s:SEIR}.

\section{Problem Formulation}
\label{s:prob}
\subsection{Notations, function and probability spaces}
The setup here is directly motivated by \cite{Surowiec-risk-2022}.
Let the physical domain $D \subset \mathbb{R}^d$ be an open bounded set such that either $D$ is a convex polyhedron, or the boundary $\partial D$ is $C^{1,1}$.
Let $(\Omega, \mathcal{F}, \mathbb{P})$ denote a complete probability space, where $\Omega$ represents the sample space, $\mathcal{F}$ is the Borel $\sigma$-algebra of events on the power set of $\Omega$, and $\mathbb{P} : \Omega \rightarrow [0, 1]$ is an appropriate probability measure.
For a real Banach space $V$ the expectation of a random element $X: \Omega \rightarrow V$ reads
$$
\mathbb{E}_{\mathbb{P}}[X] = \int_{\Omega} X(\omega) d\mathbb{P}(\omega) \in V.
$$
For an arbitrary convex set $K$, we define the standard convex normal cone,
$$
\mathcal{N}_K(x) = \left\{\begin{array}{ll}\{x^* \in V^* \; | \; \langle x^*, y - x \rangle \le 0, \quad \forall y \in K\}, & \text{if } x \in K, \\ \emptyset, & \text{otherwise.}\end{array}\right.
$$
Strong convergence of a sequence is denoted by $\rightarrow$, weak convergence by $\rightharpoonup$, and weak* convergence by $\rightharpoonup^*$.
A closed $\delta$-ball centered at $x$ in a normed space is denoted by $\mathbb{B}_{\delta}(x)$.
For two Banach spaces $V$ and $W$, the set of bounded linear operators from $V$ to $W$ is denoted by $\mathcal{L}(V,W)$.

For the multipliers of the state constraints we need the following space.
Firstly, let
$$
\Xi = \Omega \times D, \quad \mathcal{B} = \mathcal{F} \times \mathcal{F}_{\mathcal{L}}, \quad \pi = \mathbb{P} \times \mathcal{L},
$$
where $\mathcal{L}$ is the Lebesgue measure on $V$, and $\mathcal{F}_{\mathcal{L}}$ is the $\sigma$-algebra of Lebesgue-measurable sets.
Now let $\mathbf{ba}(\Xi, \mathcal{B}, \pi)$ be the space of real-valued set-functions $\tau: \mathcal{B} \rightarrow \mathbb{R}$ such that
\begin{itemize}
 \item $\sup \{|\tau(A)| \; : \; A \in \mathcal{B}\} < \infty,$
 \item $\tau(A \cup B) = \tau(A) + \tau(B)$ for $A,B\in\mathcal{B}$ with $A \cap B = \emptyset$,
 \item $\tau(A) = 0$ if $A \in \mathcal{B}$ is $\pi$-null, i.e. $\tau \ll \pi$. 
\end{itemize}
It was shown in \cite[Thm.~IV.8.16]{Dunford-ops-1988} that $\mathbf{ba}(\Xi, \mathcal{B}, \pi)$ is isomorphic to $L^{\infty}_{\pi}(\Xi)^*$.
We also recall that the dual of continuous functions $C(\overline\Xi)$ is identified with signed measures
$\mathcal{M}(\overline{\Xi})$.

\subsection{Linear elliptic random PDEs as equality constraint}\label{linearrandpde}
For simplicity, we derive a rigorous convergence proof in the following scenario of the equality constraint $c(y,u; \omega)=0$.
Let $\mathcal{Y}:=L^2_{\mathbb{P}}(\Omega; H^1_0(D))$, $u \in \mathcal{U}_{ad} \subset L^2(D)$, and consider the following PDE constraint:
\begin{align}\label{eq:elliptic}
\mathbb{E}_{\mathbb{P}}\left[\int_D A(x,\cdot) \nabla y(x,\cdot) \cdot \nabla v(x,\cdot) dx\right] & = \mathbb{E}_{\mathbb{P}}\left[\int_D \left((B(\cdot) u)(x) + f(x,\cdot)\right) v(x,\cdot) dx\right],
\end{align}
for all test functions $v \in \mathcal{Y}$.
\begin{assumption}\label{as:elliptic}
For a well-behaved PDE solution operator we need the following assumptions.
\begin{itemize}
 \item(Minimum regularity) The coefficient $A: D \times \Omega \rightarrow \mathbb{R}$ is $(\mathcal{L} \times \mathbb{P})$-measurable and there exist $\overline{A} > \underline{A} > 0$ such that
 $$
 \underline{A} \le A(x,\omega) \le \overline{A} \qquad (\mathcal{L} \times \mathbb{P})\text{-a.e.~} (x,\omega) \in D \times \Omega,
 $$
 and 
 $$
 f \in L^{\infty}_{\mathbb{P}}(\Omega; L^2(D)).
 $$
 \item (Higher regularity) $A \in L^{\infty}_{\mathbb{P}}(\Omega; C^{0,1}(\overline{D}))$.
 \item (Control mapping) $B: \Omega \rightarrow \mathcal{L}(L^2(D), L^2(D))$ is measurable and essentially bounded, 
 \linebreak $B \in L^{\infty}_{\mathbb{P}}(\Omega; \mathcal{L}(L^2(D), L^2(D))).$
 Moreover, $B: \Omega \rightarrow \mathcal{L}(L^2(D), H^{-1}(D))$ is completely continuous,
 $$
 u_k \rightharpoonup u \quad \text{in~} L^2(D) \quad \Rightarrow \quad B(\omega) u_k \rightarrow B(\omega) u \quad \text{in~}H^{-1}(D) \quad \mathbb{P}\text{-a.e.~}\omega\in\Omega.
 $$
\end{itemize}
\end{assumption}
Under this assumption, \cite[Corollary~1]{Surowiec-risk-2022} establishes that any solution $y$ of \eqref{eq:elliptic} can be written as
\begin{align}
 y & = S u + y_f,
\end{align}
where
\begin{itemize}
 \item $S: L^2(D) \rightarrow L^q_{\mathbb{P}}(\Omega; H^1_0(D))$ is completely continuous, bounded and linear for any $q \in [1, \infty)$.
 \item $S: L^2(D) \rightarrow L^{\infty}_{\mathbb{P}}(\Omega; H^1_0(D) \cap H^2(D))$ is bounded and linear, and 
 \item $y_f \in L^{\infty}_{\mathbb{P}}(\Omega; H^1_0(D) \cap H^2(D))$.
\end{itemize}
Moreover, introducing operators
$$
\mathbf{A}: \Omega \rightarrow \mathcal{L}(H_0^1(D) \cap H^2(D), L^2(D)) \quad \text{and} \quad \mathbf{B}: \Omega \rightarrow \mathcal{L}(L^2(D))
$$
defined by
$$
\langle \mathbf{A}(\omega) \phi, \chi \rangle = \int_D A(x,\omega) \nabla \phi(x) \cdot \nabla \chi(x) dx
$$
for $\phi,\chi \in H^1_0(D) \cap H^2(D)$, and 
$$
\langle \mathbf{B}(\omega) u, \chi \rangle = \int_D ((B(\omega) u)(x) \chi(x) dx,
$$
respectively,
we can define $Su + y_f$ $\mathbb{P}$-pointwise as
$$
Su(\omega) = \mathbf{A}^{-1}(\omega) \mathbf{B}(\omega) u, \qquad y_f(\omega) = \mathbf{A}^{-1}(\omega) f(\omega).
$$

\subsection{Risk-neutral PDE constrained optimization of a quadratic cost}
\begin{assumption}\label{as:cost}
For the optimization problem we need the following assumptions.
\begin{itemize}
 \item (Control constraints) $U_{ad} \subset L^2(D)$ is a nonempty, closed, bounded and convex set.
 \item (Cost) We start with a deterministic cost function
\begin{align}
 J(y,u) & = \frac{1}{2}\|Ty - y_d\|_{L^2(D)}^2 + \frac{\alpha}{2} \|u\|_{L^2(D)}^2,
\end{align}
where $\alpha \ge 0$ is a regularization parameter, $y_d \in L^2(D)$, $T \in \mathcal{L}(L^2(D))$.
\item (State Constraint) Given $\psi \in C(\overline{\Omega \times D})$ for which there exists $\delta>0$ such that $\psi_{\partial D(\omega)} \le -\delta$ $\mathbb{P}$-a.s, the state constraint is defined as
\begin{align}\label{eq:state_constraint}
 S u + y_f &\ge \psi & (\mathcal{L} \times \mathbb{P})\text{-a.e.}
\end{align}
\item (Feasibility) There exists $u \in \mathcal{U}_{ad}$ such that \eqref{eq:state_constraint} holds.
\end{itemize}
\end{assumption}
Now we are to solve the following constrained optimization problem:
\begin{align}\label{eq:problema}
\min~ & \mathbb{E}_{\mathbb{P}}\left[J(y(\omega), u)\right] & \text{over } (u,y) \in \mathcal{U}_{ad} \times \mathcal{Y} \\
\text{s.t.}~ & \mathbf{A}(\omega) y = \mathbf{B}(\omega) u + f(\omega) & \mathbb{P}\text{-a.s.} \\
& y \ge \psi & (\mathcal{L} \times \mathbb{P})\text{-a.e.} \label{eq:problemc}
\end{align}

\begin{remark}
Under Assumptions~\ref{as:elliptic} and \ref{as:cost}, \cite[Theorem~4]{Surowiec-risk-2022} shows that a solution $u^*$ to \eqref{eq:problema}--\eqref{eq:problemc} exists, and is unique if $\alpha>0$.
Moreover, if the Slater condition (Assumption~\ref{ass:slater} below) holds, there also exists a Lagrange multiplier variable in a form of a measure $\mu^* \in \mathbf{ba}(\Xi, \mathcal{B}, \pi)$.
\end{remark}

\section{Smoothed Moreau-Yosida approximation}
\label{s:MY}
Solving \eqref{eq:problema}--\eqref{eq:problemc}
involves computation of the indicator function of an active set and/or Lagrange multiplier
as a random field that is nonnegative on a complicated high-dimensional domain.
This may be difficult for many function approximation methods, e.g. the tensor decompositions considered in this paper.
We tackle this difficulty by first turning the constrained optimization problem
\eqref{eq:problema}--\eqref{eq:problemc} into an unconstrained optimization problem with the Moreau-Yosida
penalty, and further by smoothing the indicator function in the penalty term.

The classical Moreau-Yosida problem reads, with $\gamma \ge 0$ denoting the penalty parameter,
\begin{align*}
\min_{u^{\gamma} \in \mathcal{U}_{ad}}~ & \mathbb{E}_{\mathbb{P}}\left[\frac{1}{2}\|T(Su^{\gamma} + y_f) - y_d\|_{L^2(D)}^2\right] + \frac{\alpha}{2} \|u^{\gamma}\|_{L^2(D)}^2 + \mathbb{E}_{\mathbb{P}}\left[\frac{\gamma}{2}\|(\psi - Su^{\gamma} - y_f)_+\|_{L^2(D)}^2 \right],
\end{align*}
where the so-called \emph{positive part} or \emph{ReLU} function $(\cdot)_+$ is defined as $(s)_+ = s$ if $s\ge 0$ and $0,$ otherwise.
Here, we have removed the need to optimize the Lagrange multiplier (corresponding to the inequality constraints) over the nonnegative cone, but the function approximation of a nonsmooth high-dimensional random field $(\psi - Su^{\gamma} - y_f)_+$ (and derivatives thereof) may still be inefficient.

For this reason, we replace the \emph{ReLU} function in the penalty term by a smoothed version. In this paper, we use the \emph{softplus} function (with a minus sign incorporated)
\begin{equation}\label{eq:softplus}
g_{\varepsilon}(s) = \varepsilon\cdot \log(1+\exp(-s/\varepsilon)) \in C^{\infty}(\mathbb{R}), \qquad g_0(s) = \lim_{\varepsilon\rightarrow 0} g_{\varepsilon}(x) = (-s)_+,
\end{equation}
although other (e.g. piecewise polynomial) functions are also possible \cite{KKunisch_DWachsmuth_2012a,Surowiec-vi-2022}.

The reduced-space Moreau-Yosida optimization problem reads
\begin{align}\label{eq:problemmy}
\min_{u^{\gamma} \in \mathcal{U}_{ad}}~ & \mathbb{E}_{\mathbb{P}}\left[\frac{1}{2}\|T(Su^{\gamma} + y_f) - y_d\|_{L^2(D)}^2\right] + \frac{\alpha}{2} \|u^{\gamma}\|_{L^2(D)}^2 + \mathbb{E}_{\mathbb{P}}\left[\frac{\gamma}{2}\|g_{\varepsilon_{\gamma}}(Su^{\gamma} + y_f - \psi)\|_{L^2(D)}^2 \right],
\end{align}
where we plugged in $T$ and $S$, and assume that the smoothing parameter $\varepsilon_{\gamma}$ depends on $\gamma$, so we keep only $\gamma$ in the superscript of $u^{\gamma}$.

We will need the following basic properties of the softplus function.
\begin{lemma}\label{lem:gprime}
For any $\varepsilon > 0,$ \eqref{eq:softplus} is convex and satisfies
$g_{\varepsilon}(s) \ge (-s)_+$ and $g_{\varepsilon}'(s) < 0$ for any $s\in \mathbb{R}$,
as well as $s g_{\varepsilon}(s) \le \varepsilon^2 \mathrm{e}^{-1}$ for $s \ge 0$.
\end{lemma}
\begin{proof}
Differentiating $g_{\varepsilon}(s)$ twice,
$$
g_{\varepsilon}'(s) = -\frac{1}{1+\exp(-s/\varepsilon)}, \qquad g_{\varepsilon}''(s) = \frac{1/\varepsilon}{(\exp(\frac{s}{2\varepsilon}) + \exp(-\frac{s}{2\varepsilon}))^2},
$$
we notice that $g_{\varepsilon}'(s) < 0$, and $g_{\varepsilon}''(s)>0$, so $g_{\varepsilon}$ is convex.
Using the monotonicity of the logarithm,
$$
g_{\varepsilon}(s) = \varepsilon\log\left(1+\exp(-s/\varepsilon)\right) \ge \left\{\begin{array}{ll} \varepsilon\log\left(\exp(-s/\varepsilon)\right) = -s = (-s)_+, & s\le 0, \\ 0 = (-s)_+, & s>0, \end{array}\right.
$$
we obtain the second claim.
Lastly, since $\mathrm{e}^s \ge 1 + s$, $\log(1+s) \le s$, and $s g_{\varepsilon}(s) \le s \varepsilon \exp(-s/\varepsilon)=:\hat g_{\varepsilon}(s)$ for $s \ge 0$.
Differentiating $\hat g_{\varepsilon}(s)$ gives $\hat g_{\varepsilon}'(s_*) = \varepsilon \exp(-s_*/\varepsilon) (1 - s_* / \varepsilon) = 0$,
so the maximizer is $s_* = \varepsilon$,
and the maximum is $\hat g_{\varepsilon}(s_*) = \varepsilon^2 \exp(-1)$.
\end{proof}

\subsection{Strong convergence in the linear-quadratic case}
We prove the strong convergence as $\gamma \rightarrow \infty$ and $\varepsilon \rightarrow 0$ depending on $\gamma$ in a particular manner in the simplified case of a linear elliptic PDE as a constraint, and a quadratic cost functional.
We will also need the following assumptions.
\begin{assumption}[Slater condition]
\label{ass:slater}
There exists  $\delta>0$ such that there exists  $v \in \mathcal{U}_{ad}$ that satisfies
\begin{equation}\label{eq:slater}
S v + y_f \ge \psi + \delta \qquad \mathbb{P}\text{-a.s, a.e. }D.
\end{equation}
\end{assumption}
\begin{assumption}[Higher parametric regularity]\label{as:par-reg}
\begin{itemize}
 \item $\Omega$ is a compact Polish space.
 \item $S(\cdot) + y_f$ is a continuous affine mapping from $L^2(D)$ to $C(\Omega; H_0^1(D) \cap H^2(D))$. 
\end{itemize}
\end{assumption}

The existence and characterization of solution to \eqref{eq:problemmy} emerge from the same arguments as those in \cite[Theorems~5,7]{Surowiec-risk-2022}.
If $u^{\gamma}$ is a solution to \eqref{eq:problemmy}, then the necessary and sufficient optimality conditions read
\begin{align}\label{eq:kkt1}
0 & \in \mathbb{E}_{\mathbb{P}}\left[\mathbf{B}^* \mathbf{A}^{-*} T^* (T S u^{\gamma} + T y_f - y_d)\right] + \alpha u^{\gamma} + \mathcal{N}_{\mathcal{U}_{ad}}(u^{\gamma}) + \mathbb{E}_{\mathbb{P}}\left[\mathbf{B}^* \mathbf{A}^{-*} \mu^{\gamma}\right],
\end{align}
where
\begin{align}\label{eq:mugamma}
\mu^{\gamma} & = \gamma g_{\varepsilon_{\gamma}}'(S u^{\gamma} + y_f - \psi) g_{\varepsilon_{\gamma}}(S u^{\gamma} + y_f - \psi).
\end{align}
Introducing the adjoint state $\lambda^{\gamma}$ solving the adjoint equations 
\begin{align}
\int_D A(x,\omega) \nabla \lambda^{\gamma} \cdot \nabla \varphi(x) dx & = \int_D \left(T^*\left(T S u^{\gamma} + T y_f - y_d\right) + \mu^{\gamma}\right) \varphi(x) dx & \mathbb{P}\text{-a.s. }\forall\phi\in H^1_0(D),
\end{align}
we can simplify \eqref{eq:kkt1} to
\begin{align}\label{eq:kkt2}
0 & \in \mathbb{E}_{\mathbb{P}}\left[\mathbf{B}^* \lambda^{\gamma}\right] + \alpha u^{\gamma} + \mathcal{N}_{\mathcal{U}_{ad}}(u^{\gamma}).
\end{align}

Now we are ready to formulate the convergence theorem.
\begin{theorem}\label{thm:conv}
 Suppose the problem \eqref{eq:problema}--\eqref{eq:problemc} has a quadratic cost and linear elliptic PDE constraints satisfying Assumptions~\ref{as:cost} and \ref{as:elliptic}, and in addition the Slater~\eqref{eq:slater}, 
 and regularity (Assumption~\ref{as:par-reg}) conditions hold. 
 Then there exist sequences $\gamma_k \rightarrow \infty$ and $\varepsilon_{\gamma_k} = o(\gamma_k^{-1/2}) \rightarrow 0$ such that
 \begin{itemize}
  \item $\{u^{\gamma_k}\}_{k \in \mathbb{N}} \subset L^2(D)$,
  \item $\{y^{\gamma_k}\}_{k \in \mathbb{N}} \subset L^2_{\mathbb{P}}(\Omega; H_0^1(D) \cap H^2(D))$,
  \item $\{\lambda^{\gamma_k}\}_{k \in \mathbb{N}} \subset L^2_{\mathbb{P}}(\Omega; H_0^1(D) \cap H^2(D))$,
  \item $\{\eta^{\gamma_k}\}_{k \in \mathbb{N}} \subset L^2(D)$,
  \item $\{\mu^{\gamma_k}\}_{k \in \mathbb{N}} \subset L^2_{\pi}(\Xi)$, 
 \end{itemize}
such that $(u^{\gamma_k}, y^{\gamma_k}, \lambda^{\gamma_k}, \eta^{\gamma_k}, \mu^{\gamma_k})$ satisfy \eqref{eq:kkt1}.
The sequence admits a limit point
$$
(u^{*}, y^{*}, \Lambda^{*}, \eta^{*}, \rho^{*}) \in L^2(D) \times L^2_{\mathbb{P}}(\Omega; H_0^1(D) \cap H^2(D)) \times L^2(D) \times L^2(D) \times \mathcal{M}(\overline{\Xi}),
$$
where
\begin{align}
u^{\gamma_k} & \rightarrow u^* & \text{in~} & L^2(D), \label{eq:uconv} \\
y^{\gamma_k} & \rightarrow y^* & \text{in~} & C(\Omega; H^1_0(D) \cap H^2(D)), \label{eq:yconv} \\
\mu^{\gamma_k} & \rightharpoonup^* \rho^* & \text{in~} &  \mathcal{M}(\overline{\Xi}), \label{eq:muconv} \\
\mathbb{E}_{\mathbb{P}}[B^* \lambda^{\gamma_k}] & \rightharpoonup \Lambda^* & \text{in~} & L^2(D), \label{eq:lambconv} \\
\eta^{\gamma_k} & \rightharpoonup \eta^* & \text{in~} & L^2(D). \label{eq:etaconv}
\end{align}
Moreover, the limit point satisfies
\begin{align}
u^* & \in \mathcal{U}_{ad} \label{eq:u*adm} \\
y^* & = S u^* + y_f & \text{and~} & y^* \ge \psi \label{eq:y*} \\
(\Lambda^*, \varphi) & = \left(\mathbb{E}_{\mathbb{P}}\left[\mathbf{B}^* \mathbf{A}^{-*} T^* (T (y^* + y_f) - y_d)\right], \varphi\right) \nonumber \\
& + \int_{\Xi} \mathbf{A}^{-1}(\omega) \mathbf{B}(\omega) \varphi d\rho^*(x,\omega) \label{eq:lambd*} \\
0 & = (\Lambda^*, \varphi) + \alpha (u^*, \varphi) + (\eta^*, \varphi) & \text{and~} & \eta^* \in \mathcal{N}_{\mathcal{U}_{ad}}(u^*) \label{eq:eta*}
\end{align}
for an arbitrary test function $\varphi \in L^2(D)$, and
\begin{align}
\langle \phi, \rho^* \rangle & \le 0, & \forall \phi \in C(\overline{\Xi}):~\phi \ge 0, \label{eq:rho*neg} \\
\langle \psi - (y^* + y_f), \rho^* \rangle & = 0. \label{eq:rho*constr}
\end{align}
\end{theorem}
Similarly to \cite{Surowiec-risk-2022}, the proof is split into the following lemmas.
\begin{lemma}\label{lem:weakconv}
Under the assumptions of Thm.~\ref{thm:conv} there exist sequences $\gamma_k \rightarrow \infty$ and $\varepsilon_{\gamma_k} = \mathcal{O}(\gamma_k^{-1/2}) \rightarrow 0$ such that the sequence of solutions $u^{\gamma_k}$ to \eqref{eq:problemmy} converges weakly to a feasible solution, satisfying \eqref{eq:u*adm}--\eqref{eq:y*}.
\end{lemma}
\begin{proof}
Since $\mathcal{U}_{ad}$ is weakly compact in $L^2(D)$ by Assumption~\ref{as:cost} and $u^{\gamma} \in \mathcal{U}_{ad}$ for any $\gamma>0$, for any sequence $\gamma_n \rightarrow \infty$ there exists a subsequence $\gamma_k:=\gamma_{n_k}$ and some $u^* \in \mathcal{U}_{ad}$ such that $u^{\gamma_k} \rightharpoonup u^*$ in $L^2(D)$.
By non-negativity of the Moreau-Yosida term (with any $\varepsilon\ge 0$),
\begin{align*}
& \mathbb{E}_{\mathbb{P}}\left[\frac{1}{2}\|T (S u^{\gamma_k} + y_f) - y_d\|_{L^2(D)}^2\right] + \frac{\alpha}{2}\|u^{\gamma_k}\|_{L^2(D)}^2 \\
& \le \mathbb{E}_{\mathbb{P}}\left[\frac{1}{2}\|T (S u^{\gamma_k} + y_f) - y_d\|_{L^2(D)}^2\right] + \frac{\alpha}{2}\|u^{\gamma_k}\|_{L^2(D)}^2 + \mathbb{E}_{\mathbb{P}}\left[\frac{\gamma_k}{2}\|(\psi - Su^{\gamma_k} - y_f)_+\|_{L^2(D)}^2\right] \\
& \le \mathbb{E}_{\mathbb{P}}\left[\frac{1}{2}\|T (S u^{\gamma_k} + y_f) - y_d\|_{L^2(D)}^2\right] + \frac{\alpha}{2}\|u^{\gamma_k}\|_{L^2(D)}^2 + \mathbb{E}_{\mathbb{P}}\left[\frac{\gamma_k}{2}\|g_{\varepsilon_{\gamma}}(Su^{\gamma_k} + y_f - \psi)\|_{L^2(D)}^2\right],
\end{align*}
where in the last line we used Lemma~\ref{lem:gprime}.

Due to the Slater condition \eqref{eq:slater}, there exists a $\delta>0$ and a sequence $v^{\gamma_k} \in \mathcal{U}_{ad}$ such that $S v^{\gamma_k} + y_f \ge \psi + \delta$ $\Xi$-a.s.
Again since $\mathcal{U}_{ad}$ is weakly compact, $v^{\gamma_k}$ is uniformly bounded in $L^2(D)$.
Since $u^{\gamma}$ is a minimizer in \eqref{eq:problemmy},
\begin{align*}
& \mathbb{E}_{\mathbb{P}}\left[\frac{1}{2}\|T (S u^{\gamma_k} + y_f) - y_d\|_{L^2(D)}^2\right] + \frac{\alpha}{2}\|u^{\gamma_k}\|_{L^2(D)}^2 + \mathbb{E}_{\mathbb{P}}\left[\frac{\gamma_k}{2}\|g_{\varepsilon_{\gamma}}(Su^{\gamma_k} + y_f - \psi)\|_{L^2(D)}^2\right] \\
& \le \mathbb{E}_{\mathbb{P}}\left[\frac{1}{2}\|T (S v^{\gamma_k} + y_f) - y_d\|_{L^2(D)}^2\right] + \frac{\alpha}{2}\|v^{\gamma_k}\|_{L^2(D)}^2 + \mathbb{E}_{\mathbb{P}}\left[\frac{\gamma_k}{2}\|g_{\varepsilon_{\gamma}}(Sv^{\gamma_k} + y_f - \psi)\|_{L^2(D)}^2\right].
\end{align*}
However, due to the Slater condition the last term is bounded by (also using that $g_\varepsilon$ is a decreasing function)
$$
\mathbb{E}_{\mathbb{P}}\left[\frac{\gamma_k}{2}\|g_{\varepsilon_{\gamma}}(Sv^{\gamma_k} + y_f - \psi)\|_{L^2(D)}^2\right] \le \mathbb{E}_{\mathbb{P}}\left[\frac{\gamma_k}{2} g_{\varepsilon_{\gamma_k}}(0)^2 \|1\|_{L^2(D)}^2 \right] = \frac{\gamma_k}{2} \varepsilon_{\gamma_k}^2 |D| \log^2 2 =: M_k.
$$
Choosing $\varepsilon_{\gamma_k} = \mathcal{O}(\gamma_k^{-1/2})$ we make $M_k \le M_* < \infty$.
Since $S$ is completely continuous into $L^2_{\pi}(\Xi)$ and $\mathcal{U}_{ad}$ is bounded, there exists $M < \infty$ such that
$$
\mathbb{E}_{\mathbb{P}}\left[\frac{1}{2}\|T (S v^{\gamma_k} + y_f) - y_d\|_{L^2(D)}^2\right] + \frac{\alpha}{2}\|v^{\gamma_k}\|_{L^2(D)}^2 \le M.
$$
This makes
$$
\mathbb{E}_{\mathbb{P}}\left[\frac{1}{2}\|T (S u^{\gamma_k} + y_f) - y_d\|_{L^2(D)}^2\right] + \frac{\alpha}{2}\|u^{\gamma_k}\|_{L^2(D)}^2 \le M + M_k < \infty.
$$
Since the cost function is weak lower semicontinuous,
this bound holds also for the limit
\begin{align*}
& \mathbb{E}_{\mathbb{P}}\left[\frac{1}{2}\|T (S u^{*} + y_f) - y_d\|_{L^2(D)}^2\right] + \frac{\alpha}{2}\|u^{*}\|_{L^2(D)}^2 \\
& \le \liminf_{k \rightarrow \infty} \left(\mathbb{E}_{\mathbb{P}}\left[\frac{1}{2}\|T (S u^{\gamma_k} + y_f) - y_d\|_{L^2(D)}^2\right] + \frac{\alpha}{2}\|u^{\gamma_k}\|_{L^2(D)}^2\right) \le M+M_*.
\end{align*}
Similarly, $\mathbb{E}_{\mathbb{P}}\left[\frac{\gamma_k}{2}\|(\psi - Su^{\gamma_k} - y_f)_+\|_{L^2(D)}^2\right]$ is bounded (using Lemma~\ref{lem:gprime}), which means
$$
\mathbb{E}_{\mathbb{P}}\left[\|(\psi - Su^{\gamma_k} - y_f)_+\|_{L^2(D)}^2\right] \rightarrow 0 \quad \text{as} \quad \gamma_k \rightarrow \infty.
$$
Since $Su^{\gamma_k} \rightarrow S u^*$ (strongly) in $L^2_{\pi}(\Xi)$ (due to complete continuity of $S$), we have
$$
\mathbb{E}_{\mathbb{P}}\left[\|(\psi - Su^{\gamma_k} - y_f)_+\|_{L^2(D)}^2\right] \rightarrow \mathbb{E}_{\mathbb{P}}\left[\|(\psi - Su^{*} - y_f)_+\|_{L^2(D)}^2\right] = 0.
$$
Thus, $u^* \in \mathcal{U}_{ad}$ such that $Su^* + y_f \ge \psi$ $\pi$-a.e.
\end{proof}
\begin{lemma}
Under the assumptions of Thm.~\ref{thm:conv}, there exists a smoothing parameter $\varepsilon_{\gamma} = o(\gamma^{-1/2}) \rightarrow 0$ as $\gamma \rightarrow \infty$ such that \eqref{eq:uconv} (strong convergence of control) holds.
\end{lemma}
\begin{proof}
Due to the Slater condition, we can find a sequence $v^k \in \mathcal{U}_{ad}$ that is uniformly bounded in $L^2(D)$, and $S v^k + y_f \ge \psi + \delta$.
Let $v^k(t_k) = t_k v^k + (1-t_k) u^*$,
where $u^*$ is the limit point from the previous lemma, and $t_k = 2^{-k}$.
Clearly, $\|v^k(t_k) - u^*\|_{L^2(D)} \le t_k \|v^k - u^*\|_{L^2(D)} \rightarrow 0$ as $k \rightarrow \infty$, so $v^k(t_k)$ converges strongly to $u^*$ in $L^2(D)$.
Moreover, since $S$ is linear, 
$$
S v^k(t_k) + y_f = t_k (Sv^k + y_f) + (1-t_k) (S u^* + y_f) \ge t_k (\psi + \delta) + (1-t_k) \psi \ge \psi \quad \pi\text{-a.s.,}
$$
so $v^k(t_k)$ satisfies the constraint.
By the same arguments as in the previous lemma,
\begin{align*}
& \mathbb{E}_{\mathbb{P}}\left[\frac{1}{2}\|T (S u^{\gamma_k} + y_f) - y_d\|_{L^2(D)}^2\right] + \frac{\alpha}{2}\|u^{\gamma_k}\|_{L^2(D)}^2 \\
& \le \mathbb{E}_{\mathbb{P}}\left[\frac{1}{2}\|T (S v^k(t_k) + y_f) - y_d\|_{L^2(D)}^2\right] + \frac{\alpha}{2}\|v^k(t_k)\|_{L^2(D)}^2 + \mathbb{E}_{\mathbb{P}}\left[\frac{\gamma_k}{2}\|g_{\varepsilon_{\gamma}}(Sv^k(t_k) + y_f - \psi)\|_{L^2(D)}^2\right] \\
& \le \mathbb{E}_{\mathbb{P}}\left[\frac{1}{2}\|T (S v^k(t_k) + y_f) - y_d\|_{L^2(D)}^2\right] + \frac{\alpha}{2}\|v^k(t_k)\|_{L^2(D)}^2 + \gamma_k \varepsilon_{\gamma_k}^2 \frac{|D|}{2} \log^2 2.
\end{align*}
Choosing $\varepsilon_{\gamma} = o(\gamma^{-1/2})$ such that $\gamma_k \varepsilon_{\gamma_k}^2 \rightarrow 0$ as $k \rightarrow \infty$,
we pass to the following limit:
\begin{align*}
& \limsup_{k \rightarrow \infty} \mathbb{E}_{\mathbb{P}}\left[\frac{1}{2}\|T (S u^{\gamma_k} + y_f) - y_d\|_{L^2(D)}^2\right] + \frac{\alpha}{2}\|u^{\gamma_k}\|_{L^2(D)}^2 \\
& \le \limsup_{k \rightarrow \infty} \mathbb{E}_{\mathbb{P}}\left[\frac{1}{2}\|T (S v^k(t_k) + y_f) - y_d\|_{L^2(D)}^2\right] + \frac{\alpha}{2}\|v^k(t_k)\|_{L^2(D)}^2.
\end{align*}
Due to complete continuity of $S$, $S u^{\gamma_k} \rightarrow S u^*$ and $S v^k(t_k) \rightarrow S u^*$.
Thus, resolving the limits above,
\begin{align*}
& \mathbb{E}_{\mathbb{P}}\left[\frac{1}{2}\|T (S u^{*} + y_f) - y_d\|_{L^2(D)}^2\right] + \frac{\alpha}{2}\limsup_{k \rightarrow \infty} \|u^{\gamma_k}\|_{L^2(D)}^2 \\
& \le \mathbb{E}_{\mathbb{P}}\left[\frac{1}{2}\|T (S u^* + y_f) - y_d\|_{L^2(D)}^2\right] + \frac{\alpha}{2}\|u^*\|_{L^2(D)}^2.
\end{align*}
That is, $\limsup_{k \rightarrow \infty} \|u^{\gamma_k}\|_{L^2(D)}^2 \le \|u^*\|_{L^2(D)}^2$, whereas due to the weak convergence of $u^{\gamma_k}$ from the previous lemma, $\liminf_{k \rightarrow \infty} \|u^{\gamma_k}\|_{L^2(D)}^2 \ge \|u^*\|_{L^2(D)}^2$.
This gives $\lim_{k \rightarrow \infty} \|u^{\gamma_k}\|_{L^2(D)}^2 = \|u^*\|_{L^2(D)}^2$, and, together with the weak convergence, the strong convergence.
\end{proof}
\begin{remark}
Note that $\varepsilon_{\gamma} = o(\gamma^{-1/2})$ also satisfies $\varepsilon_{\gamma} = \mathcal{O}(\gamma^{-1/2})$ in Lemma~\ref{lem:weakconv}.
\end{remark}
\begin{lemma}
Under the assumptions of Thm.~\ref{thm:conv}, the state variable converges strongly to a feasible solution, i.e. \eqref{eq:yconv} and \eqref{eq:y*} hold.
\end{lemma}
\begin{proof}
This follows trivially from the fact that $S(\cdot) + y_f: L^2(D) \rightarrow C(\Omega; H^1_0(D) \cap H^2(D))$ is continuous, and $u^{\gamma_k}$ converges strongly in $L^2(D)$.
\end{proof}
\begin{lemma}
Under the assumptions of Thm.~\ref{thm:conv}, there exists a sequence of weakly* converging multipliers $\mu^{\gamma_k}$, i.e. \eqref{eq:muconv} holds.
\end{lemma}
\begin{proof}
Recall that $\mu^{\gamma}$ is defined for any $u^{\gamma}$ by \eqref{eq:mugamma}.
We start with proving the existence of a $c_0 > 0$ such that $|(\mu^{\gamma}, v)| \le c_0$ for any $v \in \mathbb{B}_{\delta}(0) \subset L^{\infty}_{\pi}(\Xi)$ and some fixed $\delta>0$,
where $(\cdot, \cdot)$ is the inner product on $L^2_{\pi}(\Xi)$.
Let $\beta_{\varepsilon}(y):=\mathbb{E}_{\mathbb{P}}\left[\frac{1}{2}\|g_{\varepsilon}(y + y_f - \psi)\|_{L^2(D)}^2\right]$ for any $y\in L^2_{\pi}(\Xi)$.
Since $g_{\varepsilon}$ is convex and continuously differentiable, so is $\beta_{\varepsilon}$, and $\mu^{\gamma} = \gamma \nabla \beta_\varepsilon(y^{\gamma})$.
For any $y\in L^2_{\pi}(\Xi)$ such that $y + y_f \ge \psi$ $\pi$-a.s.,
using convexity of $\beta$, we obtain 
\begin{equation}\label{eq:mu*y}
\gamma \varepsilon^2 \frac{|D|}{2} \log^2 2 \ge \gamma \beta(y) \ge \gamma \beta(y^{\gamma}) + (\mu^{\gamma}, y - y^{\gamma}) \ge (\mu^{\gamma}, y - y^{\gamma}).
\end{equation}
By the Slater condition, there exists $\delta>0$ and $u^{\dagger} \in \mathcal{U}_{ad}$ such that for all $v \in \mathbb{B}_{\delta}(0) \subset L^{\infty}_{\pi}(\Xi)$, $S u^{\dagger} + y_f - \psi + v\ge 0$ holds.
Since $(\Omega, \mathcal{F}, \mathbb{P})$ is a complete probability space, $D$ is bounded, and the Lebesgue spaces are nested, it holds that $v \in L^2_{\pi}(\Xi)$ and $Su^{\dagger} + y_f \in L^2_{\pi}(\Xi)$.
Fixing an arbitrary $v \in \mathbb{B}_{\delta}(0) \subset L^\infty_\pi(\Xi)$, we have
\begin{align}
(\mu^{\gamma}, v) & = (\mu^{\gamma}, S u^{\dagger} + v + y_f - Su^{\gamma} - y_f) + (\mu^{\gamma}, Su^{\gamma} - Su^{\dagger}) \nonumber \\
& \le \gamma \varepsilon^2 \frac{|D|}{2} \log^2 2 + (\mu^{\gamma}, Su^{\gamma} - Su^{\dagger}) & \text{(by \eqref{eq:mu*y})} \nonumber \\
& = \gamma \varepsilon^2 \frac{|D|}{2} \log^2 2 + (\mu^{\gamma}, \mathbf{A}^{-1} \mathbf{B} (u^{\gamma} - u^{\dagger})) 
=: \textrm{I} + \textrm{II}
\label{eq:muv}
\end{align}
The term \textrm{I} can be uniformly bounded in $\gamma$ by letting $\varepsilon = \mathcal{O}(\gamma^{-1/2})$. The term \textrm{II}
can again be bounded exactly as in \cite[Lemma~5]{Surowiec-risk-2022}.
This gives $(\mu^{\gamma}, v) \le c_0 < \infty$ for an arbitrary $v \in \mathbb{B}_{\delta}(0)$ as needed. Subsequently, as in \cite[Lemma~5]{Surowiec-risk-2022}, we obtain that $\| \mu^\gamma \|_{L^1_\pi(\Xi)} \le \frac{1}{\delta}c_0 < \infty$ which implies that the sequence $\mu^{\gamma}$ is bounded in $L^1_{\pi}(\Xi)$, so we can extract a subsequence $\mu^{\gamma_k}$ which is weak* convergent to some regular countably additive Borel measure $\rho \in \mathcal{M}(\overline{\Xi})$ \cite[Cor.~2.4.3]{HAttouch_GButtazzo_GMichaille_2006}. Precise details are provided in Appendix~\ref{app:measure_proof}.
\end{proof}
\begin{lemma}
Under the assumptions of Thm.~\ref{thm:conv}, there hold \eqref{eq:lambconv} and \eqref{eq:etaconv} (the convergence of adjoint variables),
and the limit points satisfy \eqref{eq:lambd*} and \eqref{eq:eta*}.
\end{lemma}
\begin{proof}
The proof is identical to \cite[Lemmas~6,7,8]{Surowiec-risk-2022}, since those involve only the adjoint variables $\lambda^{\gamma},\eta^{\gamma}$ and the uniform boundedness and weak* convergence of the multiplier $\mu^{\gamma}$ in $L^1_{\pi}(\Xi)$ (established above), and hence are agnostic to the smoothing and its parameter $\varepsilon$, as long as it satisfies the previous lemmas.
\end{proof}
\begin{lemma}
Under the assumptions of Thm.~\ref{thm:conv}, \eqref{eq:rho*neg} and \eqref{eq:rho*constr} (limit measure constraints) hold.
\end{lemma}
\begin{proof}
By Lemma~\ref{lem:gprime},
$$
\mu^{\gamma} = \gamma g_{\varepsilon_{\gamma}}'(S u^{\gamma} + y_f - \psi) g_{\varepsilon_{\gamma}}(S u^{\gamma} + y_f - \psi) < 0 \qquad \pi\text{-a.s.}
$$
Therefore, for any nonnegative test function $\phi \in C(\overline{\Xi})$,
$$
\langle \phi, \mu^{\gamma} \rangle = \int_{\Omega} \int_D \phi \mu^{\gamma} d\pi \le 0.
$$
Since there exists a sequence $\mu^{\gamma_k} \rightharpoonup^* \rho^* \in \mathcal{M}(\overline\Xi)$, we have $\langle \phi, \mu^{\gamma_k} \rangle \rightarrow \langle \phi, \rho^* \rangle$ as $k \rightarrow \infty$, so $\rho^*$ is a negatively-signed measure \eqref{eq:rho*neg}.
Let $\phi_k = \psi - (S u^{\gamma_k} + y_f)$, which is continuous and converges strongly in $C(\overline{\Xi})$ to $\phi^* = \psi - (S u^* + y_f) \le 0$.
On one hand, since both $\rho^*$ and $\phi^*$ are negatively-signed, $\langle \phi^*, \rho^* \rangle \ge 0$.
On the other hand,
\begin{align*}
\langle \phi_k, \mu^{\gamma_k} \rangle & \le \max_{x,\omega: \phi_k(x,\omega) \le 0} \left[\phi_k \mu^{\gamma_k}\right] \langle 1,1 \rangle
  \le \gamma_k \max_{s \ge 0} \left[-s g_{\varepsilon_{\gamma_k}}'(s) g_{\varepsilon_{\gamma_k}}(s)\right] |D| \\
 & \le  -g_{\varepsilon_{\gamma_k}}'(0) |D| \gamma_k \max_{s \ge 0} \left[s g_{\varepsilon_{\gamma_k}}(s)\right]
 = |D| \frac{\gamma_k}{2} \max_{s \ge 0} \left[s g_{\varepsilon_{\gamma_k}}(s)\right]
 \le \gamma_k \varepsilon_{\gamma_k}^2 \frac{|D|}{2 \mathrm{e}}
\end{align*}
by Lemma~\ref{lem:gprime}, and $\langle \phi_k, \mu^{\gamma_k} \rangle \rightarrow \langle \phi^*, \rho^* \rangle$, so it holds $\langle \phi^*, \rho^* \rangle \le 0$ as long as $\varepsilon_{\gamma_k} = o(\gamma_k^{-1/2})$.
This means \eqref{eq:rho*constr}.
\end{proof}
This completes the proof of Theorem~\ref{thm:conv}.

\section{Practical Gauss-Newton method for smoothed Moreau-Yosida problem}
\label{sec:prac}
\subsection{Random field parametrization}
For practical computations, it is convenient to parametrize all random fields with independent identically distributed (i.i.d.) random variables with a known probability density function.
Those variables can then be sampled independently, and an expectation can be computed simply by quadrature.
Therefore, we will use the following assumption.
\begin{assumption}[finite noise]\label{as:rv}
There exists a $d$-dimensional random vector $\xi(\omega) \in \mathbb{R}^d$ with a product probability density function $\pi(\xi) = \pi(\xi_1) \cdots \pi(\xi_d),$ 
 such that any random field $X(\omega) \in L^2_{\mathbb{P}}(\Omega; V)$ can be expressed as a function $x(\xi) \in L^2_{\pi}(\mathbb{R}^d; V)$ such that $X(\omega) = x(\xi(\omega))$ a.s., and
 $$
 \mathbb{E}_{\mathbb{P}}[X] = \int_{\mathbb{R}^d} x(\xi) \pi(\xi) d\xi.
 $$
\end{assumption}
\begin{example}\label{ex:KLE}
Consider an elliptic PDE
$$
-\nabla \cdot (\kappa(x; \xi(\omega)) \nabla y) = u,
$$
with random diffusivity coefficient
$$
\kappa(x; \xi(\omega)) = \kappa_0(x) + \sum_{k=1}^{d} \psi_k(x) \xi_k(\omega)
$$
given by a Karhunen-Loeve expansion (see e.g., \cite{LordPowerShardlow}), where $\xi_k$ are independent random variables.
\end{example}

\subsection{Discretization and Derivatives of the Cost}
Theorem~\ref{thm:conv} was proven in the linear-quadratic case for simplicity, but the computational scheme below can be applied to more general equality constraints and cost functions.
Thus, we consider a more general version of \eqref{eq:problemmy}:
\begin{align}
\min_{u^{\gamma} \in \mathcal{U}_{ad}}~ & j(u^{\gamma}) + \mathbb{E}_{\mathbb{P}}\left[\frac{\gamma}{2}\|g_{\varepsilon_{\gamma}}(S(u^{\gamma}) + y_f - \psi)\|_{L^2(D)}^2 \right],
 \label{eq:Jmy}
\end{align}
where $j: \mathcal{U}_{ad} \rightarrow \mathbb{R}$, is bounded below, is once continuously
differentiable for first-order methods, and twice continuously differentiable for second-order methods.
In addition, we allow $S: \mathcal{U}_{ad} \rightarrow \mathcal{Y}$ to be nonlinear, but require it to 
be at least continuously differentiable. This will be sufficient for the Gauss-Newton Hessian below.
However, the exact Hessian requires $S^*$ to be twice continuously differentiable. We assume that
$\mathcal{Y}$ and $\mathcal{U}$ are appropriate Banach spaces and $\mathcal{U}_{ad} \subset \mathcal{U}$.

Since $S$ involves usually the solution of a differential equation, it needs to be discretized (using e.g. Finite Element methods and/or time integration schemes).
For a given mesh parameter $h>0$, we introduce the discretized (possibly nonlinear) operator $\mathbf{S}_h(\xi): \mathcal{U}_{ad} \rightarrow \mathbb{R}^{n_y}$,
where $n_y$ is the number of degrees of freedom in the discrete state variable.
Let $\mathbf{y}(\xi(\omega)) \in L^2_{\pi}(\mathbb{R}^d; \mathbb{R}^{n_y})$ be a random vector discretizing the random field $Y(\omega) \in L^2_{\mathbb{P}}(\Omega; V)$.
Its norm can be written as an expectation of a vector quadratic form,
$$
\|\mathbf{y}\|^2 := \mathbb{E}_{\mathbb{P}}\left[\|\mathbf{y}(\xi)\|_{\mathbf{M}}^2\right], \qquad  \|\mathbf{y}(\xi)\|_{\mathbf{M}}^2:= \mathbf{y}(\xi)^\top \mathbf{M} \mathbf{y}(\xi),
$$
where $\mathbf{M}=\mathbf{M}^\top>0 \in \mathbb{R}^{n_y \times n_y}$ is a mass matrix.
Similarly, we discretize $y_f(\omega) \mapsto \mathbf{y}_f(\xi) \in L^2_{\pi}(\mathbb{R}^d; \mathbb{R}^{n_y})$.
The discretized cost is denoted by $j_h(u)\approx j(u)$,
and the discretized constraint is $\boldsymbol\psi_h(\xi) \in L^2_{\pi}(\mathbb{R}^d; \mathbb{R}^{n_y})$.
Now, the semi-discretized Moreau-Yosida cost function \eqref{eq:Jmy} becomes
\begin{equation}
 j_{\gamma,h}(u):=j_h(u) + \frac{\gamma}{2} \mathbb{E}_{\mathbb{P}}\left[\|g_{\varepsilon_{\gamma}}(\mathbf{S}_h(\xi) u + \mathbf{y}_f(\xi) - \boldsymbol\psi_h(\xi))\|_{\mathbf{M}}^2\right].
 \label{eq:Jmyh}
\end{equation}
To derive optimization methods, we compute its gradient and Hessian:
\begin{align}
 \nabla_u j_{\gamma,h}(u) & = \nabla_u j_h + \gamma  \mathbb{E}_{\mathbb{P}}\left[\mathbf{S}_h^* \cdot \mathrm{diag}(g'_{\varepsilon_{\gamma}}(\mathbf{S}_h u + \mathbf{y}_f - \boldsymbol\psi_h)) \cdot \mathbf{M}  g_{\varepsilon_{\gamma}}(\mathbf{S}_h u + \mathbf{y}_f - \boldsymbol\psi_h)\right], \label{eq:gradJ} \\
 \nabla_{uu}^2 j_{\gamma,h}(u) & =  \nabla^2_{uu} j_h + \gamma \mathbb{E}_{\mathbb{P}}\left[\mathbf{S}_h^* \cdot \mathrm{diag}(g'_{\varepsilon_{\gamma}}(\mathbf{r})) \mathbf{M} \mathrm{diag}(g'_{\varepsilon_{\gamma}}(\mathbf{r})) \cdot  \mathbf{S}_h' \right] \label{eq:nablay_nablay}\\
 & + \gamma \mathbb{E}_{\mathbb{P}}\left[\mathbf{S}_h^* \cdot (\mathrm{tendiag}(g''_{\varepsilon_{\gamma}}(\mathbf{r})) \times_3 (\mathbf{M} g_{\varepsilon_{\gamma}}(\mathbf{r}))) \cdot  \mathbf{S}_h' \right] \label{eq:nabla2_G} \\
 & + \gamma \mathbb{E}_{\mathbb{P}}\left[\nabla_u \mathbf{S}_h^* \times_3 (\mathrm{diag}(g'_{\varepsilon_{\gamma}}(\mathbf{r}) \cdot \mathbf{M}  g_{\varepsilon_{\gamma}}(\mathbf{r})) \right], \label{eq:nabla2_y}
\end{align}
where we denoted $\mathbf{r}:=\mathbf{S}_h u + \mathbf{y}_f - \boldsymbol\psi_h$ for brevity, $\mathrm{tendiag}(\cdot)$ is producing a $3$-dimensional tensor out of vector by putting the vector elements along the diagonal, and zero elements otherwise, and $\times_3$ is the tensor-vector contraction product over the $3$d mode of the tensor.
If $\mathbf{S}_h$ is a nonlinear operator, $\mathbf{S}_h' = \nabla_u \mathbf{S}_h(u)$ denotes the gradient of an image of $u$, and $\mathbf{S}_h^*$ is the adjoint of $\mathbf{S}_h'$.

\subsection{Matrix-free Fixed Point Gauss-Newton Hessian}
The exact assembly of all terms of the Hessian \eqref{eq:nablay_nablay}--\eqref{eq:nabla2_y} can be too computationally expensive, since this involves dense tensor-valued random fields (such as $\nabla_u \mathbf{S}_h^*$).
To simplify the computations, we can firstly omit the terms \eqref{eq:nabla2_G} and \eqref{eq:nabla2_y} which contain order-3 tensors.
Secondly, we can replace the exact expectation by a fixed-point evaluation.
In the risk-neutral scenario, we can assume that the cost is defined via an expectation,
$j(u) = \mathbb{E}_{\mathbb{P}}[J(u; \xi)]$ for some random field $J(\cdot; \xi)$ (for example, $J(u;\xi) = \frac{1}{2}\|T S(\xi) u + y_f(\xi) - y_d(\xi)\|_{L^2(D)}^2 + \frac{\alpha}{2}\|u\|_{L^2(D)}^2$),
and after discretization,
$j_h(u) = \mathbb{E}_{\mathbb{P}}[J_h(u; \xi)]$.
Its Hessian can then be written as
$$
\nabla_{uu}^2 j_h(u) = \mathbb{E}_{\mathbb{P}}\left[\nabla_{uu}^2 J_h(u; \xi)\right].
$$
Now we can replace $\nabla_{uu}^2 j_h = \mathbb{E}_{\mathbb{P}}[\nabla_{uu}^2 J_h(u;\xi)]$ by
$$
\tilde \nabla_{uu}^2 j_h(u) := \nabla_{uu}^2 J_h(u; \mathbb{E}_{\mathbb{P}}[\xi]).
$$
This is exact if $\nabla_{uu}^2 J_h$ is linear in $\xi$, but we can take it as an approximation in the general case too.
Now to apply $\tilde \nabla_{uu}^2 j_h$ to a vector we just need to apply one deterministic $\nabla_{uu}^2 J_h(u; \mathbb{E}_{\mathbb{P}}[\xi])$, which involves solving one forward, one adjoint, and two linear sensitivity (of state and adjoint) deterministic problems in the most general setting \cite[Ch.~1, Algo.~2]{HAntil_DPKouri_MDLacasse_DRidzal_2018a}.

Similarly we approximate the second term in \eqref{eq:nablay_nablay} by
$$
\gamma \mathbf{S}_h^*(\xi_*) \mathbf{M} \mathbf{S}_h'(\xi_*),
$$
where
$$
\xi_* = \frac{\mathbb{E}_{\mathbb{P}}\left[-\xi \cdot \mathbf{1}^\top g'_{\varepsilon_{\gamma}}(\mathbf{S}_h u + \mathbf{y}_f - \boldsymbol\psi_h)\right]}{\mathbb{E}_{\mathbb{P}}\left[-\mathbf{1}^\top g'_{\varepsilon_{\gamma}}(\mathbf{S}_h u + \mathbf{y}_f - \boldsymbol\psi_h)\right]}
$$
is the mean of the random variable with respect to the probability density $\pi_{g'} \propto \pi \cdot (-\mathbf{1}^\top g'_{\varepsilon_{\gamma}}(\mathbf{S}_h u + \mathbf{y}_f - \boldsymbol\psi_h))$,
and $\mathbf{1}\in\mathbb{R}^{n_y}$ is the constant vector, averaging the spatial components.
Note that by Lemma~\ref{lem:gprime}, $-\mathbf{1}^\top g'_{\varepsilon_{\gamma}}(\mathbf{S}_h u + \mathbf{y}_f - \boldsymbol\psi_h)$ is a positive function bounded by $n_y$, so $-\pi \mathbf{1}^\top g'_{\varepsilon_{\gamma}}(\mathbf{S}_h u + \mathbf{y}_f - \boldsymbol\psi_h)$ is nonnegative and normalizable, and $\pi_{g'}$ is indeed a probability density.

Finally, we obtain a deterministic approximate Hessian
\begin{equation}\label{eq:hess-approx}
\mathbf{\tilde H} = \nabla_{uu}^2 J_h(u; \mathbb{E}_{\mathbb{P}}[\xi]) + \gamma \mathbf{S}_h^*(\xi_*) \mathbf{M} \mathbf{S}_h'(\xi_*),
\end{equation}
which can be applied to a vector by solving 2 forward, 2 adjoint, and 2 sensitivity problems. Similar construction of approximate Hessian has been considered in \cite{SGarreis_MUlbrich_2017a}.

\subsection{Tensor-Train decomposition and practical algorithm}
\label{s:tensor}
Recall that the bottleneck is the computation of the expectation in e.g. gradient \eqref{eq:gradJ}.
While it may be possible to use a Monte Carlo quadrature, its convergence is usually slow, which may make estimates of small values of the gradient near the optimum particularly inaccurate.
Instead, we will use the Tensor-Train (TT) decomposition~\cite{osel-tt-2011}, specifically its functional extension~\cite{Marzouk-stt-2016,Gorodetsky-ctt-2019},
to approximate discretised random fields in the form of multivariate functions of random parameters.
The TT approximation is a known methodology nowadays, but for the sake of completeness it is recalled in Appendix~\ref{app:tt}.
Specifically, to compute the gradient of the cost function \eqref{eq:gradJ},
we need to approximate the function under the expectation,
\begin{equation}\label{eq:gradXi}
 \mathbf{G}^{\varepsilon,h}_{u}(\xi) := \mathbf{S}_h(\xi)^* \cdot \mathrm{diag}(g'_{\varepsilon}(\mathbf{S}_h(\xi) u + \mathbf{y}_f - \boldsymbol\psi_h)) \cdot \mathbf{M}  g_{\varepsilon}(\mathbf{S}_h(\xi) u + \mathbf{y}_f - \boldsymbol\psi_h).
\end{equation}
This function can be approximated in the TT format directly from a moderate number of evaluations at certain points $\xi$,
or, in some cases, it can be accelerated by a two-stage approach: approximating the state $\mathbf{y}(\xi)$ first, followed by plugging in the fast TT interpolation of the approximate $\mathbf{y}(\xi)$ into \eqref{eq:gradXi}.
More details on this are described in Appendix~\ref{app:prac}.
The overall pseudocode of the smoothed Moreau-Yosida optimization is listed in Algorithm~\ref{alg:as}.

\begin{algorithm}[htb]
\caption{Inexact projected Newton optimization with smoothed a.s. constraints}
\label{alg:as}
\begin{algorithmic}[1]
\Require Procedures to compute $\mathbf{S}_h u, j_h(u),\nabla_{u} j_h(u)$, $\mathbf{y}_f$, constraint $\boldsymbol\psi_h$, initial and maximal Moreau-Yosida parameters $\gamma_0,\gamma_*,$ initial smoothing parameter $\varepsilon_0$, initial control $u_0$, approximation and stopping tolerance $\mathrm{tol}$, maximal number of iterations $L$, Armijo tuning parameter $\theta\in(0,1)$, minimal step size $\delta_{\min} \in (0,1)$.
\Ensure Optimized control $u^{\gamma_*}$.
\State Set iteration number $\ell=0$, step size $\delta=1$, $u_{-1}=u_0$.
\While{$\ell<L$ \textbf{and} $\delta>\delta_{\min}$ \textbf{and} $\|u_{\ell} - u_{\ell-1}\|_{L^2(D)} > \mathrm{tol}\cdot \|u_{\ell}\|_{L^2(D)}$ \textbf{or} $\ell=0$ \textbf{or} $\gamma_{\ell}<\gamma_*$}
 \State Set $\varepsilon_{\gamma_{\ell}} = \varepsilon_0 / \sqrt{\gamma_{\ell}}$.
 \State Approximate $\mathbf{\tilde G}^{\varepsilon_{\gamma_{\ell}},h}_{u_{\ell}}(\xi) \approx \mathbf{G}^{\varepsilon_{\gamma_{\ell}},h}_{u_{\ell}}(\xi)$ as shown in \eqref{eq:gradXi} in the TT format up to tolerance $\mathrm{tol}$.
 \State Approximate $\mathbf{\tilde g'}_{\varepsilon_{\gamma_{\ell}}}(\xi) \approx g'_{\varepsilon_{\gamma_{\ell}}}(\mathbf{S}_h(\xi) u_{\ell} + \mathbf{y}_f - \boldsymbol\psi_h)$ in the TT format up to tolerance $\mathrm{tol}$.
 \State Compute the gradient $\nabla_u j_{\gamma_{\ell},h} = \nabla_u j_h(u_{\ell}) + \gamma_{\ell}  \mathbb{E}_{\mathbb{P}}[\mathbf{\tilde G}^{\varepsilon_{\gamma_{\ell}},h}_{u_{\ell}}(\xi)]$
 \State Compute the anchor point $\xi_* = \mathbb{E}_{\mathbb{P}}[-\xi \cdot \mathbf{1}^\top \mathbf{\tilde g'}_{\varepsilon_{\gamma_{\ell}}}(\xi)] / \mathbb{E}_{\mathbb{P}}[-\mathbf{1}^\top \mathbf{\tilde g'}_{\varepsilon_{\gamma_{\ell}}}(\xi)]$.
 \State Compute the Newton direction $v = -\mathbf{\tilde H}^{-1} \nabla_u j_{\gamma_{\ell},h}$ using \eqref{eq:hess-approx}.
 \State Set step size $\delta=1$.
 \While{$j_{\gamma_{\ell},h}(\mathcal{P}_{\mathcal{U}_{ad}}(u_{\ell} + \delta v)) > j_{\gamma_{\ell},h}(u_{\ell}) + \delta \theta \left( v, \nabla_u j_{\gamma_{\ell},h} \right)_{L^2(D)}$ \textbf{and} $\delta>\delta_{\min}$}
   \State Set $\delta = \delta/2$.
 \EndWhile
 \State Set $u_{\ell+1} = \mathcal{P}_{\mathcal{U}_{ad}}(u_{\ell} + \delta v)$.
 \State Set $\gamma_{\ell+1} = \min\{2\gamma_{\ell}, \gamma_*\}.$
 \State Set $\ell=\ell+1$.
\EndWhile\\
\Return $u^{\gamma_*} = u_{\ell}$.
\end{algorithmic}
\end{algorithm}

\section{Numerical examples}
\label{s:nex}

We start with $\gamma_0=1$ and double $\gamma_{\ell+1} = 2\gamma_{\ell}$ in the course of the Newton iterations until a desired value of $\gamma_*$ is reached.
The smoothing parameter is set as $\varepsilon_{\gamma_{\ell}} = 0.5/\sqrt{\gamma_{\ell}}$.
The iteration is stopped when $\gamma_L$ has reached the maximal desired value $\gamma_*$,
and the step size has become smaller than $\delta_{\min}=10^{-3}$.
We always take a zero control as the initial guess $u_0$, and $\theta=10^{-4}$.
All computations are carried out in MATLAB 2020b on a Intel Xeon E5-2640 v4 CPU, using 
TT-Toolbox (\url{https://github.com/oseledets/TT-Toolbox}).

\subsection{One-dimensional Elliptic PDE}

We consider an elliptic PDE example from \cite{Surowiec-stability-2021,Surowiec-risk-2022}.
Here, a misfit functional
$$
j(u) = \frac{1}{2}\mathbb{E}\left[\|y(x,\xi) - y_d(x)\|_{L^2(D)}^2\right] + \frac{\alpha}{2}\|u(x)\|_{L^2(D)}^2
$$
is optimized constrained by a PDE with random coefficients\footnote{Note that \cite{Surowiec-stability-2021,Surowiec-risk-2022} considered the constraint $y \ge 0$, so here we reverse the sign of $y$ to make the constraint in the form \eqref{eq:problemc}.}
\begin{align}\label{eq:1dpde}
\begin{aligned}
 \nu(\xi) \Delta y(x,\xi) & = f(x,\xi) + u(x), & (x,\xi) & \in D \times \mathbb{R}^4, \\
 \nu(\xi) & = 10^{\xi_1(\omega)-2}, &
 f(x,\xi) & = \frac{\xi_2(\omega)}{100}, \\
 y|_{x=0} & = -1-\frac{\xi_3(\omega)}{1000}, &
 y|_{x=1} & = -\frac{2+\xi_4(\omega)}{1000}
 \end{aligned}
\end{align}
where $D = (0,1)$, and $\xi(\omega) = (\xi_1(\omega), \ldots, \xi_4(\omega)) \sim \mathcal{U}(-1,1)^4$ is uniformly distributed.
We take the desired state $y_d(x) = -\sin(50 x/\pi)$ and the regularization parameter $\alpha=10^{-2}$.
Moreover, we add the constraints
$$
 -y(x,\xi) \ge \psi(x)\equiv 0 \mbox{\quad a.s., \quad and \quad}  -0.75 \le u(x) \le 0.75 \quad \mbox{a.e}.
$$
We discretize \eqref{eq:1dpde} in the spatial coordinate $x$ using linear finite elements on a uniform grid with $n_y$ interior points,
and in each random variable $\xi_k$ using $n_{\xi}$ Gauss-Legendre quadrature nodes on $(-1,1)$.
Note that we exclude the boundary points $x=0$ and $x=1$ due to the Dirichlet boundary conditions.
This spatial discretization is used for both $y$ and $u$.

Firstly, we study precomputation of the surrogate solution $\mathbf{\tilde y}(\xi)$ and adjoint operator $\mathbf{\tilde S}_h^*(\xi)$.
We fix $n_y=63$, $n_{\xi}=65$, the TT approximation tolerance $10^{-7}$ and the final Moreau-Yosida regularization parameter $\gamma_*=1000$.
The direct computation of the TT approximation of \eqref{eq:gradXi}
requires $995$ seconds of the CPU time due to the maximal TT rank of $87$.
In contrast, $\mathbf{\tilde S}_h^*$ has the maximal TT rank of $8$,
and the computation of $\mathbf{\tilde S}_h^*$ requires only $64$ seconds despite a larger $n_y \times n_y$ TT core carrying the spatial variables.
Using the surrogates $\mathbf{\tilde y}$ and $\mathbf{\tilde S}_h^*$,
the remaining computation of $\nabla_u j_{\gamma,h}$ can be completed in less than $15$ seconds.
The relative difference between the two approximations of $\nabla_u j_{\gamma,h}$ is below the TT approximation tolerance.
This shows that the surrogate forward solution can significantly speed up Algorithm~\ref{alg:as} without degrading its convergence,
so we use it in all remaining experiments in this subsection.

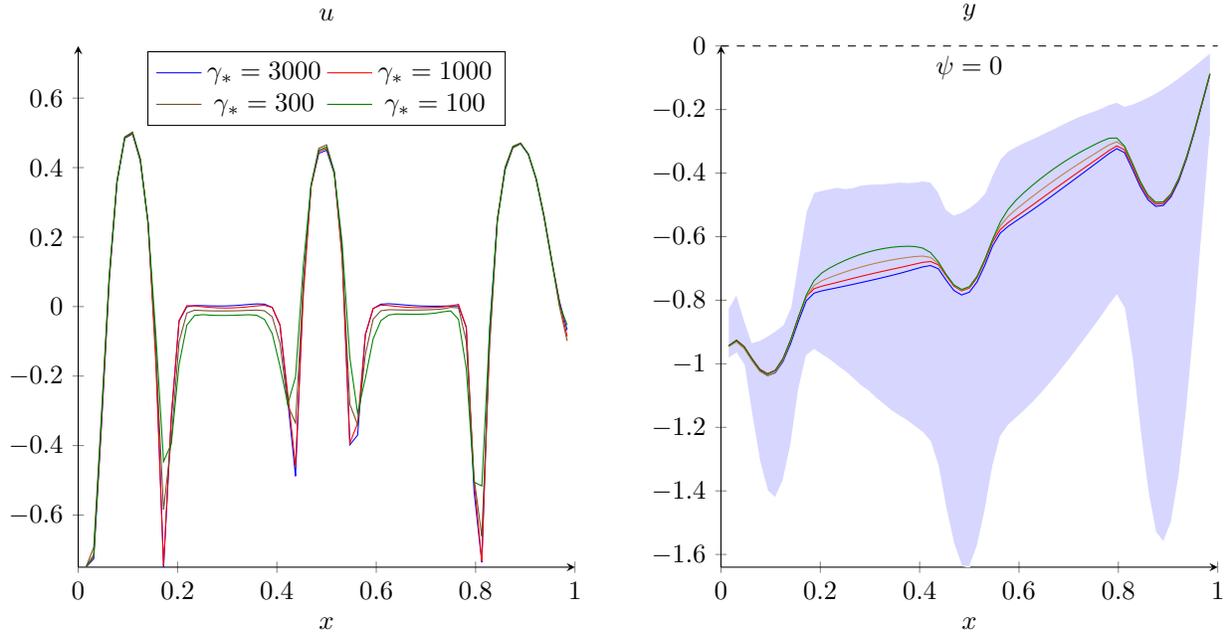
\begin{figure}[h!]
\noindent
\begin{tikzpicture}
 \begin{axis}[%
  width=0.40\linewidth,
  height=0.42\linewidth,
  xlabel=$x$,
  title=$u$,
  ymin=-0.75,ymax=0.75,
  xmin=0,xmax=1,
  legend style={at={(0.5,0.99)},anchor=north},legend columns=2,
  ]
  \addplot+[no marks] coordinates{
   (1.5625e-02,  -7.5000e-01)
   (3.1250e-02,  -7.2581e-01)
   (4.6875e-02,  -3.4279e-01)
   (6.2500e-02,   7.6452e-02)
   (7.8125e-02,   3.5760e-01)
   (9.3750e-02,   4.8468e-01)
   (1.0938e-01,   4.9702e-01)
   (1.2500e-01,   4.1953e-01)
   (1.4062e-01,   2.4272e-01)
   (1.5625e-01,  -1.2664e-01)
   (1.7188e-01,  -7.5000e-01)
   (1.8750e-01,  -3.2226e-01)
   (2.0312e-01,  -4.1730e-02)
   (2.1875e-01,  -5.7772e-04)
   (2.3438e-01,   2.9704e-03)
   (2.5000e-01,   2.7763e-03)
   (2.6562e-01,   2.0962e-03)
   (2.8125e-01,   1.6854e-03)
   (2.9688e-01,   1.8640e-03)
   (3.1250e-01,   2.7271e-03)
   (3.2812e-01,   4.2571e-03)
   (3.4375e-01,   6.2124e-03)
   (3.5938e-01,   7.8222e-03)
   (3.7500e-01,   7.1028e-03)
   (3.9062e-01,  -2.3978e-03)
   (4.0625e-01,  -5.3566e-02)
   (4.2188e-01,  -2.6936e-01)
   (4.3750e-01,  -4.8814e-01)
   (4.5312e-01,   5.9299e-02)
   (4.6875e-01,   3.4592e-01)
   (4.8438e-01,   4.4110e-01)
   (5.0000e-01,   4.5009e-01)
   (5.1562e-01,   3.8246e-01)
   (5.3125e-01,   1.6615e-01)
   (5.4688e-01,  -3.9777e-01)
   (5.6250e-01,  -3.6992e-01)
   (5.7812e-01,  -8.1033e-02)
   (5.9375e-01,  -5.8960e-03)
   (6.0938e-01,   7.0024e-03)
   (6.2500e-01,   8.3089e-03)
   (6.4062e-01,   6.8205e-03)
   (6.5625e-01,   4.8065e-03)
   (6.7188e-01,   3.0843e-03)
   (6.8750e-01,   1.8865e-03)
   (7.0312e-01,   1.2373e-03)
   (7.1875e-01,   1.0903e-03)
   (7.3438e-01,   1.3149e-03)
   (7.5000e-01,   1.9067e-03)
   (7.6562e-01,   1.6687e-03)
   (7.8125e-01,  -5.8598e-02)
   (7.9688e-01,  -5.3186e-01)
   (8.1250e-01,  -7.3532e-01)
   (8.2812e-01,  -7.0188e-02)
   (8.4375e-01,   2.5024e-01)
   (8.5938e-01,   3.9339e-01)
   (8.7500e-01,   4.5714e-01)
   (8.9062e-01,   4.6857e-01)
   (9.0625e-01,   4.3753e-01)
   (9.2188e-01,   3.6661e-01)
   (9.3750e-01,   2.5951e-01)
   (9.5312e-01,   1.3203e-01)
   (9.6875e-01,   1.5846e-02)
   (9.8438e-01,  -6.7535e-02)
  }; \addlegendentry{$\gamma_*=3000$};

  \addplot+[no marks] coordinates{
   (1.5625e-02,  -7.5000e-01)
   (3.1250e-02,  -7.1423e-01)
   (4.6875e-02,  -3.3382e-01)
   (6.2500e-02,   7.8656e-02)
   (7.8125e-02,   3.5871e-01)
   (9.3750e-02,   4.8597e-01)
   (1.0938e-01,   4.9840e-01)
   (1.2500e-01,   4.2261e-01)
   (1.4062e-01,   2.4105e-01)
   (1.5625e-01,  -1.6150e-01)
   (1.7188e-01,  -7.4453e-01)
   (1.8750e-01,  -3.1116e-01)
   (2.0312e-01,  -3.9907e-02)
   (2.1875e-01,   2.6614e-03)
   (2.3438e-01,   2.0747e-03)
   (2.5000e-01,  -5.1143e-04)
   (2.6562e-01,  -2.4386e-03)
   (2.8125e-01,  -3.6832e-03)
   (2.9688e-01,  -4.1258e-03)
   (3.1250e-01,  -3.7113e-03)
   (3.2812e-01,  -2.4161e-03)
   (3.4375e-01,  -3.2233e-04)
   (3.5938e-01,   2.1383e-03)
   (3.7500e-01,   3.6105e-03)
   (3.9062e-01,  -2.3356e-03)
   (4.0625e-01,  -5.2219e-02)
   (4.2188e-01,  -2.5075e-01)
   (4.3750e-01,  -4.5779e-01)
   (4.5312e-01,   2.9079e-02)
   (4.6875e-01,   3.4457e-01)
   (4.8438e-01,   4.4650e-01)
   (5.0000e-01,   4.5547e-01)
   (5.1562e-01,   3.8588e-01)
   (5.3125e-01,   1.3914e-01)
   (5.4688e-01,  -3.9258e-01)
   (5.6250e-01,  -3.3619e-01)
   (5.7812e-01,  -8.0133e-02)
   (5.9375e-01,  -5.2310e-03)
   (6.0938e-01,   4.1514e-03)
   (6.2500e-01,   2.9139e-03)
   (6.4062e-01,   5.4503e-04)
   (6.5625e-01,  -1.4637e-03)
   (6.7188e-01,  -2.7192e-03)
   (6.8750e-01,  -3.2091e-03)
   (7.0312e-01,  -3.0343e-03)
   (7.1875e-01,  -2.3071e-03)
   (7.3438e-01,  -1.0318e-03)
   (7.5000e-01,   2.0762e-03)
   (7.6562e-01,   6.5041e-03)
   (7.8125e-01,  -6.0967e-02)
   (7.9688e-01,  -5.0057e-01)
   (8.1250e-01,  -7.3248e-01)
   (8.2812e-01,  -1.0640e-01)
   (8.4375e-01,   2.5148e-01)
   (8.5938e-01,   3.9657e-01)
   (8.7500e-01,   4.5805e-01)
   (8.9062e-01,   4.6927e-01)
   (9.0625e-01,   4.3829e-01)
   (9.2188e-01,   3.6731e-01)
   (9.3750e-01,   2.6054e-01)
   (9.5312e-01,   1.3446e-01)
   (9.6875e-01,   1.4111e-02)
   (9.8438e-01,  -8.5872e-02)
  }; \addlegendentry{$\gamma_*=1000$};

  \addplot+[no marks] coordinates{
   (1.5625e-02,  -7.5000e-01)
   (3.1250e-02,  -6.9303e-01)
   (4.6875e-02,  -3.0900e-01)
   (6.2500e-02,   8.9941e-02)
   (7.8125e-02,   3.6208e-01)
   (9.3750e-02,   4.8830e-01)
   (1.0938e-01,   5.0239e-01)
   (1.2500e-01,   4.2726e-01)
   (1.4062e-01,   2.4044e-01)
   (1.5625e-01,  -1.2883e-01)
   (1.7188e-01,  -5.8427e-01)
   (1.8750e-01,  -3.8365e-01)
   (2.0312e-01,  -1.0109e-01)
   (2.1875e-01,  -1.8611e-02)
   (2.3438e-01,  -1.0182e-02)
   (2.5000e-01,  -1.1517e-02)
   (2.6562e-01,  -1.2294e-02)
   (2.8125e-01,  -1.2536e-02)
   (2.9688e-01,  -1.2516e-02)
   (3.1250e-01,  -1.2240e-02)
   (3.2812e-01,  -1.1644e-02)
   (3.4375e-01,  -1.0752e-02)
   (3.5938e-01,  -1.0014e-02)
   (3.7500e-01,  -1.1953e-02)
   (3.9062e-01,  -2.9442e-02)
   (4.0625e-01,  -1.0569e-01)
   (4.2188e-01,  -2.8367e-01)
   (4.3750e-01,  -3.3491e-01)
   (4.5312e-01,   6.1922e-02)
   (4.6875e-01,   3.4816e-01)
   (4.8438e-01,   4.5565e-01)
   (5.0000e-01,   4.6550e-01)
   (5.1562e-01,   3.8848e-01)
   (5.3125e-01,   1.4747e-01)
   (5.4688e-01,  -2.8096e-01)
   (5.6250e-01,  -3.3989e-01)
   (5.7812e-01,  -1.3794e-01)
   (5.9375e-01,  -3.6339e-02)
   (6.0938e-01,  -1.1998e-02)
   (6.2500e-01,  -9.0122e-03)
   (6.4062e-01,  -9.4287e-03)
   (6.5625e-01,  -9.9278e-03)
   (6.7188e-01,  -1.0059e-02)
   (6.8750e-01,  -9.8156e-03)
   (7.0312e-01,  -9.2500e-03)
   (7.1875e-01,  -8.3431e-03)
   (7.3438e-01,  -6.5481e-03)
   (7.5000e-01,  -2.1398e-03)
   (7.6562e-01,  -3.7824e-03)
   (7.8125e-01,  -9.9910e-02)
   (7.9688e-01,  -4.9343e-01)
   (8.1250e-01,  -6.6068e-01)
   (8.2812e-01,  -1.1668e-01)
   (8.4375e-01,   2.4672e-01)
   (8.5938e-01,   4.0198e-01)
   (8.7500e-01,   4.6165e-01)
   (8.9062e-01,   4.7088e-01)
   (9.0625e-01,   4.3962e-01)
   (9.2188e-01,   3.6956e-01)
   (9.3750e-01,   2.6452e-01)
   (9.5312e-01,   1.3608e-01)
   (9.6875e-01,   8.3250e-04)
   (9.8438e-01,  -9.8736e-02)
  };  \addlegendentry{$\gamma_*=300$};

  \addplot+[no marks,green!50!black] coordinates{
(1.5625e-02, -7.5000e-01)
(3.1250e-02, -7.2054e-01)
(4.6875e-02, -3.2034e-01)
(6.2500e-02, 8.5535e-02)
(7.8125e-02, 3.6069e-01)
(9.3750e-02, 4.8741e-01)
(1.0938e-01, 5.0059e-01)
(1.2500e-01, 4.2299e-01)
(1.4062e-01, 2.4113e-01)
(1.5625e-01, -7.7983e-02)
(1.7188e-01, -4.4608e-01)
(1.8750e-01, -3.9481e-01)
(2.0312e-01, -1.6650e-01)
(2.1875e-01, -5.3410e-02)
(2.3438e-01, -2.5463e-02)
(2.5000e-01, -2.3193e-02)
(2.6562e-01, -2.4587e-02)
(2.8125e-01, -2.5396e-02)
(2.9688e-01, -2.5504e-02)
(3.1250e-01, -2.5106e-02)
(3.2812e-01, -2.4236e-02)
(3.4375e-01, -2.3371e-02)
(3.5938e-01, -2.4973e-02)
(3.7500e-01, -3.6866e-02)
(3.9062e-01, -7.7838e-02)
(4.0625e-01, -1.7187e-01)
(4.2188e-01, -2.8098e-01)
(4.3750e-01, -2.0266e-01)
(4.5312e-01, 1.1783e-01)
(4.6875e-01, 3.5037e-01)
(4.8438e-01, 4.5010e-01)
(5.0000e-01, 4.5976e-01)
(5.1562e-01, 3.8550e-01)
(5.3125e-01, 1.8598e-01)
(5.4688e-01, -1.4742e-01)
(5.6250e-01, -3.0652e-01)
(5.7812e-01, -2.0617e-01)
(5.9375e-01, -9.3055e-02)
(6.0938e-01, -4.0298e-02)
(6.2500e-01, -2.4229e-02)
(6.4062e-01, -2.1368e-02)
(6.5625e-01, -2.1714e-02)
(6.7188e-01, -2.2152e-02)
(6.8750e-01, -2.1978e-02)
(7.0312e-01, -2.1000e-02)
(7.1875e-01, -1.8813e-02)
(7.3438e-01, -1.4948e-02)
(7.5000e-01, -1.2459e-02)
(7.6562e-01, -3.6990e-02)
(7.8125e-01, -1.7840e-01)
(7.9688e-01, -5.0514e-01)
(8.1250e-01, -5.1619e-01)
(8.2812e-01, -6.7618e-02)
(8.4375e-01, 2.4476e-01)
(8.5938e-01, 3.9644e-01)
(8.7500e-01, 4.5920e-01)
(8.9062e-01, 4.6950e-01)
(9.0625e-01, 4.3824e-01)
(9.2188e-01, 3.6751e-01)
(9.3750e-01, 2.6062e-01)
(9.5312e-01, 1.3052e-01)
(9.6875e-01, 5.2453e-03)
(9.8438e-01, -5.3042e-02)
  };  \addlegendentry{$\gamma_*=100$};

 \end{axis}
\end{tikzpicture}
\hfill
\begin{tikzpicture}
 \begin{axis}[%
 width=0.40\linewidth,
 height=0.42\linewidth,
 xlabel=$x$,
 title={$y$},
 xmin=0,xmax=1,
 ]
 \addplot+[no marks,color=white,name path=minus] coordinates{
(1.5625e-02,  -9.8224e-01)
(3.1250e-02,  -9.6694e-01)
(4.6875e-02,  -1.0040e+00)
(6.2500e-02,  -1.1620e+00)
(7.8125e-02,  -1.3071e+00)
(9.3750e-02,  -1.3995e+00)
(1.0938e-01,  -1.4219e+00)
(1.2500e-01,  -1.3669e+00)
(1.4062e-01,  -1.2547e+00)
(1.5625e-01,  -1.0819e+00)
(1.7188e-01,  -9.7447e-01)
(1.8750e-01,  -9.5536e-01)
(2.0312e-01,  -9.7134e-01)
(2.1875e-01,  -9.8725e-01)
(2.3438e-01,  -1.0070e+00)
(2.5000e-01,  -1.0243e+00)
(2.6562e-01,  -1.0458e+00)
(2.8125e-01,  -1.0687e+00)
(2.9688e-01,  -1.0855e+00)
(3.1250e-01,  -1.1088e+00)
(3.2812e-01,  -1.1281e+00)
(3.4375e-01,  -1.1487e+00)
(3.5938e-01,  -1.1644e+00)
(3.7500e-01,  -1.1793e+00)
(3.9062e-01,  -1.1980e+00)
(4.0625e-01,  -1.2163e+00)
(4.2188e-01,  -1.2446e+00)
(4.3750e-01,  -1.3221e+00)
(4.5312e-01,  -1.4554e+00)
(4.6875e-01,  -1.5661e+00)
(4.8438e-01,  -1.6348e+00)
(5.0000e-01,  -1.6401e+00)
(5.1562e-01,  -1.5747e+00)
(5.3125e-01,  -1.4642e+00)
(5.4688e-01,  -1.3240e+00)
(5.6250e-01,  -1.2291e+00)
(5.7812e-01,  -1.1940e+00)
(5.9375e-01,  -1.1714e+00)
(6.0938e-01,  -1.1494e+00)
(6.2500e-01,  -1.1257e+00)
(6.4062e-01,  -1.0999e+00)
(6.5625e-01,  -1.0724e+00)
(6.7188e-01,  -1.0435e+00)
(6.8750e-01,  -1.0134e+00)
(7.0312e-01,  -9.8236e-01)
(7.1875e-01,  -9.5015e-01)
(7.3438e-01,  -9.1765e-01)
(7.5000e-01,  -8.8486e-01)
(7.6562e-01,  -8.5121e-01)
(7.8125e-01,  -8.1518e-01)
(7.9688e-01,  -7.8425e-01)
(8.1250e-01,  -8.2526e-01)
(8.2812e-01,  -9.8289e-01)
(8.4375e-01,  -1.2066e+00)
(8.5938e-01,  -1.3982e+00)
(8.7500e-01,  -1.5291e+00)
(8.9062e-01,  -1.5614e+00)
(9.0625e-01,  -1.4998e+00)
(9.2188e-01,  -1.3554e+00)
(9.3750e-01,  -1.1394e+00)
(9.5312e-01,  -8.7330e-01)
(9.6875e-01,  -5.7949e-01)
(9.8438e-01,  -2.8396e-01)
};
 \addplot+[no marks,color=white,name path=plus] coordinates{
(1.5625e-02,  -8.2448e-01)
(3.1250e-02,  -7.8096e-01)
(4.6875e-02,  -8.6184e-01)
(6.2500e-02,  -9.3354e-01)
(7.8125e-02,  -9.2623e-01)
(9.3750e-02,  -9.1258e-01)
(1.0938e-01,  -8.9668e-01)
(1.2500e-01,  -8.7806e-01)
(1.4062e-01,  -8.2179e-01)
(1.5625e-01,  -6.7106e-01)
(1.7188e-01,  -5.1949e-01)
(1.8750e-01,  -4.6059e-01)
(2.0312e-01,  -4.5447e-01)
(2.1875e-01,  -4.4927e-01)
(2.3438e-01,  -4.4469e-01)
(2.5000e-01,  -4.4887e-01)
(2.6562e-01,  -4.4565e-01)
(2.8125e-01,  -4.3790e-01)
(2.9688e-01,  -4.3526e-01)
(3.1250e-01,  -4.3443e-01)
(3.2812e-01,  -4.3356e-01)
(3.4375e-01,  -4.3083e-01)
(3.5938e-01,  -4.2833e-01)
(3.7500e-01,  -4.2985e-01)
(3.9062e-01,  -4.2863e-01)
(4.0625e-01,  -4.2483e-01)
(4.2188e-01,  -4.2786e-01)
(4.3750e-01,  -4.5955e-01)
(4.5312e-01,  -5.1427e-01)
(4.6875e-01,  -5.3214e-01)
(4.8438e-01,  -5.2326e-01)
(5.0000e-01,  -5.0864e-01)
(5.1562e-01,  -4.8982e-01)
(5.3125e-01,  -4.6293e-01)
(5.4688e-01,  -4.0474e-01)
(5.6250e-01,  -3.5497e-01)
(5.7812e-01,  -3.3023e-01)
(5.9375e-01,  -3.1645e-01)
(6.0938e-01,  -3.0457e-01)
(6.2500e-01,  -2.9362e-01)
(6.4062e-01,  -2.8212e-01)
(6.5625e-01,  -2.7037e-01)
(6.7188e-01,  -2.5863e-01)
(6.8750e-01,  -2.4616e-01)
(7.0312e-01,  -2.3432e-01)
(7.1875e-01,  -2.2319e-01)
(7.3438e-01,  -2.1394e-01)
(7.5000e-01,  -2.0536e-01)
(7.6562e-01,  -1.9563e-01)
(7.8125e-01,  -1.8421e-01)
(7.9688e-01,  -1.7712e-01)
(8.1250e-01,  -1.8953e-01)
(8.2812e-01,  -1.8300e-01)
(8.4375e-01,  -1.7282e-01)
(8.5938e-01,  -1.6106e-01)
(8.7500e-01,  -1.4821e-01)
(8.9062e-01,  -1.3339e-01)
(9.0625e-01,  -1.1720e-01)
(9.2188e-01,  -9.9546e-02)
(9.3750e-01,  -8.0926e-02)
(9.5312e-01,  -6.1625e-02)
(9.6875e-01,  -4.1671e-02)
(9.8438e-01,  -2.1687e-02)
};

\addplot[blue!16!white] fill between[of = minus and plus];

 \addplot+[blue,no marks] coordinates{
(1.5625e-02, -9.4350e-01)
(3.1250e-02, -9.2636e-01)
(4.6875e-02, -9.4731e-01)
(6.2500e-02, -9.8626e-01)
(7.8125e-02, -1.0212e+00)
(9.3750e-02, -1.0374e+00)
(1.0938e-01, -1.0281e+00)
(1.2500e-01, -9.9272e-01)
(1.4062e-01, -9.3534e-01)
(1.5625e-01, -8.6522e-01)
(1.7188e-01, -8.0175e-01)
(1.8750e-01, -7.7765e-01)
(2.0312e-01, -7.7045e-01)
(2.1875e-01, -7.6545e-01)
(2.3438e-01, -7.6047e-01)
(2.5000e-01, -7.5534e-01)
(2.6562e-01, -7.5007e-01)
(2.8125e-01, -7.4468e-01)
(2.9688e-01, -7.3921e-01)
(3.1250e-01, -7.3364e-01)
(3.2812e-01, -7.2792e-01)
(3.4375e-01, -7.2198e-01)
(3.5938e-01, -7.1572e-01)
(3.7500e-01, -7.0905e-01)
(3.9062e-01, -7.0200e-01)
(4.0625e-01, -6.9508e-01)
(4.2188e-01, -6.9097e-01)
(4.3750e-01, -7.0099e-01)
(4.5312e-01, -7.3664e-01)
(4.6875e-01, -7.6917e-01)
(4.8438e-01, -7.8355e-01)
(5.0000e-01, -7.7478e-01)
(5.1562e-01, -7.4239e-01)
(5.3125e-01, -6.8992e-01)
(5.4688e-01, -6.2873e-01)
(5.6250e-01, -5.8842e-01)
(5.7812e-01, -5.6753e-01)
(5.9375e-01, -5.5088e-01)
(6.0938e-01, -5.3455e-01)
(6.2500e-01, -5.1785e-01)
(6.4062e-01, -5.0071e-01)
(6.5625e-01, -4.8322e-01)
(6.7188e-01, -4.6547e-01)
(6.8750e-01, -4.4756e-01)
(7.0312e-01, -4.2955e-01)
(7.1875e-01, -4.1148e-01)
(7.3438e-01, -3.9335e-01)
(7.5000e-01, -3.7515e-01)
(7.6562e-01, -3.5685e-01)
(7.8125e-01, -3.3846e-01)
(7.9688e-01, -3.2315e-01)
(8.1250e-01, -3.3575e-01)
(8.2812e-01, -3.8695e-01)
(8.4375e-01, -4.4183e-01)
(8.5938e-01, -4.8358e-01)
(8.7500e-01, -5.0467e-01)
(8.9062e-01, -5.0178e-01)
(9.0625e-01, -4.7429e-01)
(9.2188e-01, -4.2384e-01)
(9.3750e-01, -3.5415e-01)
(9.5312e-01, -2.7084e-01)
(9.6875e-01, -1.8060e-01)
(9.8438e-01, -8.9528e-02)
};

 \addplot+[red,no marks] coordinates{
(1.5625e-02, -9.4313e-01)
(3.1250e-02, -9.2562e-01)
(4.6875e-02, -9.4559e-01)
(6.2500e-02, -9.8309e-01)
(7.8125e-02, -1.0165e+00)
(9.3750e-02, -1.0310e+00)
(1.0938e-01, -1.0200e+00)
(1.2500e-01, -9.8291e-01)
(1.4062e-01, -9.2361e-01)
(1.5625e-01, -8.5165e-01)
(1.7188e-01, -7.8818e-01)
(1.8750e-01, -7.6377e-01)
(2.0312e-01, -7.5570e-01)
(2.1875e-01, -7.4973e-01)
(2.3438e-01, -7.4361e-01)
(2.5000e-01, -7.3738e-01)
(2.6562e-01, -7.3118e-01)
(2.8125e-01, -7.2511e-01)
(2.9688e-01, -7.1924e-01)
(3.1250e-01, -7.1358e-01)
(3.2812e-01, -7.0811e-01)
(3.4375e-01, -7.0277e-01)
(3.5938e-01, -6.9745e-01)
(3.7500e-01, -6.9201e-01)
(3.9062e-01, -6.8639e-01)
(4.0625e-01, -6.8089e-01)
(4.2188e-01, -6.7813e-01)
(4.3750e-01, -6.8853e-01)
(4.5312e-01, -7.2295e-01)
(4.6875e-01, -7.5585e-01)
(4.8438e-01, -7.7067e-01)
(5.0000e-01, -7.6205e-01)
(5.1562e-01, -7.2953e-01)
(5.3125e-01, -6.7675e-01)
(5.4688e-01, -6.1667e-01)
(5.6250e-01, -5.7720e-01)
(5.7812e-01, -5.5537e-01)
(5.9375e-01, -5.3774e-01)
(6.0938e-01, -5.2039e-01)
(6.2500e-01, -5.0283e-01)
(6.4062e-01, -4.8511e-01)
(6.5625e-01, -4.6736e-01)
(6.7188e-01, -4.4969e-01)
(6.8750e-01, -4.3216e-01)
(7.0312e-01, -4.1480e-01)
(7.1875e-01, -3.9759e-01)
(7.3438e-01, -3.8051e-01)
(7.5000e-01, -3.6349e-01)
(7.6562e-01, -3.4635e-01)
(7.8125e-01, -3.2888e-01)
(7.9688e-01, -3.1460e-01)
(8.1250e-01, -3.2659e-01)
(8.2812e-01, -3.7703e-01)
(8.4375e-01, -4.3306e-01)
(8.5938e-01, -4.7588e-01)
(8.7500e-01, -4.9789e-01)
(8.9062e-01, -4.9586e-01)
(9.0625e-01, -4.6920e-01)
(9.2188e-01, -4.1954e-01)
(9.3750e-01, -3.5060e-01)
(9.5312e-01, -2.6799e-01)
(9.6875e-01, -1.7832e-01)
(9.8438e-01, -8.7905e-02)
 };

 \addplot+[brown,solid,no marks] coordinates{
(1.5625e-02, -9.4591e-01)
(3.1250e-02, -9.3119e-01)
(4.6875e-02, -9.5284e-01)
(6.2500e-02, -9.9071e-01)
(7.8125e-02, -1.0239e+00)
(9.3750e-02, -1.0380e+00)
(1.0938e-01, -1.0265e+00)
(1.2500e-01, -9.8867e-01)
(1.4062e-01, -9.2839e-01)
(1.5625e-01, -8.5550e-01)
(1.7188e-01, -7.8937e-01)
(1.8750e-01, -7.5390e-01)
(2.0312e-01, -7.3857e-01)
(2.1875e-01, -7.2854e-01)
(2.3438e-01, -7.1949e-01)
(2.5000e-01, -7.1097e-01)
(2.6562e-01, -7.0306e-01)
(2.8125e-01, -6.9580e-01)
(2.9688e-01, -6.8919e-01)
(3.1250e-01, -6.8324e-01)
(3.2812e-01, -6.7793e-01)
(3.4375e-01, -6.7323e-01)
(3.5938e-01, -6.6910e-01)
(3.7500e-01, -6.6549e-01)
(3.9062e-01, -6.6251e-01)
(4.0625e-01, -6.6107e-01)
(4.2188e-01, -6.6519e-01)
(4.3750e-01, -6.8419e-01)
(4.5312e-01, -7.2076e-01)
(4.6875e-01, -7.5409e-01)
(4.8438e-01, -7.6915e-01)
(5.0000e-01, -7.6029e-01)
(5.1562e-01, -7.2699e-01)
(5.3125e-01, -6.7331e-01)
(5.4688e-01, -6.1189e-01)
(5.6250e-01, -5.6522e-01)
(5.7812e-01, -5.3638e-01)
(5.9375e-01, -5.1479e-01)
(6.0938e-01, -4.9510e-01)
(6.2500e-01, -4.7604e-01)
(6.4062e-01, -4.5746e-01)
(6.5625e-01, -4.3937e-01)
(6.7188e-01, -4.2180e-01)
(6.8750e-01, -4.0475e-01)
(7.0312e-01, -3.8823e-01)
(7.1875e-01, -3.7219e-01)
(7.3438e-01, -3.5658e-01)
(7.5000e-01, -3.4132e-01)
(7.6562e-01, -3.2618e-01)
(7.8125e-01, -3.1123e-01)
(7.9688e-01, -3.0152e-01)
(8.1250e-01, -3.1772e-01)
(8.2812e-01, -3.6858e-01)
(8.4375e-01, -4.2558e-01)
(8.5938e-01, -4.6962e-01)
(8.7500e-01, -4.9256e-01)
(8.9062e-01, -4.9128e-01)
(9.0625e-01, -4.6528e-01)
(9.2188e-01, -4.1620e-01)
(9.3750e-01, -3.4774e-01)
(9.5312e-01, -2.6539e-01)
(9.6875e-01, -1.7589e-01)
(9.8438e-01, -8.6355e-02)
 };

 \addplot+[green!50!black,solid,no marks] coordinates{
(1.5625e-02, -9.4369e-01)
(3.1250e-02, -9.2674e-01)
(4.6875e-02, -9.4761e-01)
(6.2500e-02, -9.8529e-01)
(7.8125e-02, -1.0185e+00)
(9.3750e-02, -1.0327e+00)
(1.0938e-01, -1.0214e+00)
(1.2500e-01, -9.8383e-01)
(1.4062e-01, -9.2404e-01)
(1.5625e-01, -8.5159e-01)
(1.7188e-01, -7.8324e-01)
(1.8750e-01, -7.3830e-01)
(2.0312e-01, -7.1407e-01)
(2.1875e-01, -6.9859e-01)
(2.3438e-01, -6.8591e-01)
(2.5000e-01, -6.7457e-01)
(2.6562e-01, -6.6444e-01)
(2.8125e-01, -6.5561e-01)
(2.9688e-01, -6.4811e-01)
(3.1250e-01, -6.4194e-01)
(3.2812e-01, -6.3710e-01)
(3.4375e-01, -6.3352e-01)
(3.5938e-01, -6.3118e-01)
(3.7500e-01, -6.3014e-01)
(3.9062e-01, -6.3104e-01)
(4.0625e-01, -6.3602e-01)
(4.2188e-01, -6.5002e-01)
(4.3750e-01, -6.7878e-01)
(4.5312e-01, -7.1816e-01)
(4.6875e-01, -7.5137e-01)
(4.8438e-01, -7.6618e-01)
(5.0000e-01, -7.5737e-01)
(5.1562e-01, -7.2443e-01)
(5.3125e-01, -6.7126e-01)
(5.4688e-01, -6.0833e-01)
(5.6250e-01, -5.5313e-01)
(5.7812e-01, -5.1402e-01)
(5.9375e-01, -4.8574e-01)
(6.0938e-01, -4.6233e-01)
(6.2500e-01, -4.4104e-01)
(6.4062e-01, -4.2103e-01)
(6.5625e-01, -4.0213e-01)
(6.7188e-01, -3.8438e-01)
(6.8750e-01, -3.6778e-01)
(7.0312e-01, -3.5234e-01)
(7.1875e-01, -3.3800e-01)
(7.3438e-01, -3.2465e-01)
(7.5000e-01, -3.1209e-01)
(7.6562e-01, -3.0018e-01)
(7.8125e-01, -2.9021e-01)
(7.9688e-01, -2.8960e-01)
(8.1250e-01, -3.1550e-01)
(8.2812e-01, -3.6850e-01)
(8.4375e-01, -4.2505e-01)
(8.5938e-01, -4.6874e-01)
(8.7500e-01, -4.9164e-01)
(8.9062e-01, -4.9043e-01)
(9.0625e-01, -4.6458e-01)
(9.2188e-01, -4.1573e-01)
(9.3750e-01, -3.4759e-01)
(9.5312e-01, -2.6577e-01)
(9.6875e-01, -1.7710e-01)
(9.8438e-01, -8.8160e-02)
 };

\addplot+[black,dashed,no marks,domain=0:1] {0*x} node [pos=0.5,anchor=north] {$\psi=0$};

\end{axis}
\end{tikzpicture}
\caption{Left: control signals $u_{\gamma_*}(x)$ for different $\gamma_*$. Right: mean (solid lines) and 95\% confidence interval (shaded area, for $\gamma_*=3000$ only) of the state $y_{\gamma_*}(x,\xi)$.
}
\label{fig:1dsol}
\end{figure}

In Figure~\ref{fig:1dsol} we show the solutions (control and state) for varying final Moreau-Yosida penalty parameter $\gamma_*$,
fixing $n_y=63$, $n_{\xi}=129$ and the TT approximation tolerance of $10^{-6}$.
We see that the solution converges with increasing $\gamma_*$, and larger $\gamma_*$ yields a smaller probability of the constraint violation, albeit at a larger misfit cost $j(u)$, as shown in Figure~\ref{fig:1dcost}.
In particular, $\gamma_* > 300$ gives a solution with less than $1\%$ of the constraint violation, such that the empirical $95\%$ confidence interval computed using $1000$ samples of the converged state field $y_{\gamma_*}$ (see Fig.~\ref{fig:1dsol}, right) is entirely within the constraint.

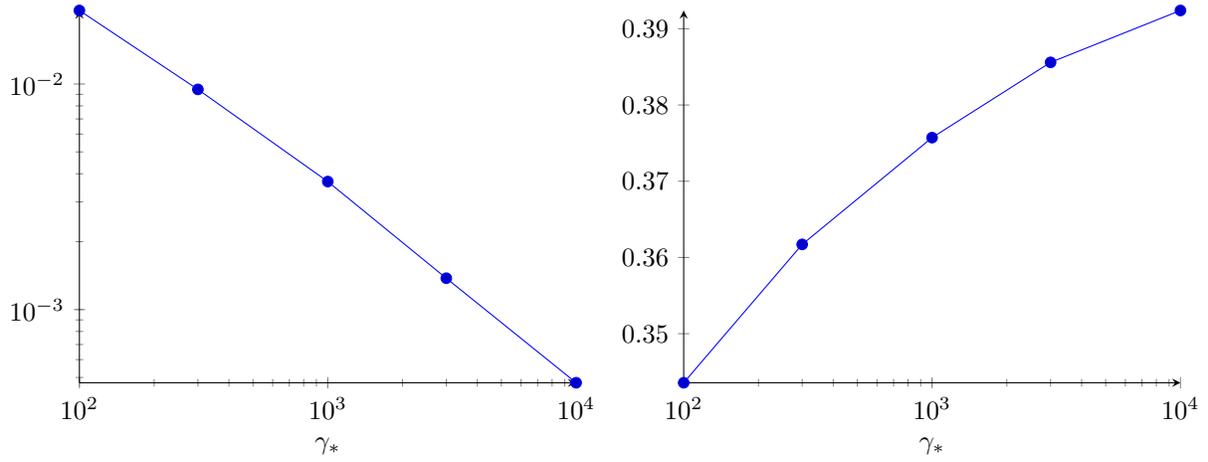
\begin{figure}[h!]
\centering
\noindent
\begin{tikzpicture}
\begin{axis}[%
 width=0.4\linewidth,
 height=0.3\linewidth,
 xmode=log,
 xlabel=$\gamma_*$,
 ymode=log,
 legend style={at={(0.01,0.01)},anchor=south west},
]
\addplot+[] coordinates{
(100,   0.021188   )
(300,   0.00946587 )
(1000,  0.00369432 )
(3000,  0.00137732 )
(10000, 0.000473718)
};
\end{axis}
\end{tikzpicture}
\begin{tikzpicture}
\begin{axis}[%
 width=0.4\linewidth,
 height=0.3\linewidth,
 xmode=log,
 xlabel=$\gamma_*$,
 ymode=normal,
 legend style={at={(0.01,0.01)},anchor=south west},
]
\addplot+[] coordinates{
(100,   0.343564)
(300,   0.361714)
(1000,  0.375718)
(3000,  0.385589)
(10000,  0.392396)
};
\end{axis}
\end{tikzpicture}
\caption{Left: probability of the constraint violation, $\mathbb{P}[y_{\gamma_*}(x,\xi)>0]$. Right: total final cost $j(u_{\gamma^*})$.}
\label{fig:1dcost}
\end{figure}

\begin{figure}[h!]
\noindent
\begin{tikzpicture}
\begin{axis}[%
 width=0.24\linewidth,
 height=0.25\linewidth,
 xmode=log,
 ymode=log,
 xlabel=$\gamma_*$,
 title=relative error,
 ]
 \addplot+[] coordinates{
   (1.0000e+02,   2.7252e-01)
   (3.0000e+02,   1.4798e-01)
   (1.0000e+03,   6.0698e-02)
   (3.0000e+03,   2.7418e-02)
 }; \addlegendentry{$u$};

 \addplot+[] coordinates{
   (1.0000e+02,   1.2611e-01)
   (3.0000e+02,   7.8444e-02)
   (1.0000e+03,   4.4097e-02)
   (3.0000e+03,   1.7411e-02)
 }; \addlegendentry{$y$};

 \addplot+[solid,black,no marks,domain=100:3000] {x^(-0.75)*11}; \addlegendentry{$\gamma_*^{-0.75}$};
 \addplot+[dashed,black,no marks,domain=100:3000] {x^(-0.5)*1.35}; \addlegendentry{$\gamma_*^{-0.5}$};
\end{axis}
\end{tikzpicture}
\hfill
\begin{tikzpicture}
\begin{axis}[%
width=0.24\linewidth,
height=0.25\linewidth,
xlabel=$n_{\xi}$,
xmode=normal,
ymode=log,
title={relative error},
xmin=5,xmax=50,
]
\addplot+[] coordinates{
(5 , 1.477e-02)
(7 , 1.379e-02)
(9 , 1.912e-03)
(13, 4.106e-04)
(17, 1.296e-04)
(25, 1.238e-05)
(33, 1.557e-05)
(49, 2.038e-05)
(65, 2.243e-05)
}; \addlegendentry{$u$};
\addplot+[] coordinates{
(5 , 2.995e-03)
(7 , 6.101e-03)
(9 , 2.467e-04)
(13, 4.197e-05)
(17, 1.469e-05)
(25, 6.073e-06)
(33, 1.194e-05)
(49, 1.551e-05)
(65, 1.699e-05)
}; \addlegendentry{$y$};
\end{axis}
\end{tikzpicture}
\begin{tikzpicture}
\begin{axis}[%
width=0.24\linewidth,
height=0.25\linewidth,
xlabel=$n_{y}$,
xmode=log,
xtick={31,65,127,255},
xticklabels={31,65,127,255},
xmin=30,xmax=300,
ymode=log,
title={relative error},
]
\addplot+[] coordinates{
(31,  1.5652e-01)
(65,  5.0169e-02)
(127, 1.8872e-02)
(255, 4.7021e-03)
}; \addlegendentry{$u$};
\addplot+[] coordinates{
(31,  3.3541e-02)
(65,  6.5895e-03)
(127, 3.7548e-03)
(255, 1.0513e-03)
}; \addlegendentry{$y$};

\addplot+[solid,black,no marks,domain=30:300] {x^(-1.5)*25}; \addlegendentry{$n_y^{-1.5}$};
\addplot+[solid,black,no marks,domain=30:300] {x^(-1.5)*5};
\end{axis}
\end{tikzpicture}
\caption{Relative $L^2$-norm difference from $y$ and $u$ to the reference solutions with $\gamma_*=10^4$ with fixed $n_{\xi}=257$, $n_y=63$ (left), $n_{\xi}=129$ with fixed $\gamma_*=100$, $n_y=63$ (middle) and $n_y=511$ with fixed $\gamma_*=100$, $n_{\xi}=25$ (right).}
\label{fig:1dconv}
\end{figure}
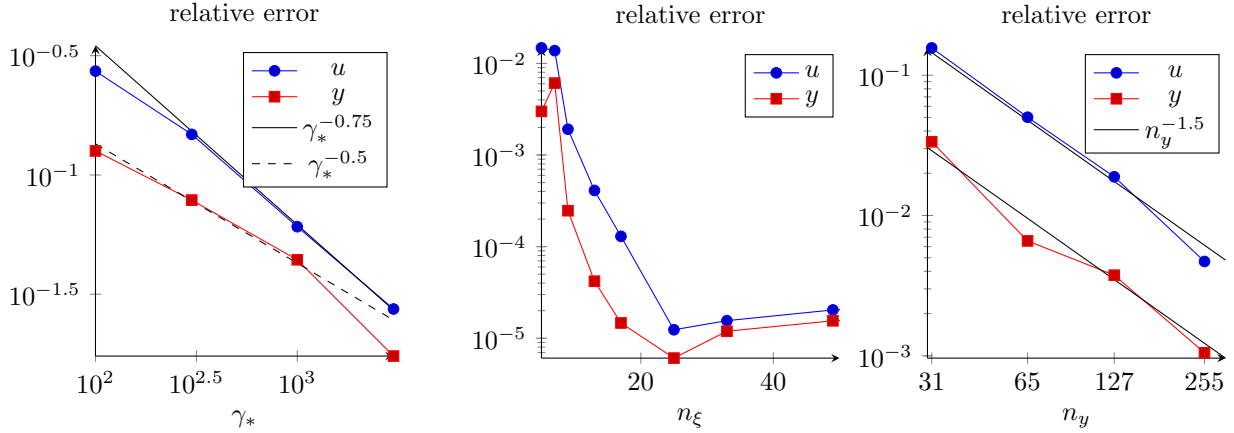

Finally, we study the convergence in the approximation parameters more systematically in Figure~\ref{fig:1dconv}.
In each plot we fix two out of three parameters: the final Moreau-Yosida penalty $\gamma_*$, the number of discretization points in the random variables $n_{\xi}$, and the number of discretization points in space $n_y$.
In addition, we fix the TT approximation threshold to $10^{-8}$ to reduce its influence.
We observe a convergence rate of $\gamma_*^{-1/2}$ in the penalty parameter, and an exponential convergence in $n_{\xi}$ (which is often the case for a polynomial approximation of smooth functions~\cite{trefethen-spectral-2000}) until the tensor approximation error is hit, and between first and second order in $n_y$, which seems to be an interplay of the discretization consistency of the linear finite elements (second order) and box constraints (first order).

\subsection{Two-dimensional elliptic PDE}
Now consider a two-dimensional extension of the previous problem, 
\begin{align}\label{eq:2dpde}
 \nu(\xi) \Delta y(x,\xi) & = f(x,\xi) + u(x), & (x,\xi) & \in D \times \mathbb{R}^6, \\
 y|_{x_1=0} & = b_1(\xi) (1-x_2) + b_2(\xi) x_2, &
 y|_{x_2=1} & = b_2(\xi) (1-x_1) + b_3(\xi) x_1 \\
 y|_{x_1=1} & = b_4(\xi) (1-x_2) + b_3(\xi) x_2, &
 y|_{x_2=0} & = b_1(\xi) (1-x_1) + b_4(\xi) x_1, \\
 \nu(\xi) & = 10^{\xi_1(\omega)-2}, & f(x,\xi) & = \frac{\xi_2(\omega)}{100}, \\
 b_1(\xi) & = -1-\frac{\xi_3(\omega)}{1000}, &
 b_2(\xi) & = -\frac{2+\xi_4(\omega)}{1000}, \\
 b_3(\xi) & = -1-\frac{\xi_5(\omega)}{1000}, &
 b_4(\xi) & = -\frac{2+\xi_6(\omega)}{1000},
\end{align}
where $D = (0,1)^2$, and $\xi(\omega) = (\xi_1(\omega), \ldots, \xi_6(\omega)) \sim \mathcal{U}(-1,1)^6$ is uniformly distributed.
We optimize the regularized misfit functional
$$
j(u) = \frac{1}{2}\mathbb{E}\left[\|y(x,\xi) - y_d(x)\|_{L^2(D)}^2\right] + \frac{\alpha}{2}\|u(x)\|_{L^2(D)}^2
$$
with the desired state $y_d(x) = -\sin(50 x_1/\pi) \cos(50 x_2/\pi)$ and the regularization parameter $\alpha=10^{-2}$,
subject to constraints
$$
 -y(x,\xi) \ge \psi(x)\equiv 0 \mbox{\quad a.s., \quad and \quad}  -0.75 \le u(x) \le 0.75 \quad \mbox{a.e}.
$$
We smooth the almost sure constraint by the Moreau-Yosida method with the ultimate penalty parameter $\gamma_*=10^2$.

We discretize both $y$ and $u$ in \eqref{eq:2dpde} using bilinear finite elements on a $n_y \times n_y$ rectangular grid.
For the two-dimensional problem, the operator $\mathbf{\tilde S}_h^*$ is a dense matrix of size $n_y^2 \times n_y^2$, which we are unable to precompute.
Therefore, we use the TT-Cross to approximate $\mathbf{G}_u^{\varepsilon_{\gamma_{\ell}},h}(\xi)$ directly.

In Figure~\ref{fig:2dsol} we show the optimal control, mean and standard deviation of the solution for $n_y=63$ and $n_{\xi}=17$.
We see that the mean solution reflects the desired state subject to the constraints.
The final cost $j(u_{\gamma_*})$ is about $0.222634$,
and the probability of the constraint violation is $0.0139223$.
The Newton method took $L=37$ iterations to converge,
the maximal TT rank of $\mathbf{\tilde y}(\xi)$ was $10$ which was the same in all iterations,
the maximal rank of $g'_{\varepsilon_{\gamma_{\ell}}}(\mathbf{\tilde y} + \mathbf{y}_f-\boldsymbol\psi_{h})$ was $300$, attained at the iteration after reaching $\gamma_*$ (iteration $9$),
and the maximal rank of $\mathbf{\tilde G}_u^{\varepsilon_{\gamma_{\ell}},h}(\xi)$ was $56$ (in the final iterations).
The computation took about a day of CPU time.
However, these TT ranks are comparable to those in the one-dimensional example.
This shows that the proposed technique can be also applied to a high-dimensional physical space, including complex domains and non-uniform grids, since the TT structure is independent of the spatial discretization.

\begin{figure}[h!]
\noindent
\includegraphics[width=0.32\linewidth]{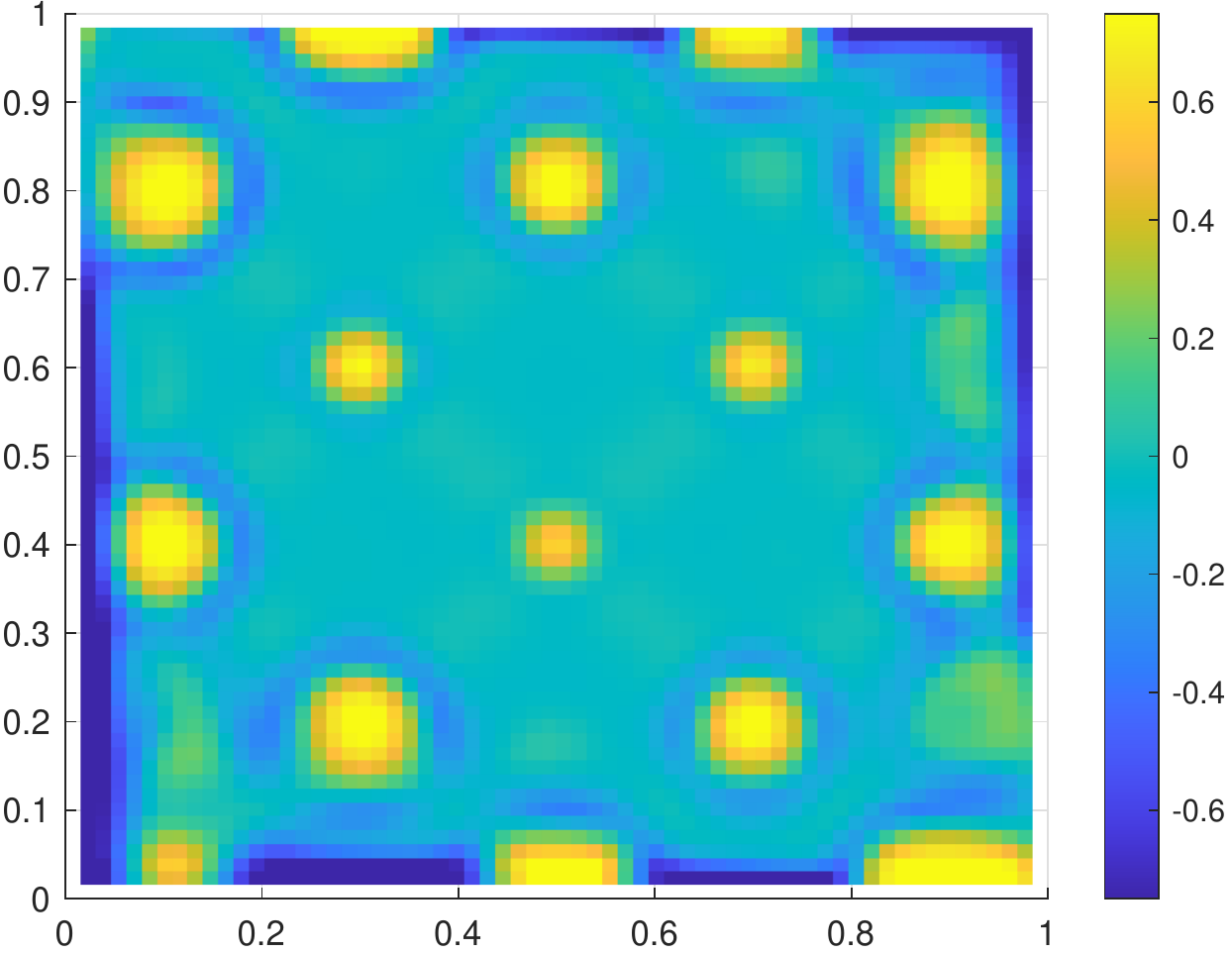}
\includegraphics[width=0.32\linewidth]{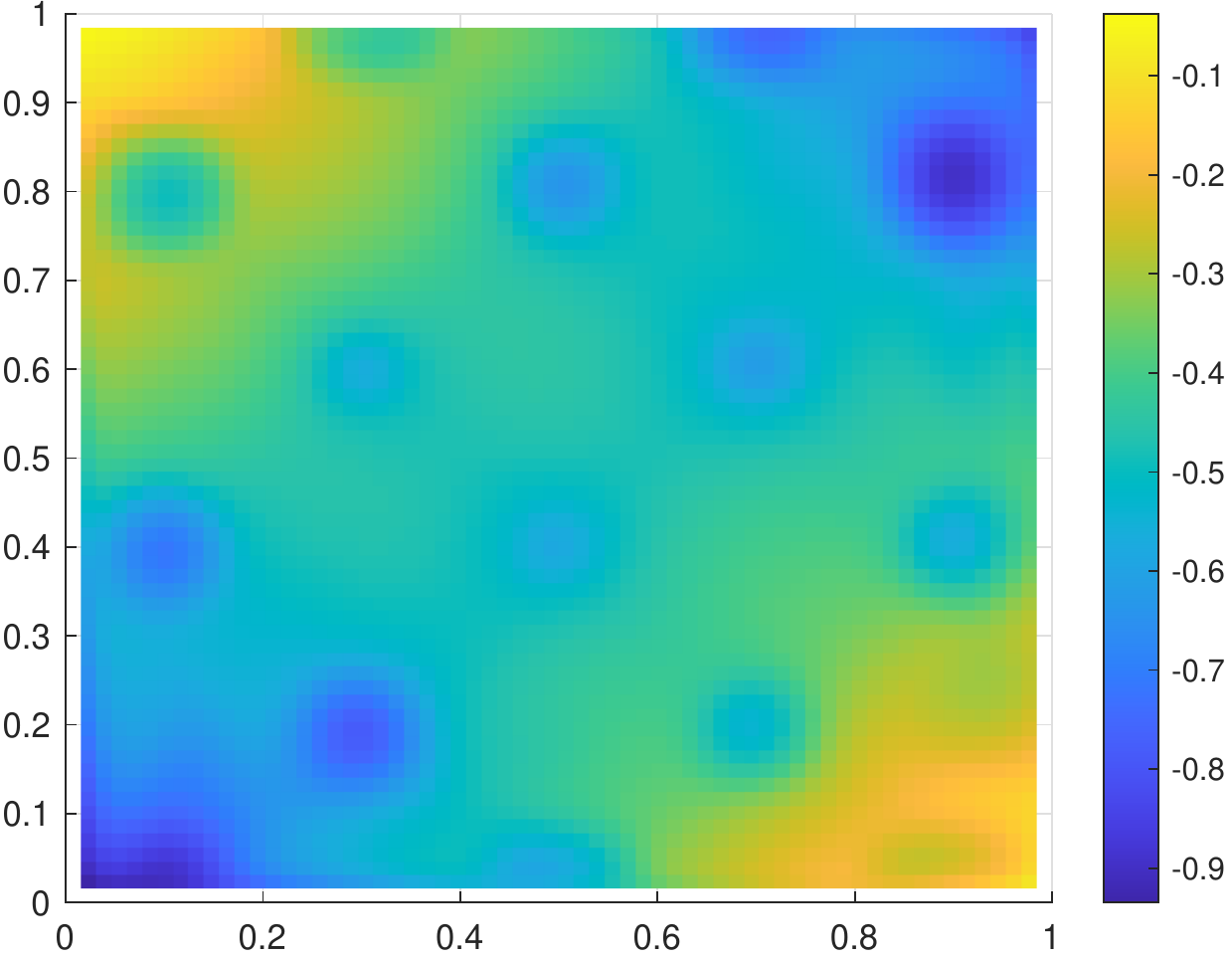}
\includegraphics[width=0.32\linewidth]{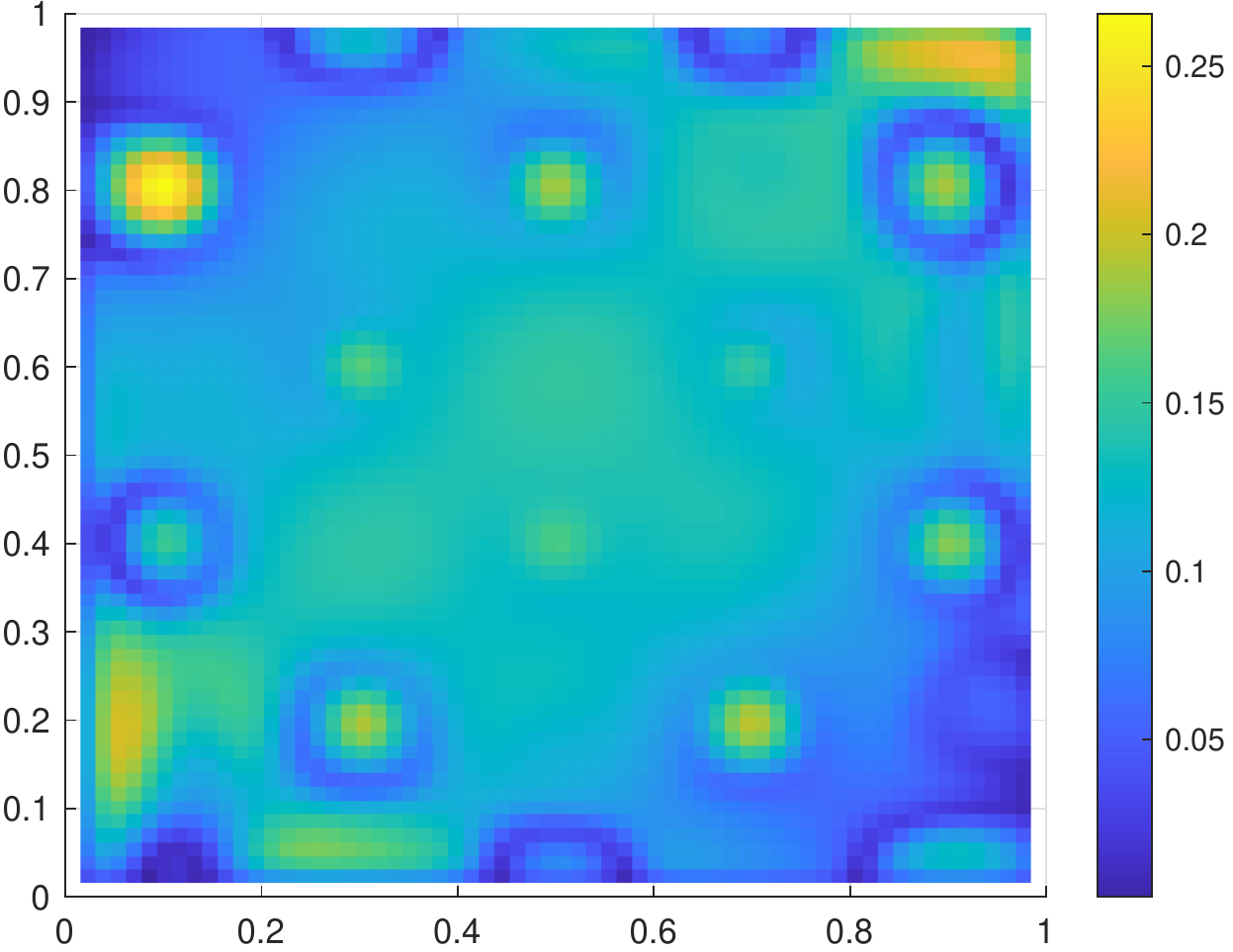}
\caption{Left: control signal $u_{\gamma_*}(x)$. Middle: mean $\mathbb{E}_{\mathbb{P}}[y_{\gamma_*}(x,\xi)]$. Right: standard deviation $\sqrt{\mathbb{E}_{\mathbb{P}}[(y_{\gamma_*}(x,\xi) - \mathbb{E}_{\mathbb{P}}[y_{\gamma_*}(x,\xi)])^2]}$.}
\label{fig:2dsol}
\end{figure}

\subsection{Variational inequality constraints}
\label{s:VI}

In this section we minimize the regularized misfit
\begin{equation}\label{eq:vi-cost}
j(u) = \frac{1}{2}\mathbb{E}_{\mathbb{P}}[\|y(x,\xi) - y_d(x)\|_{L^2(D)}^2] + \frac{1}{2}\|u(x)\|_{L^2(D)}^2
\end{equation}
subject to a random elliptic variational inequality (VI) constraint,
\begin{equation}\label{eq:vi}
y(x,\xi) \le 0: \quad \langle A(\xi) y(x,\xi) - f(x,\xi) - B(x,\xi) u, y(x,\xi) - v \rangle \le 0, \quad \forall v: v \le 0.
\end{equation}
We use Example 5.1 from \cite{Surowiec-vi-2022} (with the reversed sign of $y$), where $D=(0,1)^2$, $A=-\Delta$, $B=\mathrm{Id}$, and deterministic functions constructing the desired state:
\begin{align*}
 \hat y(x) & = \begin{cases}160(x_1^3 - x_1^2 +0.25x_1)(x_2^3 - x_2^2 + 0.25x_2) & \mbox{in } (0,0.5)^2, \\ 0, & \mbox{otherwise},\end{cases} \\
 \hat \zeta(x) & = \max(0,-2|x_1-0.8|-2|x_1x_2-0.3|+0.5), \\
 y_d(x) & = -\hat y - \hat \zeta + \Delta \hat y.
\end{align*}
In contrast, the right hand side depends on the random variables,
\begin{align*}
f(x,\xi(\omega)) & = \Delta \hat y + \hat y + \hat \zeta + b(x,\xi(\omega)), \\
b(x,\xi(\omega)) & = \begin{cases}\sum_{i=1}^{d} \sqrt{\lambda_i} \phi_i(x) \xi_i(\omega), & \mbox{in } (0,0.5) \times (0,1), \\ 0, & \mbox{otherwise}.\end{cases}
\end{align*}
The Karhunen-Loeve expansion in $b(x,\xi)$ is an affine-uniform random field, with $\xi_i(\omega) \sim \mathcal{U}(-1,1)$, $\phi_i(x) = 2 \cos(\pi j x_2) \cos(\pi k x_1)$ and $\lambda_i = \frac{1}{100} \exp(-\frac{\pi}{4}(j^2+k^2))$,
where the pairs $(j,k)$, $j,k=1,2,\ldots,$ are permuted such that $\lambda_1 \ge \lambda_2 \ge \cdots$.

The VI \eqref{eq:vi} is replaced by the penalized problem
\begin{equation}\label{eq:vi-smooth}
A y + \frac{1}{\varepsilon} g_{\varepsilon} (y) = f(x,\xi) + B u,
\end{equation}
so we minimize \eqref{eq:vi-cost} with $y(x,\xi)$ plugged in from \eqref{eq:vi-smooth}.
The latter equation is solved via the Newton method, initialized with $y=0$ as the initial guess, and stopped when the relative difference between two consecutive iterations of $y$ falls below $10^{-12}$.
The problem is discretized in $x$ via the piecewise bilinear finite elements on a uniform $n_y \times n_y$ grid with cell size $h=1/(n_y+1)$.
The homogeneous Dirichlet boundary conditions $y=0$ on $\partial D$ allow us to store only interior grid points.
This gives us a discrete problem of minimizing
\begin{equation}\label{eq:vi-cost-d}
j_h(\mathbf{u}) = \frac{1}{2}\mathbb{E}_{\mathbb{P}}[\|\mathbf{y}(\xi) - \mathbf{y}_d\|_{\mathbf{M}_h}^2] + \frac{1}{2}\|\mathbf{u}\|_{\mathbf{M}_h}^2
\end{equation}
subject to
\begin{equation}\label{eq:vi-smooth-d}
\mathbf{A}_{h} \mathbf{y} + \frac{1}{\varepsilon} g_{\varepsilon} (\mathbf{y}) = \mathbf{f}(\xi) + \mathbf{u},
\end{equation}
where $\mathbf{A}_h, \mathbf{M}_h \in \mathbb{R}^{n_y^2 \times n_y^2}$ are the stiffness and mass matrices, respectively.
The resolution operator of \eqref{eq:vi-smooth-d} is denoted by $\mathbf{S}_h(\xi)$.

The state part of the cost
$$
j_y(\mathbf{u}; \xi) = \frac{1}{2}\|\mathbf{S}_h(\xi) \mathbf{u} - \mathbf{y}_d\|_{\mathbf{M}_h}^2
$$
and its gradient
$$
\nabla_u j_y(\mathbf{u}; \xi) = \mathbf{S}_h^*(\xi) \mathbf{M}_h (\mathbf{S}_h(\xi) \mathbf{u} - \mathbf{y}_d)
$$
are approximated by the TT-Cross (as functions of $\xi$), which allows one to compute the expectation of $\tilde j_y(\mathbf{u}; \xi) \approx j_y(\mathbf{u};\xi)$ and $\nabla_u \tilde j_y(\mathbf{u}; \xi) \approx \nabla_u j_y(\mathbf{u};\xi)$ easily.
The forward model \eqref{eq:vi-smooth-d} is solved at each evaluation of $\xi$ in the TT-Cross.
However, to avoid excessive computations, the Hessian of \eqref{eq:vi-cost-d} is approximated by that anchored at the mean point $\xi=0$:
$$
\nabla_{uu} j_h(\mathbf{u}) \approx \mathbf{\tilde H} := \mathbf{S}_h^*(0) \mathbf{M}_h \mathbf{S}_h'(0) + \mathbf{M}_h.
$$
The Newton system $\mathbf{\tilde H}^{-1} \nabla_u j_h$ is solved iteratively by using the CG method, since the matrix-vector product with $\mathbf{\tilde H}$ requires the solution of only one forward and one adjoint problem,
\begin{equation}\label{eq:vi-grady}
\mathbf{S}_h^* \cdot \mathbf{v} = \mathbf{S}_h' \cdot \mathbf{v} = \left(\mathbf{A}_h + \mathrm{diag}\left(\frac{1}{\varepsilon} g_{\varepsilon}'(\mathbf{y})\right)\right)^{-1} \mathbf{v}, \quad \forall \mathbf{v} \in \mathbb{R}^{n_y^2}.
\end{equation}

In Table~\ref{tab:vi} we vary the dimension of the random variable $d$, the number of quadrature points in each random variable $n_{\xi}$, and the approximation tolerance in the TT-Cross (tol).
The spatial grid size is fixed to $n_y=31$, which is comparable with the resolution in \cite{Surowiec-vi-2022},
and the smoothing parameter $\varepsilon=10^{-6}$.
As a reference solution $\mathbf{u}_*$, we take the control computed with $d=20$, $n_{\xi}=5$ and $\mbox{tol}=10^{-4}$.
We see that the control and the cost can be approximated quite accurately even with a very low order of the polynomial approximation in $\xi$.
It also seems unnecessary to keep $20$ terms in the Karhunen-Loeve expansion.

The computation complexity is dominated by the solutions of the forward and adjoint problems. The article \cite{Surowiec-vi-2022} reports a ``\# PDE solves'' in a path-following stochastic variance reduced gradient method solving \eqref{eq:vi-cost}--\eqref{eq:vi}.
We believe this indicates the number of the complete solutions of the PDE \eqref{eq:vi-smooth-d}.
However, each solution of \eqref{eq:vi-smooth-d} to the increment tolerance $10^{-12}$ requires 23--25 Newton iterations, each of which requires the linear system solution of the form \eqref{eq:vi-grady},
Moreover, the anchored outer Hessian $\mathbf{\tilde H}$ requires two extra linear solves.
Therefore, in Table~\ref{tab:vi}, we show both the number of PDE solutions till convergence, $N_{pde}$, and the number of all linear system solutions $N_{lin}$,
occurred during the optimization of \eqref{eq:vi-cost-d} till the relative increment of $\mathbf{u}$ falls below the TT-Cross tolerance.
In addition, we report the maximal TT ranks of the state cost gradient and the state itself.
Note that assembly of the full state is not needed during the optimization of \eqref{eq:vi-cost-d} -- only certain samples of $\mathbf{y}(\xi)$ are needed in the TT-Cross approximation of $\nabla_u j_h$.
To save the computing time, the TT tensor of the entire state is computed only after the optimization of $\mathbf{u}$ has converged.

\begin{table}[h]
\caption{Cost, error in the control, number of solutions of $n_y^2 \times n_y^2$ linear system as in \eqref{eq:vi-grady}, number of complete forward PDE solutions \eqref{eq:vi-smooth-d}, and the TT ranks of the cost gradient and forward solution.}
\label{tab:vi}
\centering
\begin{tabular}{ccc|cc|cc|cc}
$d$ & $n_{\xi}$ & tol & $j_h(\mathbf{u})$ & $\frac{\|\mathbf{u}-\mathbf{u}_*\|_{\mathbf{M}_h}}{\|\mathbf{u}_*\|_{\mathbf{M}_h}}$ & $N_{lin}$ & $N_{pde}$ & $r(\nabla_u \tilde j_y)$ & $r(\mathbf{\tilde y})$ \\\hline
10 & 5 & $10^{-4}$ & 1.261333069 & 1.1473e-06 & 1070007& 44584 & 85 & 316 \\\hline
20 & 3 & $10^{-3}$ & 1.261333069 & 2.9012e-05 & 46312  & 1976  & 7  & 29 \\
20 & 3 & $10^{-4}$ & 1.261333069 & 4.2713e-06 & 433134 & 18153 & 56 & 183 \\
20 & 5 & $10^{-4}$ & 1.261333069 & ---          & 1840467& 76243 & 102& 402 \\
\end{tabular}
\end{table}

\begin{figure}[h]
\noindent
\includegraphics[width=0.32\linewidth]{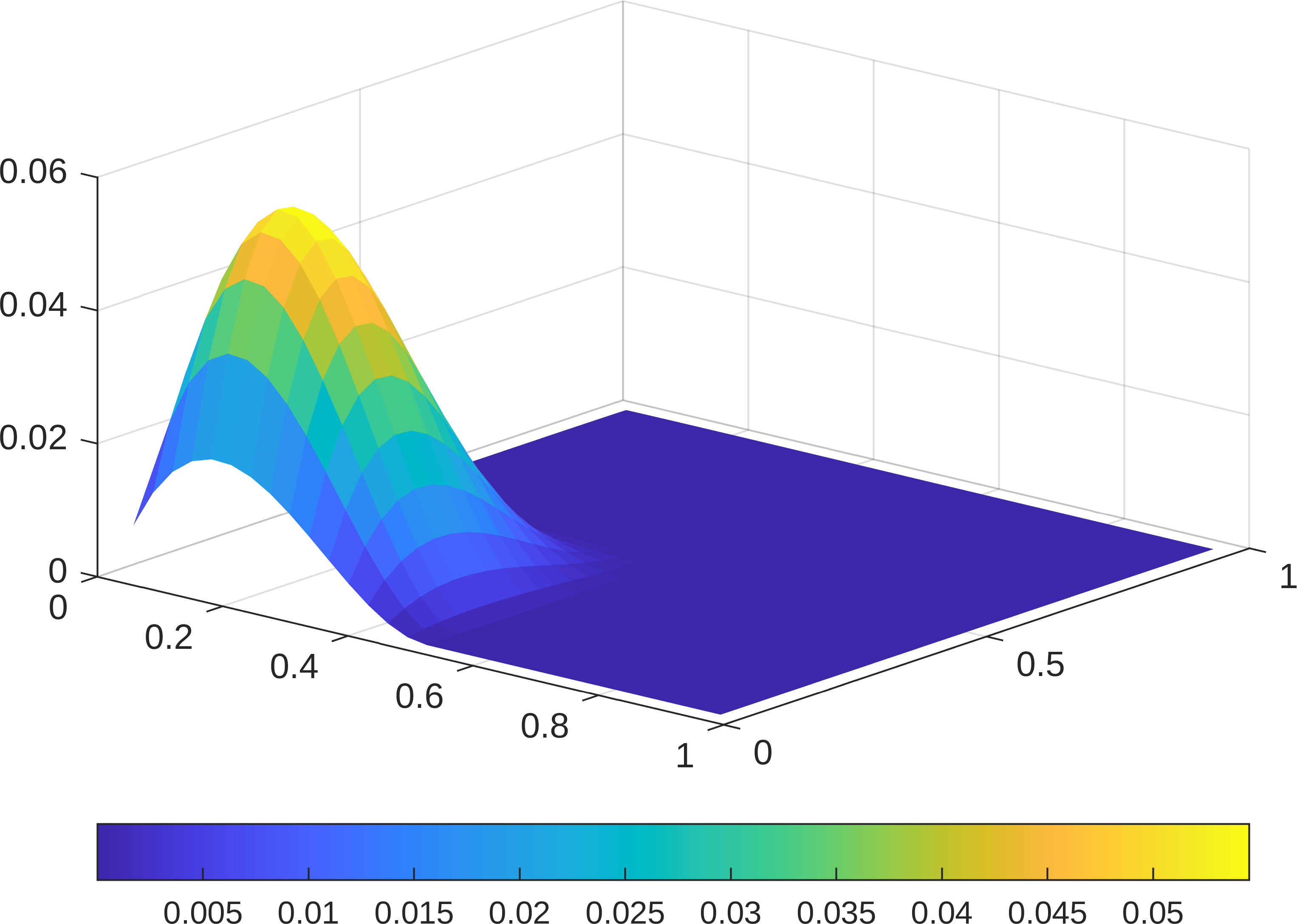}
\includegraphics[width=0.32\linewidth]{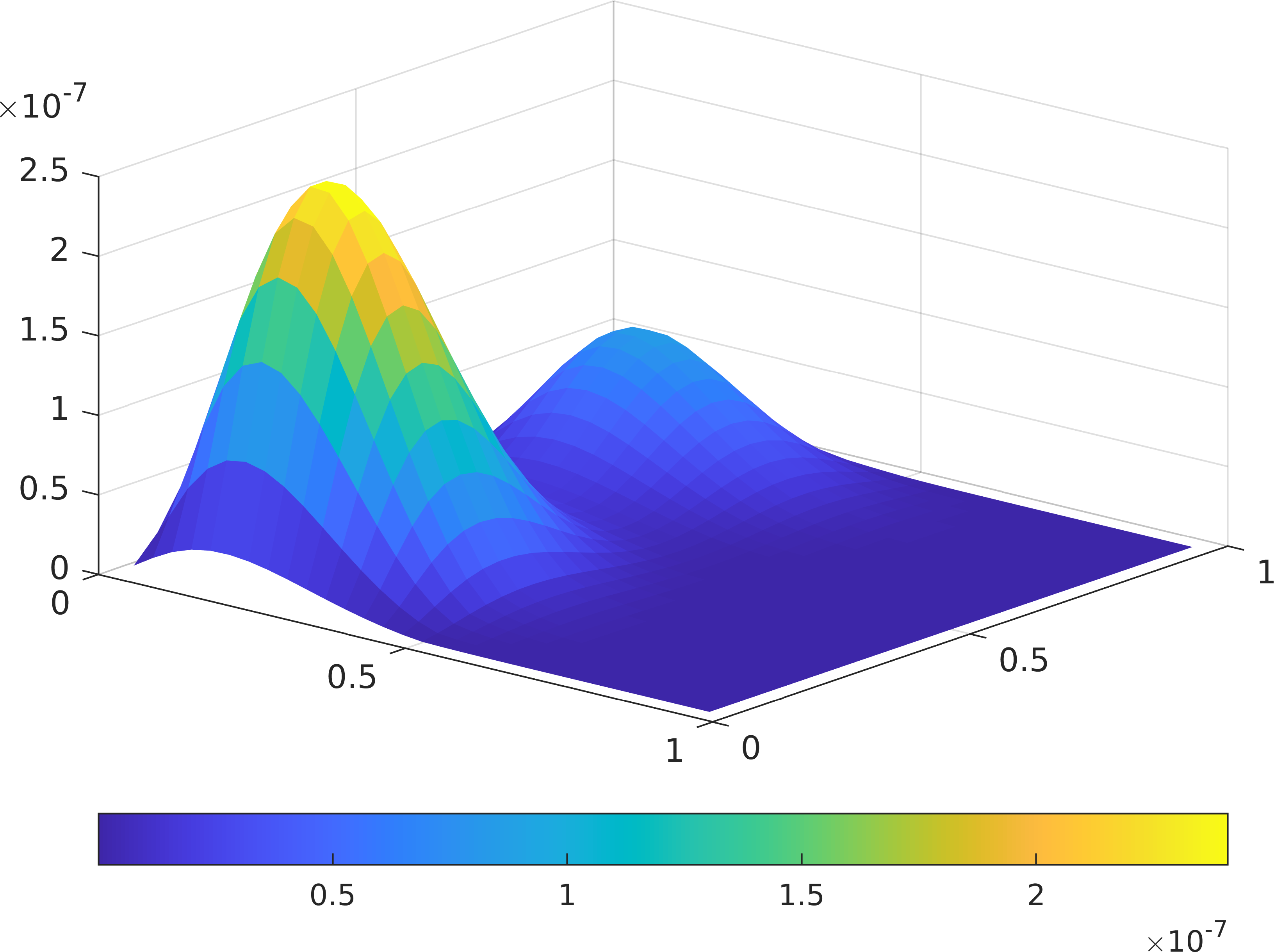}
\includegraphics[width=0.32\linewidth]{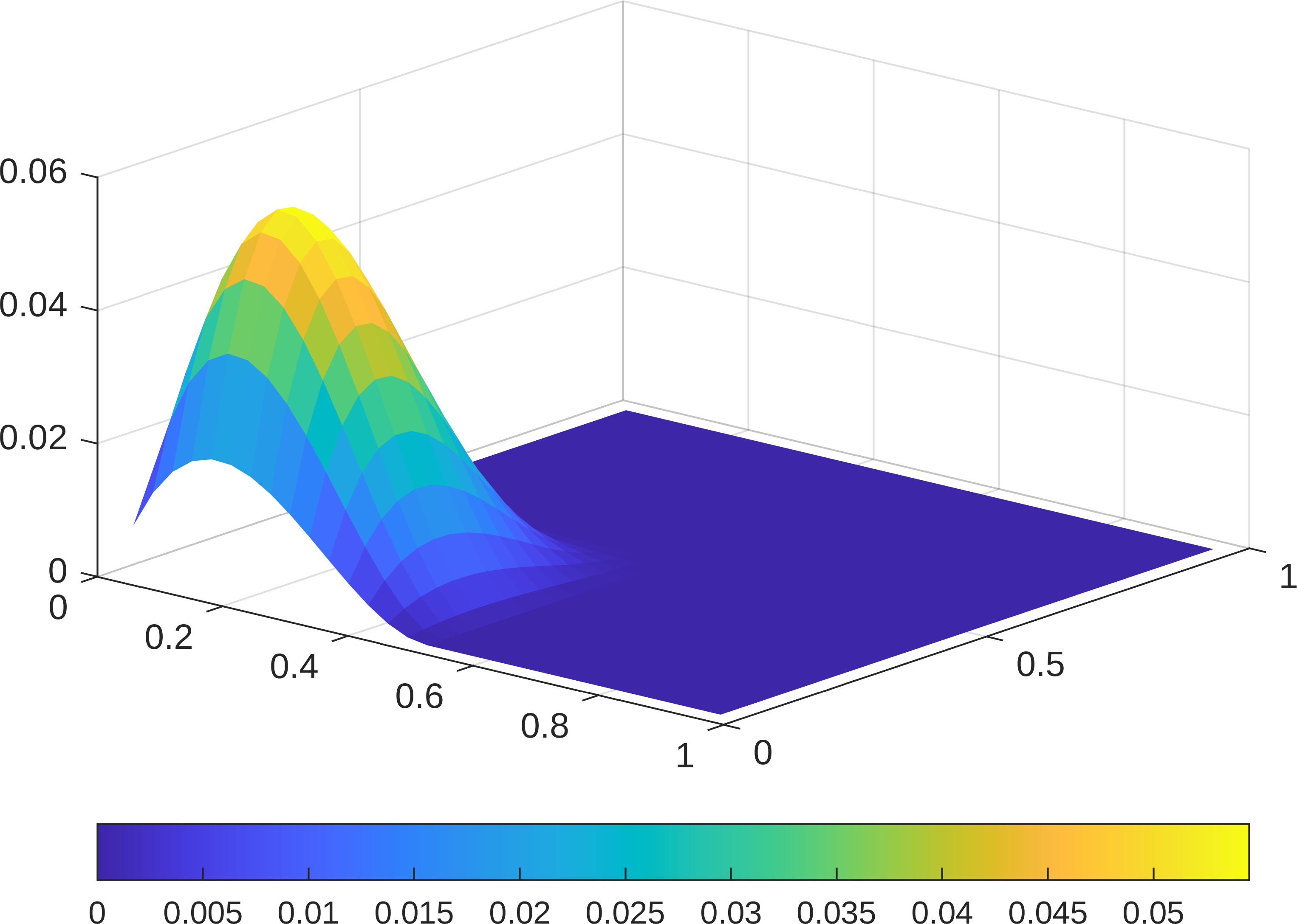} \\
\caption{Left: mean optimised state $\mathbb{E}_{\mathbb{P}}[-y]$ with $d=20,$ $n_{\xi}=3$ and $\mbox{tol}=10^{-3}$. Middle: variance of the optimized state $\mathbb{E}_{\mathbb{P}}[(y - \mathbb{E}_{\mathbb{P}}[y])^2]$. Right: optimised control $u$.
}
\label{fig:vi}
\end{figure}

In Figure~\ref{fig:vi} we show the mean optimized forward state and the control.
The results coincide qualitatively with those in \cite{Surowiec-vi-2022}.
If we consider the computational cost necessary to compute the optimal control only, we can notice that $N_{pde}$ is significantly lower than the 291808 PDE solves in the stochastic variance reduced gradient method of \cite{Surowiec-vi-2022}.

\subsection{SEIR ODE model}
\label{s:SEIR}

Now consider a slightly simplified version of the epidemiological ODE model used for the propagation of COVID-19 in the UK using the data from March-May 2020 \cite{DGKP-SEIR-2021}.
This is a compartmental differential equation model with the following compartments.
\begin{itemize}
 \item Susceptible ($S$).
 \item Exposed ($E$), but not yet infectious.
 \item Infected SubClinical type 1 ($I^{SC1}$): may require hospitalization in the future.
 \item Infected SubClinical type 2 ($I^{SC2}$): will recover without hospitalization.
 \item Infected Clinical type 1 ($I^{C1}$): individuals in the hospital who may decease.
 \item Infected Clinical type 2 ($I^{C2}$): individuals in the hospital who will recover.
 \item Recovered ($R$) and immune to reinfections.
 \item Deceased ($D$).
\end{itemize}
In turn, each of these compartments are split into 5 further sub-compartments corresponding to age bands: 0-19, 20-39, 40-59, 60-79 and 80+.
The number of individuals in each compartment is denoted by the name of the compartment and age band index,
For example, $S_i$ denotes the number of susceptible individuals in the $i$th age band ($i=1,\ldots,5$), $E_i$ denotes the number of exposed individuals in the $i$th age band, and so on.
Variables corresponding to different age bands but same compartment are collected into vectors, $S = (S_1,\ldots,S_5)$, $E=(E_1,\ldots, E_5)$ and so on.

Some of the variables introduced above are coupled to others only one way, and can be removed from the actual simulations.
First, when the number of infected individuals is small compared to the population size (which is typically the case in the early stages of the epidemic),
the relative variation of $S$ is small.
Hence, $S$ can be taken constant instead of solving an ODE on it.
Similarly, none of the variables depend on $R$ and $D$,
so they can be excluded from a coupled system of ODEs too, and computed separately after the solution of the ODEs.
With these considerations in mind, the forward model reads as follows:
\begin{align}\label{eq:SIR}
\frac{d}{dt}\begin{bmatrix}E \\ I^{SC1} \\ I^{SC2} \\ I^{C1} \\ I^{C2}\end{bmatrix}
& - \begin{bmatrix}
     -\kappa \mathtt{I} & A_u & A_u & 0 & 0 \\
     \kappa \cdot \mathrm{diag}(\rho) & -\eta_C \mathtt{I} & 0 & 0 & 0 \\
     \kappa \cdot \mathrm{diag}(1-\rho) & 0 & -\eta_R \mathtt{I} & 0 & 0 \\
     0 & \eta_C \cdot \mathrm{diag}(\rho') & 0 & -\nu \mathtt{I} & 0 \\
     0 & \eta_C \cdot \mathrm{diag}(1-\rho') & 0 & 0 & -\eta_{R,C} \mathtt{I}
    \end{bmatrix}
    \begin{bmatrix}E \\ I^{SC1} \\ I^{SC2} \\ I^{C1} \\ I^{C2}\end{bmatrix} = 0.
\end{align}
Here $\mathtt{I} \in \mathbb{R}^{5 \times 5}$ is the identity matrix and $\mathrm{diag}(\cdot)$ produces a diagonal matrix from a vector.
The control is defined in terms of the intensity of lockdown measures, and affects the susceptible-infected interaction matrix
$A_u = \chi \cdot \mathrm{diag}(S) \cdot C_u \cdot \mathrm{diag}(\frac{1}{N})$,
where
\begin{equation}
 C_u = \mathrm{diag}(c^{home}) C^{home} + \mathrm{diag}(c^{work}_u) C^{work} + \mathrm{diag}(c^{school}_u) C^{school} + \mathrm{diag}(c^{other}_u) C^{other}
\end{equation}
is the matrix of contact intensities between the age compartments.
The total contact intensity is a sum of pre-pandemic contact intensity matrices in the four setting $C^{home},C^{work},C^{school}$ and $C^{other}$,
multiplied by the reduction factors $c^{home},c^{work}_u,c^{school}_u$ and $c^{other}_u$ due to the lockdown measures.
Since home contacts cannot be controlled, $c^{home} = (1,\ldots,1)$,
but the remaining factors vary proportionally to the lockdown control applied from day 17 onwards,
\begin{equation}
c^{\mu}_u(t) = \left\{\begin{array}{ll}(1,1,1,1,1)^\top, & t<17, \\ (c_{123}(1-u^{\mu}(t)) ,c_{123} (1-u^{\mu}(t)),c_{123} (1-u^{\mu}(t)),c_4,c_5)^\top, & 17 \le t \le 90, \\
(c_{123}(1-u^{\mu}(90)) ,c_{123} (1-u^{\mu}(90)),c_{123} (1-u^{\mu}(90)),c_4,c_5)^\top, & t > 90,
\end{array}\right.
\end{equation}
where $\mu \in \{work, school, other\}$, $u^{\mu}$ are the intensities of lockdown measures applied to each setting $\mu$,
and $c_{123},c_4,c_5$ are the initial contact intensities in the corresponding age groups.
Note that the control will be optimized only on the time interval $[17,90]$.
Before day 17 the contact intensities are not reduced (no lockdown).
From day 90 onwards we continue applying the last value of the control.

In addition, the model depends on the following parameters:
\begin{itemize}
 \item $\chi$: probability of $S$--$I^{SC}$ interactions.
 \item $\kappa=1/d_L$: average rate of an Exposed individual becoming SubClinical. It is inversely proportional to the average number of days $d_L$ an individual stays in the Exposed state.
 \item $\eta_C = 1/d_C$: average rate of a SubClinical individual becoming Clinical. Similarly, $d_C$ is the average time spent in the SubClinical state.
 \item $\eta_R = 1/d_R$: rate of recovery from $I^{SC2}$.
 \item $\eta_{R,C} = 1/d_{R,C}$: rate of recovery from $I^{C2}$.
 \item $\nu = 1/d_D$: rate of decease in the $I^{C1}$ state.
 \item $\rho=(\rho_1,\ldots,\rho_5)^\top \in \mathbb{R}^5$: correction coefficients of the Exposed $\rightarrow$ SubClinical 1 transition rate for different age bands.
 \item $\rho'=(\rho'_1,\ldots,\rho'_5)^\top \in \mathbb{R}^5$: correction coefficients of the SubClinical $\rightarrow$ Clinical 1 transition.
 \item $N = (N_1,\ldots,N_5)^\top \in \mathbb{R}^5$: total number of individuals in each age group.
 \item $N^0$: total number of infected individuals on day 0.
 \item $N^{in} = (0.1, 0.4, 0.35, 0.1, 0.05)^\top N^{0}$: age partition of the initial number of infected individuals.
\end{itemize}
The ODE \eqref{eq:SIR} is initialized by setting
$$
E(0) = \frac{N^{in}}{3}, \quad I^{SC1}(0) = \frac{2}{3} \mathrm{diag}(\rho) N^{in}, \quad I^{SC2}(0) = \frac{2}{3} \mathrm{diag}(1-\rho) N^{in}, \quad I^{C1}(0) = I^{C2}(0)=0.
$$
The population sizes $S=N$ are taken from the Office for National Statistics, mid 2018 estimate.

However, none of the model parameters above are known beforehand.
In \cite{DGKP-SEIR-2021}, those were treated as random variables,
and their distributions were estimated from observed numbers of infections and hospitalizations during the first $90$ days using Approximate Bayesian Computation (ABC).
In general, these variables are correlated through the posterior distribution, sampling from which is a daunting problem.
Here, we replace the joint ABC posterior distribution by independent uniform distributions with a scaled posterior standard deviation centered around the posterior mean:
\begin{align}\label{eq:SIR-params}
\chi          & \sim \mathcal{U}(0.13- 0.03\sigma, 0.13+ 0.03\sigma), & d_L           & \sim \mathcal{U}(1.57- 0.42\sigma, 1.57+ 0.42\sigma), \\\nonumber
d_C           & \sim \mathcal{U}(2.12 -0.80\sigma, 2.12 +0.80\sigma), & d_R           & \sim \mathcal{U}(1.54 -0.40\sigma, 1.54 +0.40\sigma), \\\nonumber
d_{R,C}       & \sim \mathcal{U}(12.08-1.51\sigma, 12.08+1.51\sigma), & d_D           & \sim \mathcal{U}(5.54 -2.19\sigma, 5.54 +2.19\sigma), \\\nonumber
\rho_1        & \sim \mathcal{U}(0.06 -0.03\sigma, 0.06 +0.03\sigma), & \rho_2        & \sim \mathcal{U}(0.05 -0.03\sigma, 0.05 +0.03\sigma), \\\nonumber
\rho_3        & \sim \mathcal{U}(0.08 -0.04\sigma, 0.08 +0.04\sigma), & \rho_4        & \sim \mathcal{U}(0.54 -0.22\sigma, 0.54 +0.22\sigma), \\\nonumber
\rho_5        & \sim \mathcal{U}(0.79 -0.14\sigma, 0.79 +0.14\sigma), & \rho'_1       & \sim \mathcal{U}(0.26 -0.23\sigma, 0.26 +0.23\sigma), \\\nonumber
\rho'_2       & \sim \mathcal{U}(0.28 -0.25\sigma, 0.28 +0.25\sigma), & \rho'_3       & \sim \mathcal{U}(0.33 -0.27\sigma, 0.33 +0.27\sigma), \\\nonumber
\rho'_4       & \sim \mathcal{U}(0.26 -0.11\sigma, 0.26 +0.11\sigma), & \rho'_5       & \sim \mathcal{U}(0.80 -0.13\sigma, 0.80 +0.13\sigma), \\\nonumber
N^{0}         & \sim \mathcal{U}(276  - 133\sigma, 276  +133 \sigma), & c_{123}  & \sim \mathcal{U}(0.63 -0.21\sigma, 0.63 +0.21\sigma), \\\nonumber
c_4      & \sim \mathcal{U}(0.57 -0.23\sigma, 0.57 +0.23\sigma), & c_5      & \sim \mathcal{U}(0.71 -0.23\sigma, 0.71 +0.23\sigma).   \nonumber
\end{align}
Here, $\sigma$ is the standard deviation scaling parameter, taken to be $0.03$ in our experiment.
This distribution behaves qualitatively similar to the posterior distribution in the vicinity of the posterior mean.
It provides sufficient randomness to benchmark the constrained optimization method, while admitting independent sampling and gridding, needed for the TT approximations.
That is, \eqref{eq:SIR-params} form a random vector
$$
\xi=(\chi,d_L,d_C,d_R,d_{R,C},d_D,\rho_1,\rho_2,\rho_3,\rho_4,\rho_5,\rho'_1,\rho'_2,\rho'_3,\rho'_4,\rho'_5,N^{0},c_{123},c_4,c_5)
$$
of $d=20$ independent random variables,
the state vector is
$$
y(t,\xi)=(E_1,\ldots,E_5,~I^{SC1}_1,\ldots,I^{SC1}_5,~I^{SC2}_1,\ldots,I^{SC2}_5,~I^{C1}_1,\ldots,I^{C1}_5,~I^{C2}_1,\ldots,I^{C2}_5),
$$
and
the ODE~\eqref{eq:SIR} constitutes the forward problem.

For the inverse problem, we use the total number of deceased patients as the cost function.
The rate of decease is proportional to the number of Clinical type 1 individuals, so the total number of deceased individuals can be computed as
\begin{equation}\label{eq:deaths}
D(t,\xi) = \nu \int_{0}^{t} I^{C1}(s,\xi) ds.
\end{equation}
To regularize the problem,
we add also
the norm of the control $u(t) = (u^{work}(t), u^{school}(t), u^{other}(t))$.
Thus, the total cost function reads
\begin{equation}\label{eq:cost-det}
 j(u) = \frac{1}{2}\mathbb{E}_{\mathbb{P}}[D(T,\xi)] + \frac{\alpha}{2} \int_{17}^{90} \|u(t)\|_2^2 dt,
\end{equation}
where $T=100$ is the final simulation time, and $\alpha$ is the regularization parameter, which we set to $100$ in our experiment.
Note that the norm of the control is taken only over the time interval $[17,90]$ where the control varies.

We introduce the following constraints.
Firstly, we limit the control components to the intervals
$u^{work} \in [0, 0.69]$, $u^{school} \in [0, 0.9]$ and $u^{other} \in [0, 0.59]$.
Next, we constrain the $\mathcal{R}$ number at the end of the variable control interval, $\mathcal{R}(90,\xi) \le 1$.
In our model, the $\mathcal{R}$ number can be computed as $\mathcal{R}(t,\xi) = \lambda_{\max}(K),$ where
$$
K =
-\begin{bmatrix}
0 & A_u & A_u & 0 & 0 \\
0 & 0 & 0 & 0 & 0 \\
0 & 0 & 0 & 0 & 0 \\
0 & 0 & 0 & 0 & 0 \\
0 & 0 & 0 & 0 & 0 \\
\end{bmatrix}
\begin{bmatrix}
     -\kappa \mathtt{I} & 0 & 0 & 0 & 0 \\
     \kappa \cdot \mathrm{diag}(\rho) & -\eta_C \mathtt{I} & 0 & 0 & 0 \\
     \kappa \cdot \mathrm{diag}(1-\rho) & 0 & -\eta_R \mathtt{I} & 0 & 0 \\
     0 & \eta_C \cdot \mathrm{diag}(\rho') & 0 & -\nu \mathtt{I} & 0 \\
     0 & \eta_C \cdot \mathrm{diag}(1-\rho') & 0 & 0 & -\eta_{R,C} \mathtt{I}
    \end{bmatrix}^{-1},
$$
and $\lambda_{\max}$ denotes the maximal in modulus eigenvalue.
Recall that $\mathcal{R}<1$ implies that the epidemic decays, while $\mathcal{R}>1$ corresponds to an expanding epidemic.
The full smoothed Moreau-Yosida cost function becomes
\begin{equation}\label{eq:cost-MY}
 j_{\gamma}(u) = \frac{1}{2}\mathbb{E}_{\mathbb{P}}[D(T,\xi)] + \frac{\alpha}{2} \int_{17}^{90} \|u(t)\|_2^2 dt + \frac{\gamma}{2} \mathbb{E}_{\mathbb{P}}\left[\left|g_{\varepsilon_{\gamma}}(\mathcal{R}(90,\xi)-1)\right|^2\right].
\end{equation}

Since the control is applied nonlinearly in the model, computation of derivatives of the cost function \eqref{eq:cost-MY} is complicated.
Thus, instead of the Newton method, we use the projected gradient descent method,
where the gradient of \eqref{eq:cost-MY} is calculated using finite differencing with anisotropic step sizes $10^{-6} \cdot \max(|u|, 0.1)$.
The ODE \eqref{eq:SIR} is solved using an implicit Euler method with a time step $0.1$.
In this experiment, we use a fixed
Moreau-Yosida parameter $\gamma=5 \cdot 10^5$ in all iterations,
and the smoothing width is chosen as $\varepsilon_{\gamma} = 50/\sqrt{\gamma}$.
The iteration is stopped when the cost value does not decrease in two consecutive iterations.
Each random variable \eqref{eq:SIR-params} is discretized with $n=3$ Gauss-Legendre quadrature nodes,
and the TT approximations are carried out with a relative error tolerance of $10^{-2}$.
The control $u(t)$ is discretized using $7$ Gauss-Legendre nodes on $[17, 90]$ with a Lagrangian interpolation in between.

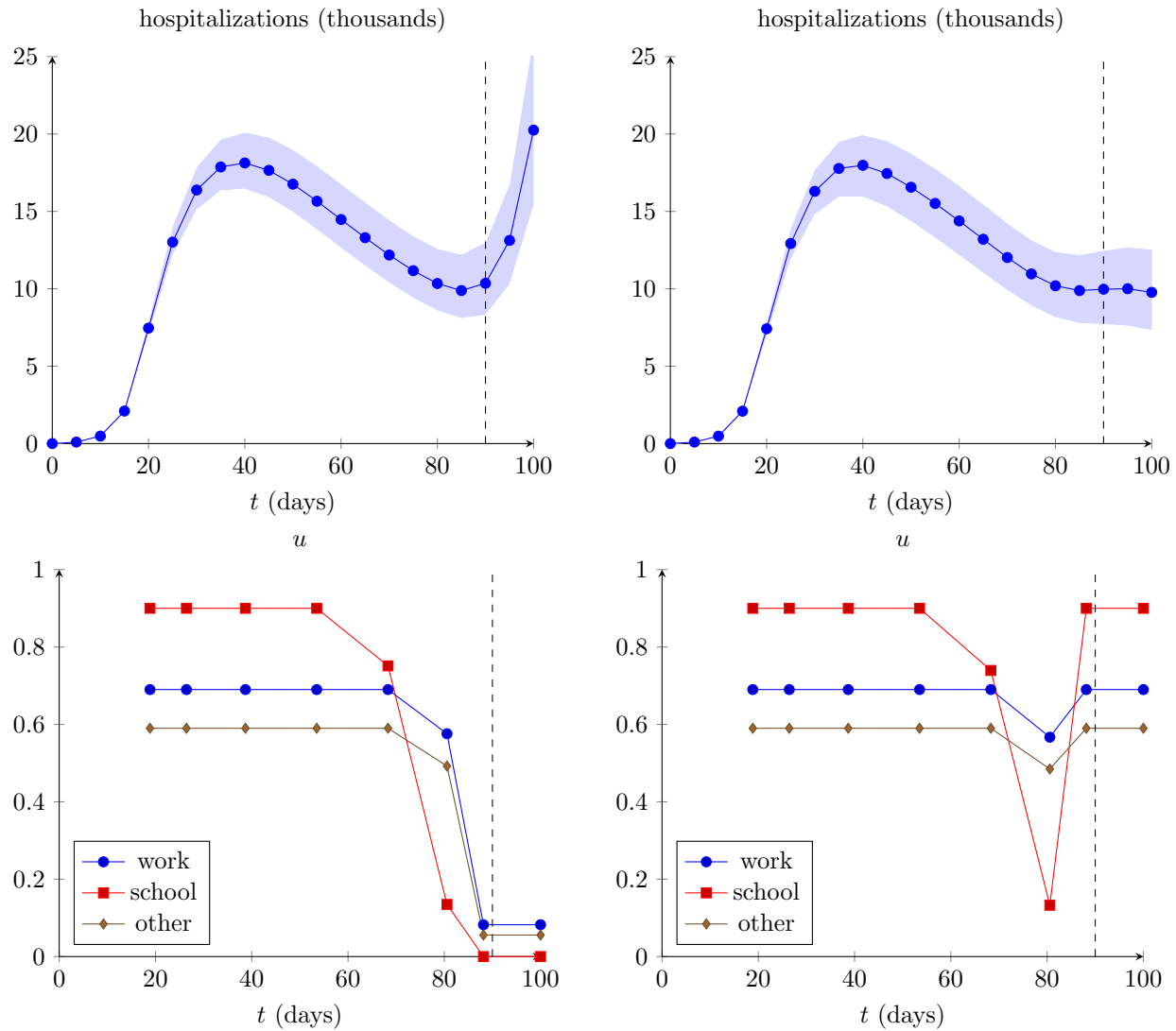
\begin{figure}[htb]
\noindent
\begin{tikzpicture}
 \begin{axis}[%
 width=0.41\linewidth,
 height=0.33\linewidth,
 xlabel=$t$ (days),
 title={hospitalizations (thousands)},
 xmin=0,xmax=100,
 ymin=0,ymax=25,
 ]
 \addplot+[no marks,color=white,name path=minus] coordinates{
   (  0,         0)
   (  5,    0.0964)
   ( 10,    0.4684)
   ( 15,    2.0011)
   ( 20,    7.0146)
   ( 25,   12.1383)
   ( 30,   15.0915)
   ( 35,   16.3239)
   ( 40,   16.4396)
   ( 45,   15.8858)
   ( 50,   14.9448)
   ( 55,   13.8104)
   ( 60,   12.6253)
   ( 65,   11.4643)
   ( 70,   10.3737)
   ( 75,    9.3911)
   ( 80,    8.5775)
   ( 85,    8.0760)
   ( 90,    8.2952)
   ( 95,   10.2422)
   (100,   15.3517)
};
 \addplot+[no marks,color=white,name path=plus] coordinates{
   (  0,         0)
   (  5,    0.1003)
   ( 10,    0.4993)
   ( 15,    2.2159)
   ( 20,    8.0076)
   ( 25,   14.0794)
   ( 30,   17.8973)
   ( 35,   19.6667)
   ( 40,   20.1237)
   ( 45,   19.8022)
   ( 50,   19.0070)
   ( 55,   17.9566)
   ( 60,   16.7939)
   ( 65,   15.6115)
   ( 70,   14.4710)
   ( 75,   13.4322)
   ( 80,   12.6068)
   ( 85,   12.2286)
   ( 90,   13.0183)
   ( 95,   16.8137)
   (100,   26.4233)
};

\addplot[blue!16!white] fill between[of = minus and plus];

 \addplot+[blue,mark=*] coordinates{
   (  0,         0)
   (  5,    0.0982)
   ( 10,    0.4840)
   ( 15,    2.1021)
   ( 20,    7.4616)
   ( 25,   13.0157)
   ( 30,   16.3806)
   ( 35,   17.8693)
   ( 40,   18.1287)
   ( 45,   17.6519)
   ( 50,   16.7599)
   ( 55,   15.6579)
   ( 60,   14.4773)
   ( 65,   13.3014)
   ( 70,   12.1833)
   ( 75,   11.1700)
   ( 80,   10.3409)
   ( 85,    9.8865)
   ( 90,   10.3549)
   ( 95,   13.1258)
   (100,   20.2519)
}; 
\addplot[dashed,black] coordinates{(90,0) (90,25)};
\end{axis}
\end{tikzpicture}
\hfill
\begin{tikzpicture}
 \begin{axis}[%
 width=0.41\linewidth,
 height=0.33\linewidth,
 xlabel=$t$ (days),
 title={hospitalizations (thousands)},
 xmin=0,xmax=100,
 ymin=0,ymax=25,
 ]
 \addplot+[no marks,color=white,name path=minus] coordinates{
   (  0,         0)
   (  5,    0.0963)
   ( 10,    0.4665)
   ( 15,    1.9832)
   ( 20,    6.8923)
   ( 25,   11.8665)
   ( 30,   14.7810)
   ( 35,   15.9285)
   ( 40,   15.9328)
   ( 45,   15.2913)
   ( 50,   14.3456)
   ( 55,   13.2777)
   ( 60,   12.1550)
   ( 65,   11.0053)
   ( 70,    9.8851)
   ( 75,    8.8884)
   ( 80,    8.1426)
   ( 85,    7.7640)
   ( 90,    7.6969)
   ( 95,    7.5938)
   (100,    7.2956)
};
 \addplot+[no marks,color=white,name path=plus] coordinates{
   (  0,         0)
   (  5,    0.1000)
   ( 10,    0.4978)
   ( 15,    2.1963)
   ( 20,    7.9218)
   ( 25,   13.9442)
   ( 30,   17.7215)
   ( 35,   19.5188)
   ( 40,   19.9580)
   ( 45,   19.5710)
   ( 50,   18.7605)
   ( 55,   17.7860)
   ( 60,   16.6774)
   ( 65,   15.4791)
   ( 70,   14.2577)
   ( 75,   13.1648)
   ( 80,   12.4066)
   ( 85,   12.1982)
   ( 90,   12.4841)
   ( 95,   12.7015)
   (100,   12.5635)
};

\addplot[blue!16!white] fill between[of = minus and plus];

 \addplot+[blue,mark=*] coordinates{
   (  0,          0)
   (  5,     0.0982)
   ( 10,     0.4829)
   ( 15,     2.0945)
   ( 20,     7.4175)
   ( 25,    12.9262)
   ( 30,    16.2965)
   ( 35,    17.7743)
   ( 40,    17.9767)
   ( 45,    17.4476)
   ( 50,    16.5607)
   ( 55,    15.5148)
   ( 60,    14.3843)
   ( 65,    13.1996)
   ( 70,    12.0194)
   ( 75,    10.9621)
   ( 80,    10.1959)
   ( 85,     9.8848)
   ( 90,     9.9722)
   ( 95,    10.0044)
   (100,     9.7702)
}; 
\addplot[dashed,black] coordinates{(90,0) (90,25)};
\end{axis}
\end{tikzpicture} \\
\noindent
\begin{tikzpicture}
 \begin{axis}[%
 width=0.41\linewidth,
 height=0.33\linewidth,
 xlabel=$t$ (days),
 title={$u$},
 xmin=0,xmax=100,
 ymin=0,ymax=1,
 legend style={at={(0.03,0.03)},anchor=south west},
 ]
 \addplot+[] coordinates{
  ( 18.8576,    0.6900)
  ( 26.4341,    0.6900)
  ( 38.6867,    0.6900)
  ( 53.5000,    0.6900)
  ( 68.3133,    0.6900)
  ( 80.5659,    0.5759)
  ( 88.1424,    0.0823)
  (100.0000,    0.0823)
 }; \addlegendentry{work};
 \addplot+[] coordinates{
  ( 18.8576,    0.9000)
  ( 26.4341,    0.9000)
  ( 38.6867,    0.9000)
  ( 53.5000,    0.9000)
  ( 68.3133,    0.7512)
  ( 80.5659,    0.1348)
  ( 88.1424,         0)
  (100.0000,         0)
 }; \addlegendentry{school};
 \addplot+[mark=diamond*] coordinates{
  ( 18.8576,    0.5900)
  ( 26.4341,    0.5900)
  ( 38.6867,    0.5900)
  ( 53.5000,    0.5900)
  ( 68.3133,    0.5900)
  ( 80.5659,    0.4924)
  ( 88.1424,    0.0557)
  (100.0000,    0.0557)
 }; \addlegendentry{other};
 \addplot[dashed,black] coordinates{(90,0) (90,1)};
 \end{axis}
\end{tikzpicture}
\hfill
\begin{tikzpicture}
 \begin{axis}[%
 width=0.41\linewidth,
 height=0.33\linewidth,
 xlabel=$t$ (days),
 title={$u$},
 xmin=0,xmax=100,
 ymin=0,ymax=1,
 legend style={at={(0.03,0.03)},anchor=south west},
 ]
 \addplot+[] coordinates{
  ( 18.8576,      0.6900)
  ( 26.4341,      0.6900)
  ( 38.6867,      0.6900)
  ( 53.5000,      0.6900)
  ( 68.3133,      0.6900)
  ( 80.5659,      0.5669)
  ( 88.1424,      0.6900)
  (100.0000,      0.6900)
 }; \addlegendentry{work};
 \addplot+[] coordinates{
  ( 18.8576,    0.9000)
  ( 26.4341,    0.9000)
  ( 38.6867,    0.9000)
  ( 53.5000,    0.9000)
  ( 68.3133,    0.7394)
  ( 80.5659,    0.1327)
  ( 88.1424,    0.9000)
  (100.0000,    0.9000)
 }; \addlegendentry{school};
 \addplot+[mark=diamond*] coordinates{
  ( 18.8576,    0.5900)
  ( 26.4341,    0.5900)
  ( 38.6867,    0.5900)
  ( 53.5000,    0.5900)
  ( 68.3133,    0.5900)
  ( 80.5659,    0.4847)
  ( 88.1424,    0.5900)
  (100.0000,    0.5900)
 }; \addlegendentry{other};
 \addplot[dashed,black] coordinates{(90,0) (90,1)};
 \end{axis}
\end{tikzpicture}
\caption{Top: optimized $I^C = I^{C1}+I^{C2}$, mean (blue circles) and 95\% confidence interval (shaded area). Bottom: optimized control signals. Left: unconstrained optimization, Right: optimization constrained with $\mathcal{R}(90,\xi) \le 1$ a.s. approximated with $\gamma=5\cdot 10^5$.
Black dashed lines indicate the end of the optimization time horizon $t=90$.
}
\label{fig:sir}
\end{figure}

In Figure~\ref{fig:sir}, we compare optimizations without constraining $\mathcal{R}(90,\xi)$ (left), and with the a.s. constraint (right) as described above.
We plot the time evolution of the mean and confidence interval of the total number of hospitalized individuals, $I^C(t) = I^{C1}(t) + I^{C2}(t)$.
The unconstrained scenario is a finite horizon optimization problem,
which drives the control to near zero values at the end of the controllable time interval, $t=90$,
due to the zero terminal condition on the adjoint state.
Naturally, this leads to infection growing again for $t>90$, since we extrapolate these small values of the control from $t=90$ onwards.

In contrast, if we constrain the $\mathcal{R}$ number at the end of the optimization interval to be below $1$ almost surely,
this drives the control to higher values again.
If we extrapolate these control values beyond the optimization window,
the epidemic continues decaying, albeit with a slightly larger uncertainty.
This indicates that almost sure constraints can suggest a more resilient control in risk-critical applications.

\bibliographystyle{siamplain}
\bibliography{refs}

\appendix
\section{Remainder of Proof to Lemma~3.9}\label{app:measure_proof}

The term \textrm{II} can be expanded as follows:
\begin{align*}
(\mu^{\gamma}, \mathbf{A}^{-1} \mathbf{B} (u^{\gamma} - u^{\dagger})) & = \left(T^*(T y^{\gamma} + T y_f - y_d) + \mu^{\gamma} - T^*(T y^{\gamma} + T y_f - y_d), \mathbf{A}^{-1} \mathbf{B} (u^{\gamma} - u^{\dagger})\right) \\
& = \left(\mathbf{A}^* \lambda^{\gamma}  - T^* T y^{\gamma} - T^* T y_f + T^* y_d, \mathbf{A}^{-1} \mathbf{B} (u^{\gamma} - u^{\dagger})\right) \\
& = \mathbb{E}_{\mathbb{P}}\left[\left(\mathbf{B}^* \lambda^{\gamma} - \mathbf{B}^* \mathbf{A}^{-*}(T^*T y^{\gamma} + T^*T y_f - T^* y_d), u^{\gamma} - u^{\dagger}\right)_{L^2(D)}\right] \, .
\end{align*}
Applying \cite[Thm.~3.7.12]{hille-FuncAn-1974} to $\mathbf{B}^* \lambda^{\gamma}$ thanks to its uniform integrability yields
\begin{align*}
(\mu^{\gamma}, \mathbf{A}^{-1} \mathbf{B} (u^{\gamma} - u^{\dagger})) & = \left(\mathbb{E}_{\mathbb{P}}\left[\mathbf{B}^* \lambda^{\gamma}\right], u^{\gamma} - u^{\dagger}\right)_{L^2(D)} \\
& - \mathbb{E}_{\mathbb{P}}\left[\left(\mathbf{B}^* \mathbf{A}^{-*}(T^*T y^{\gamma} + T^*T y_f - T^* y_d), u^{\gamma} - u^{\dagger}\right)_{L^2(D)}\right].
\end{align*}
Using $\alpha u^{\gamma} + \mathbb{E}_{\mathbb{P}}\left[\mathbf{B}^* \lambda^{\gamma}\right] + \eta^{\gamma} = 0$, and that $u^{\dagger} \in \mathcal{U}_{ad}$ (as a result, $(-\eta^\gamma, u^\gamma-u^\dagger) \le 0$),
\begin{align*}
(\mu^{\gamma}, \mathbf{A}^{-1} \mathbf{B} (u^{\gamma} - u^{\dagger})) & =
\left(-\alpha u^{\gamma} - \eta^{\gamma}, u^{\gamma} - u^{\dagger}\right)_{L^2(D)} \\
& - \mathbb{E}_{\mathbb{P}}\left[\left(\mathbf{B}^* \mathbf{A}^{-*}(T^*T y^{\gamma} + T^*T y_f - T^* y_d), u^{\gamma} - u^{\dagger}\right)_{L^2(D)}\right] \\
& \le \left(-\alpha u^{\gamma} , u^{\gamma} - u^{\dagger}\right)_{L^2(D)} \\
& - \mathbb{E}_{\mathbb{P}}\left[\left(\mathbf{B}^* \mathbf{A}^{-*}(T^*T y^{\gamma} + T^*T y_f - T^* y_d), u^{\gamma} - u^{\dagger}\right)_{L^2(D)}\right] \\
& \le c_a < \infty,
\end{align*}
since
$$
y(u) \mapsto \mathbb{E}_{\mathbb{P}}\left[\frac{1}{2}\|T y(u) + T y_f - y_d\|_{L^2(D)}^2\right] \quad \text{and} \quad u \mapsto \frac{\alpha}{2}\|u\|^2_{L^2(D)}
$$
are continuously differentiable with uniformly bounded gradients on $\mathcal{U}_{ad}$.

Now since the $L^1$-norm is positively homogeneous, subadditive and continuous, we can use the Fenchel-Moreau theorem to express
\begin{align*}
\|\mu^{\gamma}\|_{L^1_{\pi}(\Xi)} & = \frac{1}{\delta} \sup\left\{\langle \mu^{\gamma}, v\rangle_{L^1_{\pi}(\Xi) \times L^{\infty}_{\pi}(\Xi)} : v \in \mathbb{B}_{\delta}(0)\right\} \\
& = \frac{1}{\delta} \sup \left\{(\mu^{\gamma}, v)_{L^2_{\pi}(\Xi) } : v \in \mathbb{B}_{\delta}(0)\right\} \le \frac{1}{\delta} c_0 < \infty.
\end{align*}
Thus, the sequence $\mu^{\gamma}$ is bounded in $L^1_{\pi}(\Xi)$, so we can extract a subsequence $\mu^{\gamma_k}$ which is weak* convergent to some regular countably additive Borel measure $\rho \in \mathcal{M}(\overline{\Xi})$ \cite[Cor.~2.4.3]{HAttouch_GButtazzo_GMichaille_2006}.

\section{Tensor-Train decomposition and approximation}\label{app:tt}
In this section, we describe the Tensor-Train (TT) decomposition as a function approximation technique that allows fast computation of the expectation.
The original TT decomposition \cite{osel-tt-2011} was proposed for tensors (such as tensors of expansion coefficients), and the functional TT (FTT) decomposition~\cite{Marzouk-stt-2016,Gorodetsky-ctt-2019} has extended this idea to multivariate functions.

Let us introduce a basis $\{\ell_i(\xi_k)\}_{i=1}^{n_{\xi}}$ in each random variable  $\xi_k$, $k=1,\ldots,d$, and a quadrature with nodes $Z = \{z_j\}$ and weights $\{w_j\}$ which is exact on this basis,
$$
\mathbb{E}_{\mathbb{P}}[\ell_i] = \sum_{j=1}^{n_{\xi}} w_j \ell_i(z_j).
$$
For example, we can take Lagrange interpolation polynomials built upon a Gaussian quadrature, or orthogonal polynomials up to degree $n_{\xi}-1$ together with the roots of the degree-$n_{\xi}$ polynomial, or Fourier modes and the rectangular quadrature with the number of nodes corresponding to the highest frequency.
Then we can approximate any random field $y\in\mathcal{Y}$ in the tensor product basis,
$$
y(\xi) \approx \sum_{i_1=1}^{n_{\xi}} \cdots \sum_{i_d=1}^{n_{\xi}} \mathbf{Y}_{i_1,\ldots,i_d} \ell_{i_1}(\xi_1) \cdots \ell_{i_d}(\xi_d).
$$
Note that the expansion coefficients $\mathbf{Y}$ form a tensor of $n_{\xi}^d$ entries, which is impossible to store directly if $d$ is large.
The TT decomposition aims to factorize this tensor further to a product of tensors of manageable size.
\begin{definition}
 A tensor $\mathbf{Y}\in \mathbb{R}^{n_{\xi} \times \cdots \times n_{\xi}}$ is said to be approximated by the \emph{TT decomposition} with a relative approximation error $\epsilon$ if there exist 3-dimensional tensors $\mathbf{Y}^{(k)} \in \mathbb{R}^{r_{k-1} \times n_{\xi} \times r_k}$, $k=1,\ldots,d$, such that
 \begin{equation}\label{eq:tt}
  \mathbf{\tilde Y}_{i_1,\ldots,i_d} := \sum_{s_0,\ldots,s_d=1}^{r_0,\ldots,r_d} \mathbf{Y}^{{(1)}}_{s_0,i_1,s_1} \mathbf{Y}^{{(2)}}_{s_1,i_2,s_2} \cdots \mathbf{Y}^{{(d)}}_{s_{d-1},i_d,s_d},
 \end{equation}
 and $\|\mathbf{Y} - \mathbf{\tilde Y}\|_{F} = \epsilon \|\mathbf{Y}\|_{F}$.
 The factors $\mathbf{Y}^{(k)}$ are called \emph{TT cores}, and the ranges of summation indices $r_0,\ldots,r_d \in \mathbb{N}$ are called \emph{TT ranks}.
 Note that without loss of generality we can let $r_0=r_d=1$.
\end{definition}
Plugging in the basis and redistributing the summations we obtain the FTT approximation
$$
\tilde y(\xi) := \sum_{s_0,\ldots,s_d=1}^{r_0,\ldots,r_d} y^{{(1)}}_{s_0,s_1}(\xi_1) y^{{(2)}}_{s_1,s_2}(\xi_2) \cdots y^{{(d)}}_{s_{d-1},s_d}(\xi_d),
$$
where
$$
y^{(k)}_{s_{k-1},s_k}(\xi_k) = \sum_{i=1}^{n_{\xi}} \mathbf{Y}^{{(k)}}_{s_{k-1},i,s_k} \ell_i(\xi_k), \quad k=1,\ldots,d.
$$
Smooth \cite{uschmajew-approx-rate-2013}, weakly correlated \cite{rdgs-tt-gauss-2022} or certainly structured \cite{khor-low-rank-kron-P1-2006} functions
have been shown to induce rapidly converging TT approximations.

Given the TT decomposition, its expectation can be computed by first integrating each TT core, and then multiplying the TT cores one by one.
Let
\begin{equation}\label{eq:ttcoreint}
\mathbf{V}^{(k)}_{s_{k-1},s_k} = \sum_{j=1}^{n_{\xi}} w_j y^{(k)}_{s_{k-1},s_k}(z_j) = \sum_{i,j=1}^{n_{\xi}} w_j \mathbf{L}_{i,j} \mathbf{Y}^{(k)}_{s_{k-1},i,s_k}, \quad \mbox{where} \quad \mathbf{L}_{i,j} = \ell_i(z_j).
\end{equation}
Now we multiply the matrices $\mathbf{V}^{(k)} \in \mathbb{R}^{r_{k-1} \times r_k}$ in order:
\begin{equation}\label{eq:E[tt]}
\mathbb{E}_{\mathbb{P}}[\tilde y] = \left(\left(\left(\mathbf{V}^{(1)} \mathbf{V}^{(2)}\right) \mathbf{V}^{(3)}\right) \cdots \mathbf{V}^{(d)}\right).
\end{equation}
Note that each step in \eqref{eq:E[tt]} is a product of $1 \times r_{k-1}$ vector by $r_{k-1} \times r_k$ matrix.
In turn, the univariate quadrature \eqref{eq:ttcoreint} requires $n_{\xi}^2 r_{k-1}r_k$ floating point operations if the Vandermonde matrix $\mathbf{L}$ is dense, and $\mathcal{O}(n_{\xi} r_{k-1}r_k)$ if it's sparse, for example, if Lagrange polynomials are used.
Introducing $r := \max_k r_k$, we conclude that the expectation of a TT decomposition can be computed with a complexity $\mathcal{O}(dr^2)$ which is linear in the dimension.

To compute a TT approximation, we employ the TT-Cross algorithm \cite{ot-ttcross-2010}.
We start with an empirical risk minimization problem
$$
\min_{\mathbf{Y}^{(1)},\ldots,\mathbf{Y}^{(d)}} \sum_{j=1}^{N} \left(\tilde y(\xi^j) - y(\xi^j)\right)^2,
$$
where $\Xi = \{\xi^j\}$ is a certain set of samples.
To avoid minimization over all $\mathbf{Y}^{(1)},\ldots,\mathbf{Y}^{(d)}$ simultaneously (which is non-convex), we switch to an alternating direction approach:
iterate over $k=1,\ldots,d$, solving in each step
\begin{equation}\label{eq:tt-local}
\min_{\mathbf{Y}^{(k)}} \sum_{j=1}^{N} \left(\tilde y(\xi^j) - y(\xi^j)\right)^2.
\end{equation}
This problem can be solved by linear normal equations.
Indeed, introduce a matrix $\mathbf{Y}_{\neq k}\in\mathbb{R}^{N \times (r_{k-1} n_{\xi} r_k)}$ with elements
$$
(\mathbf{Y}_{\neq k})_{j,t} = \sum_{s_0,\ldots,s_{k-2}}  y^{(1)}_{s_0,s_1}(\xi_1^j) \cdots y^{(k-1)}_{s_{k-2},s_{k-1}}(\xi_{k-1}^j) \ell_i(\xi_k^j) \sum_{s_{k+1},\ldots,s_d}  y^{(k+1)}_{s_k,s_{k+1}}(\xi_{k+1}^j) \cdots y^{(d)}_{s_{d-1},s_{d}}(\xi_{d}^j),
$$
where $t = (s_{k-1}-1)n_{\xi} r_k + (i-1)r_{k} + s_k$,
and a vector $\mathbf{y}^{(k)} \in \mathbb{R}^{r_{k-1} n_{\xi} r_k}$ with elements $\mathbf{y}^{(k)}_t = \mathbf{Y}^{(k)}_{s_{k-1},i,s_k}$.
Now $\tilde y(\Xi) = \mathbf{Y}_{\neq k} \mathbf{y}^{(k)}$, and \eqref{eq:tt-local} is minimized by
\begin{equation}\label{eq:tt-normeq1}
\mathbf{y}^{(k)} = (\mathbf{Y}_{\neq k}^\top \mathbf{Y}_{\neq k})^{-1} (\mathbf{Y}_{\neq k}^\top y(\Xi)).
\end{equation}

To both select ``good'' sample set $\Xi$ and simplify the assembly of $\mathbf{Y}_{\neq k}$, we restrict the set to have the Cartesian form
$$
\Xi = \Xi_{<k} \times Z \times \Xi_{>k},
$$
where $\Xi_{<k} = \{(\xi_1,\ldots,\xi_{k-1})\}$, $\Xi_{>k} = \{(\xi_{k+1},\ldots,\xi_{d})\}$
with \emph{nestedness} conditions
\begin{align*}
(\xi_1,\ldots,\xi_{k-1},\xi_{k}) \in \Xi_{<k+1} & \Rightarrow (\xi_1,\ldots,\xi_{k-1}) \in \Xi_{<k}, \\
(\xi_k,\xi_{k+1},\ldots,\xi_{d}) \in \Xi_{>k-1} & \Rightarrow (\xi_{k+1},\ldots,\xi_{d}) \in \Xi_{>k}.
\end{align*}
This makes
$$
\mathbf{Y}_{\neq k} = \mathbf{Y}_{<k} \otimes \mathbf{L} \otimes \mathbf{Y}_{>k},
$$
where
\begin{align*}
(\mathbf{Y}_{<k})_{j,s} & = \sum_{s_0,\ldots,s_{k-2}} y^{(1)}_{s_0,s_1}(\xi_1^j) \cdots y^{(k-1)}_{s_{k-2},s}(\xi_{k-1}^j), & (\xi_1^j,\ldots,\xi_{k-1}^j) \in \Xi_{<k}, \\
(\mathbf{Y}_{>k})_{j,s} & = \sum_{s_{k+1},\ldots,s_{d}} y^{(k+1)}_{s,s_{k+1}}(\xi_{k+1}^j) \cdots y^{(d)}_{s_{d-1},s_{d}}(\xi_{d}^j), & (\xi_{k+1}^j,\ldots,\xi_{d}^j) \in \Xi_{>k}.
\end{align*}
Moreover, $\mathbf{Y}_{<k+1}$ and $\mathbf{Y}_{>k-1}$ are
submatrices of
\begin{equation}\label{eq:Ylek}
\mathbf{Y}_{\le k} := \begin{bmatrix}\mathbf{Y}_{<k} y^{(k)}(z_1) \\ \vdots \\\mathbf{Y}_{<k} y^{(k)}(z_{n_{\xi}}) \end{bmatrix} \quad \mbox{and} \quad \mathbf{Y}_{\ge k} := \begin{bmatrix} y^{(k)}(z_1) \mathbf{Y}_{>k} & \cdots & y^{(k)}(z_{n_{\xi}}) \mathbf{Y}_{>k} \end{bmatrix},
\end{equation}
respectively.
This allows us to build the sampling sets by selecting $r_{k}$ rows of $\mathbf{Y}_{\le k}$ (resp. columns of $\mathbf{Y}_{\ge k}$) by the \emph{maximum volume principle} \cite{gostz-maxvol-2010},
which needs only $\mathcal{O}(n_{\xi}r^3)$ floating point operations per single matrix $\mathbf{Y}_{\le k}$ or $\mathbf{Y}_{\ge k}$.
The $r_k$ indices of rows of $\mathbf{Y}_{\le k}$ constituting the maximum volume submatrix $\mathbf{Y}_{<k}$ are also indices of the $r_k$ tuples in $\Xi_{<k} \times Z$ constituting the next ``left'' set $\Xi_{<k+1}$. The ``right'' set $\Xi_{>k-1}$ is constructed analogously.
This closes the recursion and allows us to carry out the alternating iteration in either direction, $k=1,\ldots,d$ or $k=d,\ldots,1$.
By this construction, the cardinality of $\Xi_{<k+1}$ is $r_k$, and the cardinality of $\Xi_{>k-1}$ is $r_{k-1}$.
Hence, the cardinality of $\Xi$ is $r_{k-1}n_{\xi} r_k$, and one full iteration of the TT-Cross algorithm needs $\mathcal{O}(dn_{\xi}r^2)$ samples of $y$.

One drawback of the ``naive'' TT-Cross algorithm outlines above is that the TT ranks are fixed.
To adapt them to a desired error tolerance, several modifications have been proposed:
merge $\xi_k,\xi_{k+1}$ into one variable, optimize the corresponding larger TT core, and separate it into two actual TT cores using truncated singular value decomposition (SVD) \cite{so-dmrgi-2011proc} or matrix adaptive cross approximation \cite{ds-parcross-2020};
oversample $\Xi_{<k}$ or $\Xi_{>k}$ with random or error-targeting points \cite{ds-amen-2014};
oversample the selection of submatrices from \eqref{eq:Ylek} by using the \emph{rectangular} maximum volume principle \cite{mo-rectmaxvol-2018}.

However, in this paper we can pursue a somewhat more natural regression approach \cite{dklm-tt-pce-2015}.
We will always need to approximate a vector function, where different components correspond to different degrees of freedom of an ODE or  a PDE solution, or different components of a gradient.
Since the procedure to evaluate $y$ is now taking two arguments ($\xi$ and, say, $m=1,\ldots,M$ indexing extra degrees of freedom), we can replace the normal equations \eqref{eq:tt-normeq1} by
$$
\mathbf{y}^{(k)}(m) = (\mathbf{Y}_{\neq k}^\top \mathbf{Y}_{\neq k})^{-1} (\mathbf{Y}_{\neq k}^\top y(\Xi,m)),
$$
which can be reshaped into a 4-dimensional tensor $\mathbf{\hat Y}^{(k)} \in \mathbb{R}^{r_{k-1} \times n_{\xi} \times r_k \times M}$ with elements $\mathbf{\hat Y}^{(k)}_{s_{k-1},i,s_k,m} = \mathbf{y}^{(k)}_t(m)$.
To compute the usual 3-dimensional TT core, we can use a simple Principal Component Analysis (PCA), which selects $\hat r$ slices $\mathbf{Y}^{(k)}_{s_{k-1},i,1},\ldots,\mathbf{Y}^{(k)}_{s_{k-1},i,\hat r}$ with the minimal $\hat r$ such that
$$
\min_\mathbf{W} \sum_{s_{k-1},i,s_k,m}\left(\sum_{s=1}^{\hat r}\mathbf{Y}^{(k)}_{s_{k-1},i,s} \mathbf{W}_{s,s_{k},m} - \mathbf{\hat Y}^{(k)}_{s_{k-1},i,s_k,m}\right)^2 \le \mathrm{tol}^2 \cdot \|\mathbf{\hat Y}^{(k)}\|_F^2.
$$
Note that this problem is solved easily by the truncated SVD,
where the new TT rank $\hat r$ can be chosen anywhere between $1$ and $\min\{r_{k-1}n_{\xi}, r_kM\}$ to satisfy the error tolerance $\mathrm{tol}$.
After replacing $r_k$ with $\hat r$, the TT-Cross iteration $k=1,\ldots,d$ can proceed as previously.
In the last step ($k=d$), the PCA step is omitted, and we obtain the so-called \emph{block} TT decomposition \cite{dkos-eigb-2014}, which in the functional form reads
$$
\tilde y(\xi,m) = \sum_{s_0,\ldots,s_d} y^{(1)}_{s_0,s_1}(\xi_1) \cdots y^{(d-1)}_{s_{d-2},s_{d-1}}(\xi_{d-1}) \hat y^{(d)}_{s_{d-1},s_d}(\xi_d,m).
$$
The ``backward'' iteration $k=d,\ldots,1$ can be generalized similarly.

\section{Speeding up the gradient by precomputing the state}
\label{app:prac}
using the TT-Cross, followed by taking the expectation of the TT decomposition.\footnote{Note that $\mathbf{G}^{\varepsilon,h}_{u}(\xi)$ is a vector function with $M$ being the number of degrees of freedom in the discretized $u$.}
This can be performed in two ways.
To begin with, we can apply the TT-Cross algorithm to approximate directly $\mathbf{G}^{\varepsilon,h}_{u}(\xi)$.
For each sample $\xi^j \in \Xi$, one needs to solve one forward problem to compute $\mathbf{S}_h(\xi^j) u$,
and one adjoint problem to apply $\mathbf{S}_h(\xi^j)^*$ to the rest of the function.
Recall that the TT-Cross needs $\mathcal{O}(dn_{\xi}r^2)$ samples, hence $\mathcal{O}(dn_{\xi}r^2)$ solutions of the forward, adjoint and sensitivity problems.
However, the maximal TT rank $r$ of the softplus and sigmoid functions typically grows proportionally to $1/\varepsilon$.
When the solution of the forward and adjoint problem is expensive (for example, in the PDE-constrained optimization), this may result in an excessive computational complexity.

Alternatively, we can first compute TT approximations $\mathbf{\tilde y}(\xi) \approx \mathbf{S}_h(\xi)u$
and $\mathbf{\tilde S}_h(\xi)^* \approx \mathbf{S}_h(\xi)^*$,
followed by TT approximations $\mathbf{\tilde g}_{\varepsilon}(\xi) :\approx g_{\varepsilon}(\mathbf{\tilde y}(\xi) + \mathbf{y}_f - \boldsymbol\psi_h)$,
$\mathbf{\tilde g'}_{\varepsilon}(\xi) :\approx g'_{\varepsilon}(\mathbf{\tilde y}(\xi) + \mathbf{y}_f - \boldsymbol\psi_h)$, and finally
$\mathbf{\tilde G}^{\varepsilon,h}_{u}(\xi) \approx \mathbf{\tilde S}_h(\xi)^* \mathrm{diag}(\mathbf{\tilde g'}_{\varepsilon}(\xi)) \mathbf{M} \mathbf{\tilde g}_{\varepsilon}(\xi)$.
The product of TT tensors in $\mathbf{\tilde G}^{\varepsilon,h}_{u}(\xi)$ is also computed using the TT-Cross, but now using the approximate solution $\mathbf{\tilde y}(\xi)$,
which can be interpolated cheaply in the TT-Cross, instead of solving the full PDE.
The bottleneck now is the approximation of the matrix-valued function $\mathbf{S}_h(\xi)^* \in \mathbb{R}^{n_u \times n_y}$.
If both $n_y$ and $n_u$ are large (for example, in a case of a distributed control),
the computation of $\mathbf{S}_h(\xi)^*$ for each sample of $\xi$ requires assembling this large dense matrix, equivalent to the solution of the adjoint problem with $n_u$ right hand sides.
Nevertheless, the tensor approximation of $\mathbf{S}_h(\xi)^*$ converges usually much faster (e.g. exponentially) compared to the approximation of $\mathbf{G}^{\varepsilon,h}_{u}(\xi)$ directly,
hence the TT approximation of $\mathbf{S}_h(\xi)^*$ may need much smaller TT ranks compared to the TT approximation of $\mathbf{G}^{\varepsilon,h}_{u}(\xi)$.
In turn, the TT-Cross applied to $\mathbf{S}_h(\xi)^*$ requires much fewer solutions of the forward problem.
For a moderate $n_u$ this makes it faster to precompute $\mathbf{\tilde y}(\xi)$ and $\mathbf{\tilde S}_h(\xi)^*$.

\end{document}